\documentclass[11pt,letterpaper]{amsart}

\usepackage{amsmath}
\usepackage{amssymb}
\usepackage{amsthm}
\usepackage{mathrsfs}
\usepackage{mathtools}

\usepackage{tikz}
\usepackage{hyperref}
\usepackage{cleveref}
\usepackage{tikz-cd}
\usepackage{mathtools}

\usepackage{placeins}

\usepackage{float}

\usepackage{comment}

\usepackage{enumerate}

\usetikzlibrary{calc}

\theoremstyle{definition}
\newtheorem{thm}{Theorem}[section]
\crefname{thm}{Theorem}{Theorems}
\newtheorem{cor}[thm]{Corollary}
\newtheorem{prop}[thm]{Proposition}
\crefname{prop}{Proposition}{Propositions}
\newtheorem{lem}[thm]{Lemma}
\crefname{lem}{Lemma}{Lemmas}
\newtheorem{clm}[thm]{Claim}
\newtheorem{conj}[thm]{Conjecture}

\crefname{defn}{Definition}{Definitions}

\newtheorem{exmp}[thm]{Example}

\newtheorem{rmk}[thm]{Remark}

\newtheorem*{ack*}{Acknowledgements}
\newtheorem{setup}[thm]{Setup}

\newcommand{\card}{\#}
\newcommand{\cc}{\operatorname{coco}}
\newcommand{\ccc}{\operatorname{cococo}}

\newcommand{\ap}{\operatorname{ap}}
\newcommand{\gap}{\operatorname{gap}}
\newcommand{\sco}{\operatorname{sco}}
\newcommand{\co}{\operatorname{co}}

\newcommand{\pvh}{\textcolor{red}}

\newcommand{\mb}[1]{\mathbb{#1}}
\newcommand{\sub}{\subseteq}
\newcommand{\eps}{\epsilon}
\newcommand{\dD}{\delta}
\newcommand{\aA}{\alpha}
\newcommand{\bB}{\beta}
\newcommand{\gG}{\gamma}
\newcommand{\lL}{\lambda}
\newcommand{\tT}{\theta}
\newcommand{\DD}{\Delta}
\newcommand{\LL}{\Lambda}
\newcommand{\OO}{\Omega}
\newcommand{\sm}{\setminus}
\newcommand{\es}{\emptyset}
\newcommand{\pl}{\partial}

\allowdisplaybreaks

\title{Locality in sumsets.}
\author{Peter van Hintum}
\thanks{New College, University of Oxford, UK. email: peter.vanhintum@new.ox.ac.uk}
\author{Peter Keevash}
\thanks{Mathematical Institute, University of Oxford, UK. Supported by ERC Advanced Grant 883810.}


\begin{document}

\maketitle
\begin{abstract}
Motivated by the Polynomial Freiman-Ruzsa (PFR) Conjecture, we develop a theory of locality in sumsets, with several applications to John-type approximation and stability of sets with small doubling. One highlight shows that if $A \subset \mathbb{Z}$ with $|A+A| \le (1-\epsilon) 2^d |A|$ is non-degenerate then $A$ is covered by $O(2^d)$ translates of a $d$-dimensional generalised arithmetic progression ($d$-GAP) $P$ with $|P| \le O_{d,\epsilon}(|A|)$; thus we obtain one of the polynomial bounds required by PFR, under the non-degeneracy assumption that $A$ is not efficiently covered by $O_{d,\epsilon}(1)$ translates of a $(d-1)$-GAP.

We also prove a stability result showing for any $\epsilon,\alpha>0$ that if $A \sub \mathbb{Z}$ with $|A+A| \le (2-\epsilon)2^d|A|$ is non-degenerate then some $A' \subset A$ with $|A'|>(1-\alpha)|A|$ is efficiently covered by either a $(d+1)$-GAP or $O_{\alpha}(1)$ translates of a $d$-GAP.
This `dimension-free' bound for approximate covering makes for a surprising contrast with exact covering, where the required number of translates not only grows with $d$, but does so exponentially. Another highlight shows that if $A \subset \mathbb{Z}$ is non-degenerate with $|A+A| \le (2^d + \ell)|A|$ and $\ell \le 0.1 \cdot 2^d$ then $A$ is covered by $\ell+1$ translates of a $d$-GAP $P$ with $|P| \le O_d(|A|)$; this is tight, in that $\ell+1$ cannot be replaced by any smaller number. 

The above results also hold for $A \subset \mathbb{R}^d$, replacing GAPs by a suitable common generalisation of GAPs and convex bodies, which we call generalised convex progressions. In this setting the non-degeneracy condition holds automatically, so we obtain essentially optimal bounds with no additional assumption on $A$. 
Here we show that if $A\subset\mathbb{R}^k$ satisfies $|\frac{A+A}{2}|\leq (1+\delta)|A|$ with $\delta\in(0,1)$, then $\exists A'\subset A$ with $|A'|\geq (1-\delta)|A|$ so that $|\co(A')|\leq O_{k,1-\delta}(|A|)$. This is a dimensionally independent sharp stability result for the Brunn-Minkowski inequality for equal sets, which hints towards a possible analogue for the Pr\'ekopa-Leindler inequality.

These results are all deduced from a unifying theory, in which we introduce a new intrinsic structural approximation of any set, which we call the `additive hull', and develop its theory via a refinement of Freiman's theorem with additional separation properties. A further application that will be published separately is a proof of Ruzsa's Discrete Brunn-Minkowski Conjecture.
\end{abstract}

\newpage

\setcounter{tocdepth}{1}
\tableofcontents

\section{Introduction}
A foundational result in Additive Combinatorics is \emph{Freiman's Theorem} \cite{freiman1959addition} that any subset $A$ of integers with bounded doubling is a dense subset of a generalised arithmetic progression $P$ of bounded dimension (see the book of Tao and Vu \cite{tao2006additive} for definitions and background). This gives a satisfactory qualitative description of $A$: it can be approximated by some $P$ belonging to a simple class of sets with bounded doubling. However, this may be very weak quantitatively, as the doubling of $P$ may be much larger than that of $A$, so the quest for a more quantitative version of Freiman's Theorem has been a major driving force in the development of Additive Combinatorics. 
A little thought reveals that one must allow $P$ to come from a broader class of sets than just generalised arithmetic progressions. One natural attempt is to approximate $A$ by sets $X+P$ with $|X|$ bounded, i.e.\ the union of a bounded number of translates of $P$. One might hope to find a such an approximation with polynomial bounds, in the sense that if $|A+A| \le e^{O(d)}|A|$ then we could find such $X,P$ with $\dim(P)=O(d)$, $|X|\leq e^{O(d)}$ and $|X+P| \le e^{O(d)}|A|$. However, an example of Lovett and Regev \cite{lovett2017} shows that this is not always possible. Improving the bounds in Freiman's theorem has been the subject of a rich body of research \cite{ruzsa1994generalized, bilu1999structure, chang2002polynomial, green2006compressions, sanders2008appendix, schoen2011near}. The best known bounds for this type of approximation, due to Sanders \cite{sanders2012bogolyubov}, have $\dim(P)$ and $\log(|X+P|/|A|)$ about $O(d^6)$.

The \emph{Polynomial Freiman-Ruzsa (PFR) Conjecture} (see \cite{PFR}) attempts another type of approximation, for which it is plausible that polynomial bounds may be true (hence its name). Here one attempts to approximate by translates of a \emph{convex progression}, i.e.\ a set $P$ of the form $\phi(C\cap \mathbb{Z}^k)$ for some convex set $C \subset \mathbb{R}^k$ and linear map $\phi:\mathbb{Z}^k\to\mathbb{Z}$. The conjecture states that if $|A+A| \le e^{O(d)}|A|$ then one can find such $P$ with $k = O(d)$ and $|P| \le e^{O(d)}|A|$ such that $A \sub X+P$ for some $X$ with $|X| \le e^{O(d)}$. Below we will describe three perspectives on the PFR Conjecture that provide a thematic overview of our results in this paper; these are (1) John-type approximation, (2) Stability, (3) Locality. Our third theme of locality is the primary focus of this paper, hence its title, in the sense that most of the technical work goes into developing the theory of locality, which is then used to deduce the results discussed within the first two themes of John-type approximation and Stability. Our results hold both for the discrete setting $A \sub \mb{Z}$ considered in PFR and the continuous setting $A \sub \mb{R}^k$. For now we will continue to focus on the discrete setting (which is in some sense the hardest case; we achieve better bounds in the continuous setting). 

Our first perspective interprets PFR as a \emph{John-type approximation}. In general terms, a John-type theorem says that any object in some class is approximated efficiently (i.e.\ up to some constant factor) by some object from some simpler class. Some examples are John's Theorem approximating convex bodies by ellipsoids, Freiman's Theorem approximating sets of small doubling by generalised arithmetic progressions, and a theorem of Tao and Vu \cite{tao2008john} approximating convex progressions by generalised arithmetic progressions. 

\medskip

\emph{Theme 1: John-type approximation}. 
One question that we address in this paper is the existence of John-type approximations $P$ for sets of bounded doubling $A$. For example, if $A \sub \mb{Z}$ is non-degenerate (see \Cref{nondegeneracy}) with $|A+A| \le (2^d + \ell)|A|$ and $\ell \le 0.1 \cdot 2^d$ we show that $A$ is covered by $\ell+1$ translates of a $d$-GAP $P$ with $|P| \le O_d(|A|)$. Here $\ell+1$ translates is optimal (as shown by adding $\ell$ scattered points to a $d$-GAP), so in the sense of John-type approximation we have a precise characterisation of such sets $A$. We also show (see \Cref{coveringcorintegers}) that if $A \sub \mb{Z}$ with $|A+A| \le (1-\eps) 2^d |A|$ is non-degenerate then $A$ is contained in $O(2^d)$ translates of a $d$-dimensional convex progression $P$ with $|P| \le O_{d,\eps}(|A|)$; thus we obtain one of the polynomial bounds required by PFR.

\medskip

Our second perspective sees PFR as a \emph{stability} statement. In Extremal Combinatorics, the general form of such a statement is that if an object in some class has close to the maximum possible size (or more generally is close to maximising some function on the class) then it must be structurally close to some extremal example. The various possible meanings of `structurally close' are nicely expressed by some informal terminology of Tao: we speak of 1\%, 99\% or 100\% stability according to whether we approximate some constant fraction (1\%), all bar some constant fraction (99\%), or everything (100\%). For some examples of stability in Extremal (Hyper)Graph Theory see the survey \cite{keevash2011hypergraph}; for examples of Isoperimetric Stability (closer to our additive setting here) see \cite{maggi2008stability, barber2020isoperimetric}. For sets of small doubling, stability results are only known when the doubling is quite close to the minimum possible, such as the celebrated Freiman $3k-4$ Theorem (see \cite{freiman1959addition}) and various results (described below) for `non-degenerate' $A$ in $\mb{R}^k$ or $\mb{Z}^k$ with $|A+A| \le (2^k + \dD)|A|$ for small $\dD$. In this context, a well-known argument known as Ruzsa's Covering Lemma (see e.g. \cite{ruzsa1999analog}) converts any 1\% stability theorem into a 100\% stability theorem, i.e.\ if we find some structure in a constant fraction of $A$ then we can cover $A$ by a constant number of translates of this structure. However, this argument is quantitatively weak, so we require an alternative approach for optimal bounds. 

\medskip

\emph{Theme 2: Stability}.
We will approach 100\% stability via 99\% stability, i.e.\ we first seek a structural description for almost all of $A$, and then use it to deduce the remaining structure. This approach is well-known in Extremal Combinatorics (the `stability method'), but we are not aware of any previous literature applying it to Freiman's Theorem. Our 99\% stability result (see \Cref{ApproximateStructureThmIntegers}) shows for any $\eps,\aA>0$ that if $A \sub \mb{Z}$ with $|A+A| \le (2-\eps)2^d|A|$ is non-degenerate 
then some $A' \sub A$ with $|A'|>(1-\aA)|A|$ is efficiently covered 
by either a $(d+1)$-GAP or $O_{\eps,\aA}(1)$ translates of a $d$-GAP.
This `dimension-free' bound for approximate covering makes for a surprising contrast with exact covering, where the required number of translates not only grows with $d$, but does so exponentially.

\medskip

Our third perspective on PFR sees it as describing the \emph{locality} of $A$. We think of the locality $|X|$ in the statement as the number of locations for $A$, taking the view that elements of the same convex progression are close additively, even though they need not be close metrically. This perspective is particularly clarifying for the continuous setting of $A \sub \mb{R}^k$. Here the classical Brunn-Minkowski inequality shows that\footnote{ For now we use $| \cdot |$ notation for both (discrete) cardinality and (continuous) measure, but for clarity later we will use $| \cdot |$ for measure and $\card( \cdot )$ for cardinality.} $|A+A| \ge 2^k|A|$, with equality if and only if $A$ is convex up to a null set. There is a substantial literature on $A \sub \mb{R}^k$ with $|A+A| \le (2^k+\dD)|A|$ for small $\dD>0$. For such $A$, Christ \cite{christ2012near} showed that the convex hull $\co(A)$ satisfies $|\co(A)| \le (1+\eps(\dD)) |A|$, where $\eps(\dD) \to 0$ as $\dD \to 0$. Improvements were obtained by Figalli and Jerison \cite{figalli2015quantitative,figalli2021quantitative}, recently culminating in a sharp stability result with optimal parameters by van Hintum, Spink and Tiba \cite{van2021sharp}. Thus for small $\dD$ the locality of $A$ is simply its convex hull, but for larger $\dD$ the picture becomes more complicated.

\medskip

\emph{Theme 3: Locality}.
A natural starting point is to consider $A \sub \mb{R}^k$ with $|A+A| \le (2^k+\dD)|A|$ and ask how large $\dD>0$ may be for us to still efficiently cover $A$ by a convex set. This is clearly impossible for $\dD \ge 1$: consider a convex set and add one faraway point. One consequence of our general theory will be that the threshold is exactly at $1$, i.e.\ if $\dD<1$ then $|\co(A)|<O_k(|A|)$ and thus $|\co(A)\setminus A|\leq O_k(\delta |A|)$. We have an independent geometric proof for this case, which will be published separately in \cite{doublingthreshold}. As $\dD$ increases for a while, we can efficiently cover $A$ by translates of a convex set; moreover, while $\dD < 0.1 \cdot 2^k$ we can do so with $\dD+1$ translates, which is optimal (similarly to our analogous result in $\mb{Z}$ mentioned above; both are subsumed in a more general picture). However, when $\dD$ reaches $2^k$, any fixed number of translates of a convex set will not suffice, as $A$ may be of the form $X+P$ where $P$ is convex and $X$ is an arbitrarily long arithmetic progression. 

\emph{Generalised convex progressions}.
We are thus led to introduce a common generalisation of convex progressions and generalised arithmetic progressions: a \emph{proper convex $(k,d)$-progression} is a set $P=\phi(C\cap (\mb{R}^k\times \mb{Z}^d))$, where $C\subset\mb{R}^{k+d}$ is convex and $\phi:\mb{R}^{k+d}\to\mb{R}^k$ is a linear map injective on $C\cap (\mb{R}^k\times \mb{Z}^d)$. In  \Cref{FewLocationsI} we consider non-degenerate $A \sub \mb{R}^k$ with $|A+A| \le (2^{k+d}+\ell)|A|$ for $d,\ell \ge 0$ and find an efficient covering of $A$ by $X+P$, where $P$ is a convex $(k,d)$-progression and $|X|$ is tightly controlled in terms of $\ell$; in particular, for $\ell \le 0.1 \cdot 2^{k+d}$ we obtain the optimal bound $|X| \le \ell+1$. Setting $d=0$ recovers our results discussed above for $A \sub \mb{R}^k$. We will also see that the results in $\mb{Z}$ follow from those in $\mb{R}$. Thus in our general setting we can still think of $|X|$ as measuring locality, provided that we think of elements of the same generalised convex progression as being close additively, if not metrically.

\medskip

In the next subsection we will introduce notation that will be used to formally develop the above concepts and state our precise results during the remainder of this introduction. These will be organised by subsection according to our main theme of locality, covering the results discussed above, and some further results that we will now summarise.
\begin{itemize}
\item Our main technical contribution introduces a new intrinsic structural approximation of any set, which we call the `additive hull', and develops its theory via a refinement of Freiman's theorem with additional \emph{separation properties} (see \Cref{SeparatedFreimanThm}). 
\item For non-degenerate $A \sub \mb{R}^k$ with $|A+A| \le (2^{k+d}+\dD)|A|$ where $\dD\in(0,1)$ we find a convex $(k,d)$-progression $P$ with $|P \sm A| < O_{k,d}(\dD |A|)$ (see \Cref{2^k+dstability}). Setting $d=0$ recovers the previously mentioned sharp stability result of \cite{van2021sharp} for $A \sub \mb{R}^k$, whereas setting $k=0$ recovers a sharp stability result for non-degenerate $A \sub \mb{Z}^d$ by the same authors \cite{van2020sets}.
\item We obtain a very precise structural description of sets $A \sub \mb{R}$ in the line with doubling $4-\eps$ (see \Cref{4kthm}): either $|\co(A)|=O(|A|)$ or $A \sub \co_t(A)$, a union of $t < 2/\eps$ intervals with $|\co_t(A)| < 2|A|$, with linear stability $|P\sm A| \le O(\dD |A|)$ in terms of $\dD = |A+A|/|A| - (4-2/t)$, where $P$ is a convex $(1,1)$-progression.
\end{itemize}

\subsection{Overview and notation}

The following generalised notion of convex hull will play a crucial role throughout the paper. 
For $A\subset \mathbb{R}^k$, we write $\co_{t}^{\mathbb{R}^k,d}(A) = X+P$,
where $P$ is a proper convex $(k,d)$-progression and $\card X \le t$, 
choosing $X$ and $P$ so that $A \sub X+P$ and $|X+P|$ is minimal; 
we fix an arbitrary choice if $X+P$ is not unique. 
In some cases we will omit $k,d$ if $d=0$ and $t$ if $t=1$, 
e.g.\ $\co(X)$ should be understood as $\co_1^{\mathbb{R}^k,0}(X)$, 
where $k$ is the dimension of the ambient space for $X$, 
so that it coincides with the common notion of the convex hull. 

We also require the closely related notion $\gap^{\mathbb{R}^k,d}_t(A)$, 
defined as a minimum volume set $X+P+Q$ containing $A$ such that
$\card X\leq t$, $P$ is a proper $d$-GAP and $Q$ is a parallelotope. 
This is roughly equivalent to the variant $\sco_{t}^{\mathbb{R}^k,d}(A)$
defined exactly as  $\co_{t}^{\mathbb{R}^k,d}(A)$ but imposing
the symmetry requirement $P=-P$. Indeed, 
$|\co_{t}^{\mathbb{R}^k,d}(A)| \le |\sco^{\mathbb{R}^k,d}_t(A)| \le |\gap^{\mathbb{R}^k,d}_t(A)|$ is clear,
and \Cref{gap=sco} will show 
$|\gap^{\mathbb{R}^k,d}_t(A)| = O_{d,k}(|\sco^{\mathbb{R}^k,d}_t(A)|)$. 
Most results in this paper will be stated using $\gap^{\mathbb{R}^k,d}_t$,
are equivalent to the corresponding statement using $\sco^{\mathbb{R}^k,d}_t$,
and imply the corresponding statement using $\co^{\mathbb{R}^k,d}_t$.
However, for some very precise statements we require $\co^{\mathbb{R}^k,d}_t$. 

To state our results in $\mathbb{Z}$, let  $\co^{\mathbb{Z},d}_t,\sco^{\mathbb{Z},d}_t,$ and $\gap_t^{\mathbb{Z},d}$ be the corresponding functions for subsets of $\mathbb{Z}$, replacing `convex $(k,d)$-progression' by `convex $d$-progression'. To be precise, $\co^{\mathbb{Z},d}_t(A)=X+P$ where $P$ is a convex $d$-progression and $\card X\leq t$, choosing $X$ and $P$ so that $A\subset X+P$ and $\card(X+P)$ minimal. Analogously define $\sco^{\mathbb{Z},d}_t(A)$  and $\gap^{\mathbb{Z},d}_t(A)$ with the additional requirement that $P$ is origin symmetric and a generalised arithmetic progression, respectively. Intuitively, throughout the paper we think of $k$ as `continuous dimension' and $d$ as `discrete dimension'. To stress this connection we write $\co^{k,d}_t$ for $\co^{\mathbb{R}^k,d}_t$ and $\co^{0,d}_t$ for $\co^{\mathbb{Z},d}_t$.

As further illustrations of this notation we can restate Freiman's Theorem and PFR as follows.
\begin{align*}
\text{Freiman's Theorem.} & \qquad \text{If } A \sub \mb{Z} \text{ with } \card(A+A)\leq K \card A, \\
&\qquad  \quad \quad \text{ then }\card\left(\co^{0,O_K(1)}_{O_K(1)}(A)\right) \leq O_{K}(\card A ). \\
\text{\hfill PFR.} &\qquad \text{If } A \sub \mb{Z} \text{ with } \card(A+A)\leq e^{O(d)} \card A, \\
&\qquad  \quad \quad \text{ then }\card\left(\co^{0,O(d)}_{e^{O(d)}}(A)\right) \leq e^{O(d)} \card A .
\end{align*}

Many of our theorems include a non-degeneracy condition which should be interpreted as follows
\begin{align}\label{nondegeneracy}
\card\left(\gap_{O_{d,\xi}(1)}^{0,d-1}(A)\right)\geq \Omega_{d,\xi}(\card A)\iff \forall d,\xi, \exists C,c: \card\left(\gap_{c}^{0,d-1}(A)\right)\geq C\card A,\end{align}
i.e. there is no collection of few ($c$ depending only on $d$, and $\xi$) translates of a $d-1$ dimensional generalized arithmetic progression covering $A$ efficiently (exceeding the size of $A$ by at most a factor $C$ depending only on $d$ and $\xi$).

To this end, we say a set $A\subset\mathbb{Z}$ is \emph{$(d,t)$-non-degenerate } if $\card\left(\gap_{t}^{0,d-1}(A)\right)\geq t\card A$, i.e. if $A$ is not covered by at most $t$ translates of a $d-1$-dimensional GAP of  combined size $t|A|$.

We will state our results in the following subsections according to the theme of locality, i.e.\ with respect to the bounds on the parameter $t$ in $\co_{t}^{k,d}(A)$. A rough summary of their contents is as follows:
\begin{itemize}
\item A big part of the set is in one place.
\item The entire set is in few places.
\item Almost all of the set is in a constant number of places.
\item Sets in the line with doubling less than 4 are almost convex.
\item A sharp doubling condition for almost convexity.
\item The additive hull.
\end{itemize}

The theory of sumsets has been developed in several groups. 
Particular attention has been given to the continuous setting of $\mathbb{R}^k$ (e.g. \cite{figalli2009refined,figalli2010mass,christ2012near,figalli2015quantitative,figalli2017quantitative,figalli2021quantitative,planarBM,van2021sharp,figalli2023sharp}) and to the discrete setting of $\mathbb{Z}$ (e.g. \cite{freiman1959addition,ruzsa1994generalized,bilu1999structure,chang2002polynomial,green2006compressions,sanders2008appendix,schoen2011near}). In the context of this paper, the setting does not make much difference to our results, so although we state some results below  for $\mathbb{Z}$, for the proofs we will generally prefer to work in $\mathbb{R}^k$. This is justified as (a) the proofs are the same modulo the theory of the additive hull, and (b) the results in $\mathbb{Z}$ follow from the results in $\mathbb{R}$ via the following proposition proved in \Cref{sec:merge}.

\begin{prop}\label{zrkequivalence}
For any $A\subset \mathbb{Z}$ and $d,t\in\mb{N}$ there is $\epsilon=\epsilon(A,d,t)>0$ 
so that $B:=A+[-\epsilon,\epsilon]\subset\mathbb{R}$ has
$\left|\gap^{1,d}_t(B)\right|=\Theta_{d,t}\left(\epsilon \card \left(\gap^{0,d}_t(A)\right)\right).$
\end{prop}

\subsection{A big part of the set is in one place}

We start with the continuous setting of $\mathbb{R}^k$, where we obtain a clean unified formulation of the 1\% stability phenomenon to be discussed in this subsection (a big part in one place) and as $\delta \to 0$ the 99\% stability phenomenon (for which we will describe sharper results below).  Combining 1\% stability with Ruzsa covering one obtains 100\% stability (for which we will also describe sharper results below). Previous stability results for the Brunn-Minkowski inequality (see \Cref{2dstabilitysection}) only applied for much smaller $\delta$; in particular, we are not aware of any previous results of this kind where $\delta$ does not decrease with the dimension.

\begin{thm}\label{linearroughstructure}
Let $A\subset\mathbb{R}^k$ with $\left|\frac{A+A}{2}\right|\leq (1+\delta)|A|$, 
where $\delta \in (0,1)$. Then there exists a subset $A'\subset A$ with $|A'|\geq (1-\min\{\delta,\delta^2(1+O(\delta))\}) |A|$
and $|\co(A')|\leq O_{1-\delta,k}(|A|)$.
\end{thm}

Since the size of $A'$ is also independent of the dimension, this result is closely related to the stability question of the Pr\'ekopa-Leindler inequality for equal functions, which can be seen as a dimensionally independent version of Brunn-Minkowski. We expand on this connection and formulate a conjectured extension of \Cref{linearroughstructure} in \Cref{prekopasec}.

Now we consider the integer setting, where the Freiman-Bilu theorem \cite{bilu1999structure} shows that for sets with small doubling a large part of the set is contained in a small generalized arithmetic progression of low dimension. A quantitative version of Green and Tao \cite{green2006compressions} shows that for $A\subset \mathbb{Z}$ with $\card(A+A)\leq 2^{d}(2-\epsilon)\card A$ there exists some $A'\subset A$ with $\card\left(\gap^{0,d}(A')\right)\leq \card A$ and
$$\card A'\geq \exp\left(-O(8^d d^3)\right)\epsilon^{O(2^d)}\card A.$$
With the following result we establish a bound on $\card A'$ that is optimal up to lower order terms, at the cost of a non-degeneracy assumption and relaxing the bound on $\card\gap^{0,d}(A')$ in the spirit of our John-type theme. We emphasise that our bound on $\card A'$ does not depend on the dimension.

\begin{thm}\label{linearandquadraticintegers}
$\forall d\in\mathbb{N},\delta\in(0,1), \beta>0,\exists t=t_{d,\delta,\beta}$ so that the following holds.\\ Let $A\subset\mathbb{Z}$ finite, let  $d:=\left\lfloor\log_2\left(\frac{\card(A+A)}{\card A}\right)\right\rfloor$, and $\delta:= \frac{\card(A+A)}{2^d\card A}-1$. 
If $A$ is $(d,t)$-non-degenerate,
then there exists a $d$-GAP $P\subset\mathbb{Z}$ with $\card P\leq t\card A$, so that 
$$\frac{\card\left(A\setminus P\right)}{\card A}\leq \min\left\{(1+\beta)\delta, \delta^2+60\delta^3\right\}.$$
\end{thm}

The above results are sharp up to lower order terms 
both as $\delta\to 0$ for large $k$ and as $\delta\to 1$. We also find the sharp result as $\delta\to 0$ for small $k$ (see \Cref{sharplinearquadraticremark}).
For $\delta\to 0$, consider the union of two homothetic convex sets
of volumes $\dD^2$ and $1-\dD^2$ (see \Cref{ex:2});
for $\delta\to 1$ consider an arithmetic progression 
of $\frac{1}{1-\delta}$ equal convex sets (see  \Cref{ex:AP}).
The second example suggests a $(k+1)$-dimensional convex structure,
which we will indeed establish in \Cref{approximatesection}.

\subsection{The entire set is in few places}

Now we consider the 100\% stability problem:
what is the maximum `locality' for given doubling?
Here our fundamental example (see \Cref{ex:scatter}) 
consists of a generalised arithmetic progression (GAP)
together with some scattered points.
Consider $A = P \cup S \sub \mb{Z}$, where $P$ 
is a proper $d$-dimensional GAP 
and $S$ is `scattered' with $|S|=\ell$.
Then $\card(A+A) \le  (2^{d}+\ell)\card A$
and $A$ has locality $\ell + 1$.

The following result shows for non-degenerate $A$ 
that this example is exactly sharp for a large range of $\ell$
and asymptotically sharp when $2^d-\ell \gg 2^{d/2}$.
We remark that even the much weaker bound of $O(2^d)$ for the locality
is already sufficient to cover $A$ by $X+P$ with doubling $O(2^{2d})$,
i.e.\ our John-type approximation only loses a square in the doubling,
whereas the PFR setting allows any polynomial loss.

\begin{thm}\label{FewLocationsIntegers}
$\forall d\in\mathbb{N},\ell\in[0,2^d),\exists t=t_{d,\ell}$ so that the following holds.\\
Let $A\subset\mathbb{Z}$, let $d:=\left\lfloor\log_2\left(\frac{\card(A+A)}{\card A}\right)\right\rfloor$ and $\ell:=\frac{\card(A+A)}{\card A}-2^d$. If 
$A$ is $(d,t)$-non-degenerate, then 
$A$ is contained in $\ell'$ translates of a $d$-dimensional GAP of size $t\card A$,
where
$$\ell'\leq\begin{cases}
\ell+1 & \text{ for } \ell \in \mb{N} \text{ if }\ell\leq 0.1 \cdot 2^{d},\\
&\text{ or if } \ell\leq 0.315\cdot 2^{d} \text{ and } d\geq 13,\\
\ell \left(1+O\left( \sqrt[3]{\frac{2^{d}}{(2^{d}-\ell)^2}}\right)\right) &\text{ if } 0.1\cdot 2^{d}\leq \ell\leq \left(1-\frac{1}{\sqrt{2^{d}}}\right)2^{d},\\
(1+o(1))\frac{d+1}{2\epsilon} &\text{ where $\epsilon=\frac{2^{d}-\ell}{2^{d}}$ and $o(1)\to 0$ as $\epsilon\to 0$.}
\end{cases}$$
\end{thm}

The final bound in \Cref{FewLocationsIntegers} gives an asymptotically sharp result for the limiting case $\epsilon=\frac{2^{d}-\ell}{2^{d}} \to 0$, i.e.\ as the doubling approaches $2^{d+1}$. Here the above fundamental example breaks down and a new example takes over, which can be thought of as a cone over a GAP (see \Cref{ex:cone}); intuitively, this describes the `most $d$-dimensional'  $(d+1)$-dimensional construction and it shows that even the constant $\frac{d+1}{2}$ is optimal.

The above results are very sharp for non-degenerate sets, but to make further progress towards PFR we need to weaken the non-degeneracy condition. Our next result takes a step in this direction, but its applicability is limited by the double-exponential dependence on $d'$.

\begin{thm}\label{coveringcorintegers}
$\exists c$ and $\forall d\in\mathbb{N}, \exists t=t_d$, so that the following holds.
Let $A\subset\mathbb{Z}$, let $d:=\left\lfloor\log_2\left(\frac{\card(A+A)}{\card A}\right)\right\rfloor+1$ and $d'<d$.
If $A$ is $(d,t)$-non-degenerate then $A$ is covered by at most $d\exp\exp(cd')$ translates of a $d$-dimensional GAP of size $t\card A$.
\end{thm}

\subsection{Almost all of the set is in a constant number of places}\label{approximatesection}

Now we consider 99\% stability. Here we will find that the number of locations can be bounded by a constant independently of the doubling. This is a remarkable contrast with the 100\% stability problem considered above, for which we needed an exponential number of locations to cover the set unless it has close to the minimum possible doubling. Our fundamental example had many scattered points but essentially all of the mass of the set in one location, which hints that one should be able to do much better if one can discard a small part of the set. Furthermore, the second example that takes over as the doubling approaches $2^{d+1}$ is highly structured so that $\co^{0,d+1}(A)$ is small. The following result shows that any non-degenerate set is approximately described by one of these two configurations: it is concentrated in a single $(d+1)$-progression or few $d$-progressions,
where `few' depends only on the approximation accuracy, not on $d$. 

\begin{thm}\label{ApproximateStructureThmIntegers}
$\forall\epsilon>0,d\in\mathbb{N},\exists t=t_{d,\epsilon}$ and $\forall \alpha>0,\exists \ell=\ell_{\alpha}$ so that the following holds.\\ Let $A\subset\mathbb{Z}$, let $d:=\left\lfloor\log_2\left(\frac{\card(A+A)}{\card A}\right)\right\rfloor$ and $\epsilon:=2-\frac{\card(A+A)}{2^{d}\card A}$.
If $A$ is $(d,t)$-non-degenerate, then there is a $A'\subset A$ with $\card A'\geq (1-\alpha)\card A$ so that $A'$ is covered by:
\begin{itemize}
    \item one $(d+1)$-dimensional GAP of size $t\card A'$, or 
    \item $\ell$ translates of a $d$-dimensional GAP of size $t\card{A}'$.
\end{itemize}
\end{thm}

One consequence is that we can always find arithmetic structure
in some absolute constant fraction of our set. This gives an improvement over \Cref{linearandquadraticintegers} for small $\epsilon$.

\begin{cor}\label{FreimanCor}
$\forall\epsilon>0,d\in\mathbb{N},\exists t=t_{d,\epsilon}$ so that the following holds.\\
Let $A\subset\mathbb{Z}$, let $d:=\left\lfloor\log_2\left(\frac{\card(A+A)}{\card A}\right)\right\rfloor$ and $\epsilon:=2-\frac{\card(A+A)}{2^{d}\card A}$.
If $A$ is $(d,t)$-non-degenerate,
then there is a $(d+1)$-dimensional GAP $P$ with $\card P\leq t\card A$ so that $\card(A\cap P)\geq \frac{1}{30000}\card A$.
\end{cor}

\subsection{Sets in the line with doubling less than 4 are almost convex}

Arguably the sharpest stability result for sets with small doubling is Freiman's $3\card A-4$ theorem, which states that for $A\subset\mathbb{Z}$, we have $$\card(A+A)\geq 2\card A-1+\min\left\{\card\left(\co^{0,1}(A)\setminus A\right),\card A-3\right\},$$
and for continuous sets $A\subset\mathbb{R}$, we have \cite{figalli2015quantitative} 
$$|A+A|\geq 2|A|+\min\{|\co(A)\setminus A|,|A|\},$$
cf. \Cref{basicadditioninequalities}.
These results are optimal in the sense that none of the constants can be improved. 
There is a rich literature trying to find results beyond doubling three. 
Jin \cite{jin2007freiman} pushed just beyond $3$, by showing that there are $\Delta,N>0$ 
so that if $A \subset \mathbb{Z}$ with $\card A>N$ and $\card(A+A)\leq (3+\delta)\card A-3$
with $\delta \leq \Delta$ then either $\card\left(\co^{0,1}(A)\right)\leq 2(1+\delta)\card A -1$
or $A$ is contained in two arithmetic progressions $I$ and $J$
with the same step size and $\card I+\card J\leq (1+\delta)\card A$. 
Eberhard, Green, and Manners \cite{eberhard2014sets} proved various Freiman-Bilu style statements for sets with doubling less than four, of which the most relevant to our context is \cite[Theorem 6.4]{eberhard2014sets}, showing that 
if $\card(A-A)\leq (4-\epsilon)\card A$ then there is $A'\subset A$ 
with $\card A'=\Omega_{\epsilon}(\card A)$ so that 
$\card\left(\co^{0,1}(A')\right)=(2-\Omega(\epsilon))\card A'$. \footnote{Though not directly equivalent, the quantity $\card(A-A)/\card A$ is closely related to the doubling $\card(A+A)/\card A$.}

We prove the following sharp characterization of $A\subset\mathbb{R}$ with $|A+A|<4|A|$,
which obtains much stronger structure than that in the above results for the continuous context.
To interpret its statement, suppose for example that $|A+A| < 3.3334 |A|$.
The theorem implies that $A$ is covered by $2$ or $3$ intervals with total size $<2|A|$ (as $\delta>0$).
Furthermore, if $3$ intervals are required then $A$ is closely approximated (with error $<0.02$)
 $3$ intervals whose starting points and lengths are in arithmetic progression.
In general, we let $\ap_t(A)$ be a minimum size set containing $A$ that is  
an arithmetic progression of $t$ intervals whose lengths are in arithmetic progression, e.g. $[0,1]\cup [10,12]\cup [20,23]$.
Then the theorem gives a bound on $t$ in terms of the doubling that is sharp
and has linear stability to this extremal example.
  
\begin{thm}\label{4kthm}
There is an absolute constant $C>0$ such that the following holds.
Suppose $A\subset\mathbb{R}$ with $|A+A|<4|A|$ and $|\co(A)|\geq C|A|$. 
Let $t$ be minimal so that $|\co_t(A)|<2|A|$,
and let $\delta := \frac{|A+A|}{|A|} - (4-2/t)$.
Then $|\co^{1,1}(A) \sm A| \le |A|$.
Moreover, if $\dD \le (2t)^{-2}$ then $|\co^{1,1}(A) \sm A| \le 150\delta|A|$ and $|\ap_t(A) \sm A| \le 500t\dD |A|$.
\end{thm}
The factor $t$ in the bound on $\text{ap}_t$ is necessary as shown by \Cref{ex:apt}.

\subsection{A sharp doubling condition for almost convexity}
\label{2dstabilitysection}

For $A\subset \mathbb{R}^k$ we have $|A+A|\geq |2A| = 2^k|A|$, with equality for convex sets (a trivial instance of the Brunn-Minkowski inequality).  Christ \cite{christ2012near,christ2012planar} obtained a stability result for this inequality, i.e.\ if $|A+A|\leq (2^k+\delta)|A|$ then $|\co(A)\setminus A|\leq \epsilon|A|$, where $\eps \to 0$ as $\delta \to 0$. Quantitative bounds were given by Figalli and Jerison in \cite{figalli2015quantitative,figalli2021quantitative} and van Hintum, Spink and Tiba \cite{van2021sharp}, showing that for some $\Delta_k>0$ if $|A+A|\leq (2^k+\delta)|A|$ and $\delta< \Delta_k$ then $|\co(A)\setminus A|\leq O_k(\delta)|A|$. 

This raises the question of determining the threshold for $\Delta_k$, i.e.\ the bound on the doubling that guarantees such an approximation of $A$ by its convex hull. A consequence of our more general result below is that the answer is $1$, which is clearly optimal, as shown by considering a convex body together with a faraway point. This particular case has a geometric proof extending to distinct summands that will be published separately in \cite{doublingthreshold}.

\begin{cor}\label{2^kstability}
If $A \subset \mathbb{R}^k$ satisfies $|A+A|=(2^{k}+\delta)|A|$ with $\delta< 1$ then
$\left|\co(A)\setminus A\right|\leq O_{k}(\delta) |A|.$
\end{cor}

In the discrete setting $A\subset\mathbb{Z}^k$, van Hintum, Spink, and Tiba \cite{van2020sets} obtained an analogue of their result in $\mathbb{R}^k$, i.e.\ a linear stability result for sets
$A\subset\mathbb{Z}^k$ with $\card(A+A)\leq (2^k+\delta)\card A$ 
that are non-degenerate (not covered by $O_{k,\delta}(1)$ hyperplanes),
again assuming $\delta<\Delta_k$ for some small $\Delta_k>0$ (see \Cref{zkstab}).
Our next result extends this in two ways increasing applicability. We show the result holds for sets in $\mathbb{Z}$ and we extend it to the optimal range $\delta<1$.

\begin{thm}\label{2^0+dstability}
For any $d\in\mathbb{N}$, $\gamma, \epsilon>0$, $\delta \in [0,1-\eps)$,
if $A \subset \mathbb{Z}$ with $\card(A+A) \le (2^{d}+\delta)\card A $ and 
$\card\left(\gap^{0,d-1}_{O_{d,\gamma,\epsilon}(1)}(A)\right)= \OO_{d}(\card A )$
then $\card(\co^{0,d}(A)\sm A)/\card(A) \le O_d(\gG+\dD)$.
\end{thm}

\section{The additive hull and Freiman separation}

In this section we give an overview of our new techniques 
and how we will apply them to prove our main theorems.
We start by describing a structure that we call the `additive hull'.
This terminology is intuitive rather than precisely defined,
as in fact we will describe a parameterised family of hulls,
and each hull is best understood together with its quotient,
which is a weight function on $\mb{Z}$ that in some sense 
captures the doubling properties of the original set.

The motivation for the definition that follows is seeking an
{\em intrinsic} notion of hull of a set that captures its additive properties.
This was already achieved to some extent by our refinement
of $\co$ to $\co^{k,d}_t$, but here the parameters are not intrinsic 
to the set and the doubling properties are not captured even qualitatively:
if $\frac{|co^{k,d}_t(A)|}{|A|}\to \infty$
we cannot deduce that $\frac{|A+A|}{|A|}\to\infty$. 
One can think of our alternative structure as starting with any cover of $A$
by convex bodies then repeatedly merging pairs that are `close'
(metrically or additively). For example, in the $s=2$ case
(the one needed to study $A+A$) when the merging process is complete,
if our cover $\{X_i\}$ contains sets $(X_{i^j}: j \in [4])$ with 
$(X_{i^1} + X_{i^2}) \cap (X_{i^3} + X_{i^4}) \ne \emptyset$
then $i^1+i^2=i^3+i^4$. We also remark that our definition
has some resemblance to the notion of rectifiability
(a set is rectifiable if it is Freiman isomorphic to a subset of $\mb{Z}$).

We now give the definitions. Let $A\subset\mathbb{R}^k$ and $s\in\mathbb{N}$.
The \emph{convex-continuity hull} of $A$ with parameter $s$,
denoted $\cc_s(A)$, is a smallest volume union $\bigcup_i X_i$ 
of disjoint convex sets $X_i$ so that $A\subset \bigcup_i X_i$ 
and there exists an Freiman $s$-homomorphism $f_A:\bigcup_i X_i\to \mathbb{Z}$ 
so that $f_A(x)=f_A(y)$ iff $x,y\in X_i$ for the same $i$. 
The \emph{co-convex-continuity quotient} of $A$ with parameter $s$
is $\ccc_s(A):=(\text{im}(f_A),\mu_A)$, where $\mu_A$ is the measure on $\text{im}(f_A)$
defined by $\mu_A(i) = |A_i|$ with $A_i := A \cap f_A^{-1}(i)$ (so $\mu(A)(\text{im}(f_A))=|A|$).

The following result shows that the definition of $\cc_s(A)$ satisfies
our intuitive test described above for an intrinsic additive hull.

\begin{thm}\label{ccdoubling}
For all $s,k\in\mathbb{N}$ and $M>0$ there exists $N>0$ so that 
if $A\subset\mathbb{R}^k$ has $|\cc_s(A)|>N|A|$ then $|A+A|>M|A|$.
\end{thm}

We think of $\ccc_s(A)$ as factoring out the continuous convex structure of $A$.
One should note that this passage from continuous to discrete structure captures phenomena that are fundamentally different to those seen in the well-known technique of considering limits of large finite structures (e.g.~via nonstandard analysis, as popularised by Tao).

A lower bound on the doubling of $A$ is captured by the max convolution
$$(\mu_A\star_k\mu_A)(i) := \max_{x+y=i}  \left(|\mu_A(x)|^{1/k}+|\mu_A(y)|^{1/k}\right)^k.$$
Indeed, $A_x + A_y$ and $A_{x'}+A_{y'}$ are disjoint unless $x+y=x'+y'$,
so by Brunn-Minkowski 
$$|A+A| \ge (\mu_A\star_k\mu_A)(\mathbb{Z}).$$
Many of our proofs can be thought of as applying this inequality
and then analysing $\mu_A\star_k\mu_A$, which explains why
the results are not sensitive to the ambient group.
However, for concreteness we will present our proofs in the setting $\mb{R}^k$.

The main technical engine that powers the above structures and the proofs in this paper
is a separated version of Freiman's Theorem. The statement requires some further definitions.
Given an abelian group $G$ and a subsets $A,P\subset G$ with $0\in P=-P$, 
we say $A$ is \emph{$P$-separated} if  $a-a'\not\in P$ for all distinct $a, a'\in A$. 
We say that a convex progression $P=\phi(C\cap\mathbb{Z}^d)$ is \emph{$n$-full} 
if it contains a fibre of length at least $n$ in every direction, 
i.e.\ for each $i\in[d]$ there is $x\in\mathbb{Z}^d$ 
with $\card(C\cap(x+\mathbb{Z}e_i))\geq n$. 

Similarly to Freiman's Theorem, which covers a set of bounded doubling by a constant
number of translates of a correct size GAP of the correct dimension,
the following result covers $A\subset\mathbb{R}^k$ of bounded doubling 
by a constant number of translates of a correct size $P+Q$,
where $P$ is a GAP of the correct dimension and $Q$ is a parallellotope.
Thus far this statement is not hard to deduce from Freiman's Theorem via discretization 
(although a non-trivial argument is required to maintain the correct dimension).
However, as mentioned above, the main power in the result lies in its separation properties,
namely that the iterated sumset of $P$ is suitably separated for $Q$
and the iterated sumset of $X$ is suitably separated for $P+Q$.

\begin{thm}\label{SeparatedFreimanThm}
Let $k,d,\ell_i,n_i\in\mathbb{N}$ for $1\leq i\leq d$, $\epsilon>0$, and $A\subset\mathbb{R}^k$. 
If $|A+A|\leq (2^{k+d+1}-\epsilon)|A|$ then $A\subset X+P+Q$, where
\begin{itemize}
    \item $P=-P$ is an $\ell_{d'}$-proper $n_{d'}$-full $d'$-GAP with $d'\leq d$,
    \item $Q=-Q$ is a parallelotope,
    \item $\card X=O_{k,d,\epsilon,n_{d'+1},\dots,n_{d}}(1)$
    \item $\card P \cdot  |Q|=O_{k,d,\epsilon,\ell_{d'},\dots, \ell_{d},n_{d'+1},\dots,n_{d}}(|A|)$,
    \item $\ell_{d'}\cdot X$ is $\ell_{d'}\ell\cdot \frac{(P+Q)}{\ell}$-separated for all $1\leq\ell\leq\ell_{d'} $, and
    \item $\ell_{d'}\cdot P$ is $\ell_{d'}Q$-separated.
\end{itemize}
\end{thm}

Note that $d'=0$ is a possible outcome in  \Cref{SeparatedFreimanThm}, 
meaning that $P$ consists of a single point,
which is vacuously very proper and very full.
To illustrate the power of \Cref{SeparatedFreimanThm} 
we will quickly deduce \Cref{ccdoubling}. A further demonstration of the applicability of \Cref{SeparatedFreimanThm} is the proof of Ruzsa's Discrete Brunn-Minkowski Conjecture (\Cref{ruzsaconj}) which will be published separately \cite{Ruzsaconjecture}.

\begin{proof}[Proof of \Cref{ccdoubling}]
We'll prove the contrapositive result. Assume $|A+A|\leq 2^{k+d}|A|$ for some $d$.
We apply \Cref{SeparatedFreimanThm} with $k=k$, $d=d$, $\ell_i=2s$, $\epsilon=2^{k+d}$ and any $n_i$
to find a cover $\{x+P+Q: x \in X\}$ of volume $O_{k,d,s}(|A|)$. We will show that this cover is
among those considered in the definition of $\cc_s(A)$, so that
 $|\cc_s(A)|\leq \card X \cdot \card P\cdot |Q| = O_{k,d,s}(|A|)$, as required.
 
To this end, consider any injective Freiman $s$-homomorphism $f'_A:X+P\to\mathbb{Z}$.
To see that $f'_A$ exists, note that $X+P$ generates a finitely generated free abelian group, i.e. a lattice,
and it is easy to construct such a map for any bounded subset of a lattice.
Now define $f_A:X+P+Q\to\mathbb{Z}$ by $f_A(x+p+q) = f'_A(x+p)$;
this is well-defined as $P$ is $2Q$-separated, and clearly 
$f_A(y)=f_A(y')$ iff $y,y'$ are in the same translate $x+P+Q$.

It remains to show that $f_A$ is a Freiman $s$-homomorphism.
Consider $x_i,x'_i\in X, p_i,p_i'\in P, q_i,q_i'\in Q$ for $i \in [s]$ 
with $\sum_{i=1}^s x_i+p_i+q_i=\sum_{i=1}^s x'_i+p'_i+q'_i$. 
We have $\sum_{i=1}^s x_i - \sum_{i=1}^s x'_i \in 2s\cdot (P+Q)$,
so separation of $s \cdot X$ implies  $\sum_{i=1}^s x_i= \sum_{i=1}^s x'_i$. 
This further implies $\sum_{i=1}^s p_i+q_i=\sum_{i=1}^s p'_i+q'_i$, 
so that $\sum_{i=1}^s p_i - \sum_{i=1}^s p'_i \in 2sQ$.
Now separation of $s \cdot P$ implies $\sum_{i=1}^s p_i=\sum_{i=1}^s p'_i$. 
We find $\sum_{i=1}^s x_i+p_i=\sum_{i=1}^s x'_i+p'_i$, 
so as $f'$ is a Freiman $s$-homomorphism we have
$$\sum_{i=1}^s f(x_i+p_i+q_i)=\sum_{i=1}^s f'(x_i+p_i)=\sum_{i=1}^s f'(x'_i+p'_i)=\sum_{i=1}^s f(x'_i+p'_i+q'_i).$$
Hence $f$ is a Freiman $s$-homomorphism, as required.
\end{proof}

The preceding argument did not use the fullness of $P$,
but this will be important when analysing the doubling of $A$.
Indeed, writing $A_x  = A \cap (x+P+Q)$ for $x \in X$,
as in our discussion of  $\cc_s(A)$ we can lower bound $|A+A|$
using bounds for each $|A_x + A_y|$. These satisfy an approximate 
Brunn-Minkowski inequality when $P$ is $n$-full for large $n$,
with an error term $O(1/n)$. For this to be effective, it is crucial that
in the bounds for $\card X$ and $\card P \cdot  |Q|$ in \Cref{SeparatedFreimanThm}
the implicit constants do not depend on $n_{d'}$,
so we can make the total of all error terms $o(|A|)$.
This indicates how we will deduce the following consequence 
of \Cref{SeparatedFreimanThm} that we will apply throughout the paper.

\begin{prop}\label{FreimanTool}
Let $k,d\in\mathbb{N}$, $\eta,\epsilon>0$, and $A\subset\mathbb{R}^k$. 
If $|A+A|\leq (2^{k+d+1}-\epsilon)|A|$ then $A\subset X+P+Q$, where $P$ is a $d'$-GAP with  $d'\leq d$
and $Q$ is a parallellotope, we have $\card X + \card P \cdot  |Q|/|A| = O_{k,d,\epsilon,\eta}(1)$
and, writing $A_x:=A\cap (x+P+Q)$ for $x\in X$, for all $x,y,z,w \in X$ we have
\begin{itemize}
        \item  $|A_x+A_y|\geq \left(|A_x|^{1/(k+d')}+|A_y|^{1/(k+d')}\right)^{k+d'}-\eta (\card X) ^{-2}|A|$, and
        \item if $(A_x+A_y)\cap (A_z+A_w)\neq \emptyset$ then $x+y=z+w$.
\end{itemize}
\end{prop}

\subsection{Organisation of the paper}

After gathering some preliminary tools in the next section, we will establish 
the continuous form of Freiman's Theorem in \Cref{sec:cts}.
The technical heart of the paper, in which 
we prove \Cref{SeparatedFreimanThm} and deduce \Cref{FreimanTool},
will be presented in \Cref{sec:merge} and \Cref{sec:snap}.
The former section contains some geometrical merging arguments
that achieve the separation for $P$ in \Cref{SeparatedFreimanThm},
and also will be used to establish two equivalences mentioned above
($\mb{Z}$ versus $\mb{R}$ and $\gap$ versus $\sco$);
the latter contains an algebraic argument for snapping points
into their correct places to achieve separation for $X$.
In the remainder of the paper we will apply these tools to deduce 
the theorems stated in the introduction. The main approximate structure
results are proved in \Cref{sec:approx}. We apply these
in  \Cref{sec:exactI} to establish the exact results in the basic regime
and  the doubling threshold for convexity.
The exact results in the cone regime 
are proved in \Cref{sec:exactII}.
In \Cref{sec:3k-4} we give a self-contained treatment
of our results on $A \subset \mathbb{R}$ with $|A+A|<4|A|$.
We conclude in \Cref{sec:q} with some questions for future research.

\section{Preliminaries} \label{sec:prelim}

Here we gather various definitions and give further details
of some examples briefly mentioned in the introduction.
We also collect various background results,
mostly quoted or adapted from the existing literature.
We will omit some standard definitions that may be
found in the book of Tao and Vu  \cite{tao2006additive}.
Our usual settings will either be
that of finite subsets $A \sub \mathbb{Z}$,
where we write $\card A$ for the cardinality of $A$,
or subsets $A \sub \mathbb{R}^k$ of finite measure,
where we write $|A|$ for the measure of $A$.
The latter setting does not pose any measure theoretic difficulties,
as a standard approximation argument (see Lemma \ref{cubesreduction})
reduces to the case of a finite union of cubes.

\subsection{Sumsets and covering}

Let $G$ be an abelian group, $S \subset G$ and $k \in \mathbb{N}$.
We write $kS = \{kx: x \in S\}$ and $k \cdot S$ 
for the $k$-fold iterated sumset $\left\{\sum_{i=1}^k x_i: x_i\in S\right\}$.
Our main interest will be in $2 \cdot S = S+S$.
We also write $\tfrac{S}{k}=\left\{x\in G: kx \in S\right\}$. 

The following quantitative version of Freiman's Theorem is due to 
Green and Tao \cite[Theorem 1.7 and Remark 1.10]{green2006compressions}.

\begin{thm}  \label{greentaooriginal}
Let $A$ be a finite subset of a torsion-free abelian group 
with $\card(A+A)=2^d(2-\epsilon)\card A $ and $\card A \geq2$.
Then $A \sub X + P$ for some $X$ with $|X| \le \exp(C8^d d^3)/\epsilon^{C2^d}$ 
and $\exp(C8^d)$-proper $d$-GAP $P$ with $\card P \leq \card A $.
\end{thm}

We also use the following continuous instance of the Ruzsa Covering Lemma \cite{ruzsa1999analog}, the proof of which is identical to the discrete proof.
\begin{lem} \label{ruzsa}
Suppose $A,B\subset \mathbb{R}^k$ with $|A+B|\leq \lambda |B|$.
Then $A\subset X+B-B$ for some $X\subset \mathbb{R}^k$ with $\card X\leq \lambda$.
\end{lem}

\subsection{Progressions}

For $a = (a_1,\dots,a_d) \in G^d$ and $\ell = (\ell_1,\dots,\ell_d) \in \mb{N}^d$
we write $P(a; \ell)$ for the $d$-GAP 
$\left\{\sum_{i=1}^d \lL_i a_i: \lL_i=1,\dots, \ell_i\right\}$.
This definition applies within any abelian group $G$;
in particular within $\mb{Z}$ or $\mb{R}^k$.

A \emph{convex $d$-progression} is a set $\phi(C\cap \mathbb{Z}^d)$
where $C \subset \mathbb{R}^k$ is convex 
and $\phi:\mathbb{Z}^d \to \mathbb{Z}$ is a linear map.
Note that unlike much of the literature we do not assume $C=-C$.
Note also that a $d$-GAP is a convex $d$-progression 
where $C$ is an axis-aligned box.
Note further that we can allow $\mathbb{Z}^d$ to be replaced 
by any $d$-dimensional lattice $L=\psi(\mathbb{Z}^d)$,
as $\phi(C\cap L) = (\phi \circ \psi)( \psi^{-1}(C) \cap \mb{Z}^d)$.

Let $P = \phi(C\cap \mathbb{Z}^d)$ and $s,n \in \mb{N}$.
We say $P$ is \emph{$s$-proper} 
if $\card \phi(sC\cap\mathbb{Z}^d)=\card(sC\cap\mathbb{Z}^d)$.
We say $P$ is \emph{$n$-full} if for each $i\in[d]$ there is $x\in\mathbb{Z}^d$ 
with $\card(C\cap(x+\mathbb{Z}e_i))\geq n$. 

The following corollary of \cite[Lemmas 3.24 and 3.26]{tao2006additive}
shows that the size of a centrally symmetric convex progression 
is within a constant factor of the volume of its convex hull.

\begin{prop}\label{fullcoprogswitch}
Suppose $C \subset \mathbb{R}^d$ is a convex body with $C=-C$.
Then $\card (C  \cap \mathbb{Z}^d) =\Theta_d(|C|)$.
\end{prop}

For large dilates of any convex body a classical result of Gauss
(see \cite[Lemma 3.22]{tao2006additive}) 
gives the following asymptotic approximation.

\begin{prop}\label{gauss}
Suppose $C \subset \mathbb{R}^d$ is a convex body.
Then $\card (nC  \cap \mathbb{Z}^d) = |C| n^d + O_C(n^{d-1})$ for large $n$.
\end{prop}
This proposition trivially extends to finite unions of convex bodies.

\subsection{Generalised progressions}

A \emph{convex $(k,d)$-progression} is a set $P=\phi(C')$, 
where $C\subset\mathbb{R}^{k+d}$ is convex,
$C' = C\cap (\mathbb{R}^k\times \mathbb{Z}^d)$ and 
$\phi:\mathbb{R}^{k+d}\to\mathbb{R}^k$ is a linear map.
We say $P$ is \emph{proper} if
$|P| = |C'| := \sum_{x \in \mathbb{Z}^d} |C'_x|$,
where each $C'_x = C' \cap (\mathbb{R}^k\times \{x\})$.

Note that a convex $(1,d)$-progression is obtained from
a convex $d$-progression by attaching intervals at each point.
If these intervals are sufficiently small then these structures are roughly
equivalent for our purposes (see \Cref{zrkequivalence}).
Similarly to \Cref{fullcoprogswitch}, if $C=-C$ then
the volumes of $P$ and $C$ as above are within a constant factor.

\begin{prop}\label{switch2}
Suppose $C \subset \mathbb{R}^{k+d}$ is a convex body with $C=-C$.
Then $\sum_{x \in \mb{Z}^d} |C_x| = \Theta_{d,k}(|C|)$.
\end{prop}

\begin{proof}
Let $\LL = (n^{-1} \mb{Z})^k \times \mb{Z}^d$, where $n$ is large.
Then $|\LL \cap C| = \sum_{x \in \mb{Z}^d} |C_x \cap (n^{-1} \mb{Z})^k|
= O_C(n^{k-1}) + n^k \sum_{x \in \mb{Z}^d} |C_x|$ by \Cref{gauss}.
On the other hand, $|\LL \cap C| = |\mb{Z}^{d+k} \cap C'|$
where $C'$ is obtained from $C$ by dilating $\mb{R}^k$ by $n$.
Thus by \Cref{fullcoprogswitch} $|\LL \cap C| = \Theta_{d,k} (|C'|) = \Theta_{d,k} (|C| n^k)$.
The result follows.
\end{proof}

\subsection{Covering by (generalised) progressions}

Let $A\subset \mathbb{R}^k$. We let $\gap^{k,d}_t(A)$
be a minimum volume set $X+P+Q$ containing $A$ such that
$\card X\leq t$, $P$ is a proper $d$-GAP in $\mb{R}^k$
and $Q \sub \mb{R}^k$ is a parallelotope. 
We let $\co_{t}^{k,d}(A)$ be a minimum volume set $X+P$ 
containing $A$ such that $\card X\leq t$ 
and $P$ is a proper convex $(k,d)$-progression;
we define  $\sco_{t}^{k,d}(A)$ similarly also requiring $P=-P$.
As remarked earlier, we will use the following
rough equivalence between $\sco$ and $\gap$.

\begin{lem} \label{gap=sco}
For any $A\subset \mathbb{R}^k$ we have
$\left|\sco^{d,k}_t(A)\right| =\Theta_{d,k}\left(\left|\gap^{d,k}_t(A)\right|\right)$.
\end{lem}

To see \Cref{gap=sco}, it suffices to show for any
proper convex $(k,d)$-progression $P'$ with $P'=-P'$ that $P' \sub P+Q$
for some proper $d$-GAP $P$ and parallelotope $Q$ with $|P+Q|=O_{d,k}(|P'|)$.
This can be considered a John-type theorem that mixes 
continuous approximation of a convex set by a parallelotope
(which can be achieved by John's Theorem)
and discrete approximation of a convex progression by a GAP.
We will deduce this in \Cref{sec:merge} from a merging argument
that reduces it to the following discrete John-type theorem
of Tao and Vu \cite[Theorem 1.6]{tao2008john} quantitatively optimised by the current authors \cite[Corollary 1.2]{van2023sharp}.

\begin{thm} \label{TaoVu}
Let $P$ be a proper convex $d$-progression with $P=-P$.
Then there is a proper $d$-GAP $P'$ 
with $P\subset P'$ and $\card P' =O(d)^{3d}\card P=O_{d}(\card P )$.
\end{thm}

\subsection{Separated lifts}

Given an abelian group $G$ and a subsets $A,B\subset G$ with $0\in B=-B$, 
we say $A$ is \emph{$B$-separated} if  $a-a'\not\in B$ for all distinct $a, a'\in A$. 
This notion of separation will be fundamental throughout the paper.
Here we use it to justify the following lifting construction
for embedding structures $P+Q$ appearing in the gap definition
in their `natural setting' via a Freiman isomorphism.

\begin{prop} \label{lift}
Suppose $Q \sub \mb{R}^k$ is a parallelotope
and $P = P(a; \ell)\sub{R}^k$ is a $2$-proper $4Q$-separated $d$-GAP.
Define $f_{P,Q}: P+Q \to \mb{Z}^d \times \mb{R}^k$
by $f_{P,Q}(p+q) = (\lL,q)$ where $p=\sum_i \lL_i a_i$.
Then $f_{P,Q}$ is a well-defined Freiman isomorphism
of $P+Q$ with its image.
\end{prop}

\begin{proof}
As $P$ is proper, $f_{P,Q}$ is well-defined.
Consider $p^j \in P$ and $q^j \in Q$ for $j \in [4]$.
If $f_{P,Q}(p^1+q^1) + f_{P,Q}(p^2+q^2) = f_{P,Q}(p^3+q^3) + f_{P,Q}(p^4+q^4)$
then clearly $p^1+q^1 + p^2+q^2 = p^3+q^3 + p^4+q^4$.
Conversely, if $p^1+q^1 + p^2+q^2 = p^3+q^3 + p^4+q^4$
then $p^1+ p^2 - p^3 - p^4 \in 4Q$, so by separation
$p^1+ p^2 = p^3 + p^4$. We deduce  $q^1+ q^2 = q^3 + q^4$,
and as $2P$ is proper also $\lL^1+ \lL^2 = \lL^3 + \lL^4$,
where each $p^j = \sum_i \lL^j_i a_i$. Therefore
$f_{P,Q}(p^1+q^1) + f_{P,Q}(p^2+q^2) = f_{P,Q}(p^3+q^3) + f_{P,Q}(p^4+q^4)$.
\end{proof}

\subsection{Discrete Brunn-Minkowski and stability}\label{discreteBMsec}
 
Here we consider for $A,B \sub \mb{Z}^k$ whether  $\card(A+B)$
has an approximate lower bound $((\card A)^{1/k} + (\card B)^{1/k})^k$
analogous to Brunn-Minkowski in $\mb{R}^k$.
Clearly, some non-degeneracy condition is necessary.
It is natural to consider the \emph{thickness} $h(A)$,
defined as the smallest $h$ such that 
$A$ is contained within $h$ parallel hyperplanes.
Ruzsa \cite[Conjecture 4.12]{ruzsa2006additive} made the following conjecture,
noting that the case $A=B$ holds
by a result of Freiman (also presented by Bilu \cite{bilu1999structure}).

\begin{conj}\label{ruzsaconj}
If $A,B \subset \mb{R}^k$ with $h(B)>h(\eps,k)$ sufficiently large
then $\card(A+B) \ge ((\card A)^{1/k} + (1-\eps)(\card B)^{1/k})^k$.
\end{conj}

As mentioned above, this conjecture will be deduced from the theory
developed in this paper.
We require the following stability version of this inequality for $A=B$
due to van Hintum, Spink and Tiba \cite[Corollary 1.5]{van2020sets}.
For $A\subset\mathbb{Z}^k$ we define its discrete convex hull
as $\widehat{\co}(A):=\co(A)\cap \Lambda$, 
where $\Lambda$ is the smallest sublattice of $\mathbb{Z}^k$ containing $A$.

\begin{thm} \label{zkstab}
For any $k \in \mb{N}$ there is $\Delta_k >0$ such that if $\dD \in (0,\Delta_k)$
and $A \sub \mb{Z}^k$ with $\card(A+A) \le (2^k + \dD)\card A$ 
and $h(A) \ge \OO_{k,\dD}(1)$ then 
$\card(\widehat{\co}(A) \sm A) \le (4k)^{5k} \dD \card A$.
\end{thm}

\subsection{Discrete Brunn-Minkowski for separated lifts}

We require the following adaptation of  a discrete Brunn-Minkowski inequality
due to Green and Tao  \cite[Lemma 2.8]{green2006compressions}.
For each $I \subseteq [d]$ we let 
$\pi'_I : \mathbb{Z}^d\times \mathbb{R}^k \to \mathbb{Z}^{[d] \sm I} \times \mathbb{R}^k$
denote the projection orthogonal to the directions in $I$. 
In the following statement we note that $\pi'_\es(A+B)=A+B$,
so we obtain a lower bound on $|A+B|$ that is close 
to the $(d+k)$-dimensional Brunn-Minkowski bound,
provided that we have good upper bounds on 
the volumes of the non-trivial projections
$\pi'_I(A+B)$ with $I \ne \es$.

\begin{lem}\label{discBM}
Let $A,B\subset \mathbb{Z}^d\times \mathbb{R}^k$. Then
$$\sum_{I\subseteq[d]} |\pi'_I(A+B)| \geq \left(|A|^{1/(k+d)}+|B|^{1/(k+d)}\right)^{k+d}.$$
\end{lem}

Following Green and Tao, we introduce a cube summand and apply Brunn-Minkowski.

\begin{prop} \label{cubeBM}
Let $A,B\subset \mathbb{Z}^d\times \mathbb{R}^k$. Then
$$\left|A+B+(\{0,1\}^d\times \{0\}^k)\right|^{1/(k+d)}\geq |A|^{1/(k+d)}+|B|^{1/(k+d)}.$$
\end{prop}

\begin{proof}
Note that for any $X\subset \mathbb{Z}^d\times \mathbb{R}^k$, 
if we consider $X':=X+([0,1]^d\times \{0\}^k)$ as a subset of $\mathbb{R}^{d+k}$, 
then $|X|_{\mathbb{Z}^d\times \mathbb{R}^k}=|X'|_{ \mathbb{R}^{d+k}}$.   
By Brunn-Minkowski in $\mathbb{R}^{d+k}$ we obtain
\begin{align*}
\left|A+B+(\{0,1\}^d\times \{0\}^k)\right|^{1/(k+d)}
&=\left|(A+([0,1]^d\times \{0\}^k))+(B+([0,1]^d\times \{0\}^k))\right|^{1/(k+d)}\\
&\geq  \left|A+([0,1]^d\times \{0\}^k)\right|^{1/(k+d)}+\left|B+([0,1]^d\times \{0\}^k)\right|^{1/(k+d)}\\
&=\left|A\right|^{1/(k+d)}+\left|B\right|^{1/(k+d)},
\end{align*}
where in the first and last expressions 
the volumes are in $\mathbb{Z}^d\times \mathbb{R}^k$,
whereas in the second and third the volumes are in $\mathbb{R}^{d+k}$.
\end{proof}

We now deduce the lemma by compression.

\begin{proof}[Proof of \Cref{discBM}]
For each $i \in [d]$ and $A \sub \mb{Z}^d\times \mb{R}^k$
we define the $i$-compressed set $C_i(A) \sub \mb{Z}^d\times \mb{R}^k$
by replacing each $\pi'_i$-fibre by an initial segment of $\mb{N}$:
for any $y \in \pi'_i(A)$ with $|A \cap \pi_i'^{-1}(y)|=t$ we let 
$$C_i(A) \cap \pi_i'^{-1}(y) = \{ x \in \pi_i'^{-1}(y) : 0 \le x_i \le t-1 \}.$$  
Equivalently, $ \pi'_i( \{ x \in C_i(a): x_i = t \} ) 
= \{ y \in \pi'_i(A) : |A \cap \pi_i'^{-1}(y)| \ge t \}$,
so $C_i(A)$ is measurable.

Now we consider the compressed sets
$A':=C_1(\dots C_d(A)\dots)$ and $B':=C_1(\dots C_d(B)\dots)$.
Clearly $|A'|=|A|$ and $|B'|=|B|$. Moreover, one can show that
each $|\pi'_I(A'+B')|\leq |\pi'_I(A+B)|$. Thus it suffices to prove
the lemma assuming that $A$ and $B$ are compressed.
In this case, we have an identity
\[  \sum_{I\subseteq[d]} |\pi'_I(A+B)| 
= |A+B-(\{0,1\}^d\times \{0\}^k)| = |A+B+(\{0,1\}^d\times \{0\}^k)|, \] 
as $-(\{0,1\}^d\times \{0\}^k)$ is a translate of $\{0,1\}^d\times \{0\}^k$
and the sets $\pi'_I(A+B) - \sum_{i \in I} e_i$ for $I \subseteq[d]$
partition $A+B-(\{0,1\}^d\times \{0\}^k)$. 
The lemma now follows from \Cref{cubeBM}.
\end{proof}

We deduce the following $(d+k)$-dimensional 
approximate Brunn-Minkowski bound
for separated lifts as in \Cref{lift}.

\begin{cor}\label{BMcor}
Suppose $Q \sub \mathbb{R}^k$ is a parallelotope and $P = P(a; \ell)\sub \mathbb{R}^k$ 
is an $n$-full $2$-proper $4Q$-separated $d$-GAP.
Then for any $Y,Z\subset P+Q$ we have
$$|Y+Z|\geq \left(|Y|^{1/(k+d)}+|Z|^{1/(k+d)}\right)^{k+d}-2^{2d+k}n^{-1}\card P \cdot\left|Q\right|.$$
\end{cor}

\begin{proof}
Consider $Y'=f_{P,Q}(Y)$ and $Z'=f_{P,Q}(Y)$ in $\mb{Z}^d \times \mb{R}^k$,
where $f_{P,Q}$ is the Freiman isomorphism from \Cref{lift}.
Then $|Y| = |Y'| := \sum_{x \in \mb{Z}^d} |Y_x|$, 
similarly $|Z|=|Z'|$, and $|Y+Z|=|Y'+Z'|$.
The result now follows from applying  \Cref{discBM} to $Y'$ and $Z'$,
noting that as $P$ is $n$-full, for any non-trivial projection $\pi_I$ we have
 $|\pi_I(Y'+Z')| \le |\pi_I(2\cdot P+2Q)| \le n^{-1} \card(2\cdot P)|2Q| \le 2^{d+k} n^{-1} \card P \cdot|Q|$. 
\end{proof}

\subsection{Approximation}

The following standard approximation argument (see e.g.\ \cite{figalli2015quantitative})
will allow us to approximate a general measurable set by a finite union of cubes.
We say that $P \sub \mb{R}^k$ is a \emph{polycube} if $P=S+[0,c]^k$
for some $c>0$ and finite $S \sub c\mb{Z}^k$.

\begin{lem}\label{cubesreduction}
For any $A\sub\mb{R}^k$ with $|A|<\infty$ and $\eta>0$ there is a polycube $A'$ with 
$|A\triangle A'|\leq \eta |A|$, $|\co(A')|\geq |\co(A)|-\eta|A|$ and $|A'+A'|\leq |A+A|+\eta |A|$.
\end{lem}
\begin{proof}
We can approximate $A$ by a compact set $K\subset A$ so that $|A\setminus K|\leq \frac12 \eta |A|$. 
We can also ensure $|\co(K)|\geq |\co(A)|-\eta |A|$ by additionally requiring $K$ 
to contain a large finite subset of the extreme points of $\co(A)$. Let 
$$B_N:=\left\{x\in(\mathbb{Z}/N)^k:\left(x+[0,N^{-1}]^k\right)\cap K\neq \emptyset\right\}+[0,N^{-1}]^k.$$
Then $K\subset B_N$, and as $K$ is compact, $|B_N\setminus K| \to 0$ as $N\to \infty$. 
Moreover, as $K+K$ is compact, $|(B_N+B_N)\sm (K+K)| \to 0$ as $N\to\infty$. 
Clearly, $|\co(B_N)|\geq |\co(K)|\geq |\co(A)|-\eta|A|$ for all $N$. 
Hence, choosing $N$ sufficiently large and setting $A':=B_N$, the result follows.
\end{proof}

The same proof gives the following variant that will be used in the next subsection.
Here for $A\sub\mb{Z}^d \times \mb{R}^k$
we define its generalised convex hull as 
$\widehat{\co}^{k,d}(A):=\co(A)\cap (\LL_A\times \mathbb{R}^k)$, 
where $\LL_A$ is the smallest sublattice of $\mathbb{Z}^d$ 
such that $A\subset \LL_A\times\mathbb{R}^k$. 
Note that for $A\sub\mb{Z}^d \times \mb{R}^0 = \mb{Z}^d$
we have  $\widehat{\co}^{0,d}(A) = \widehat{\co}(A)$
(the discrete convex hull defined above in \Cref{discreteBMsec}).

\begin{lem}\label{cubesreduction2}
For any $A\sub\mb{Z}^d \times \mb{R}^k$ with $|A|<\infty$ and $\eta>0$ there is a polycube $A'$ with 
$|A\triangle A'|\leq \eta |A|$, $|\widehat{\co}^{d,k}(A')|\geq |\widehat{\co}^{d,k}(A)|-\eta|A|$,
$h(A')=h(A)$ and $|A'+A'|\leq |A+A|+\eta |A|$.
\end{lem}

\subsection{Discretisation}

The following discretisation lemma will allow us to pass
from a separated lift in $\mb{Z}^d \times \mb{R}^k$
to an essentially equivalent discrete set in  $\mb{Z}^{d+k}$.
We recall that the thickness $h(A)$ is the smallest $h$ 
such that $A$ is contained within $h$ parallel hyperplanes.

\begin{lem}\label{switchtodiscrete}
For any $A\sub\mb{Z}^d \times \mb{R}^k$ with $|A|<\infty$ and $\eta>0$ 
there is $A'\subset \mathbb{Z}^{d+k}$ with $h(A')=h(A)$, 
$$\card(A'+A')/\card A' \leq (1+\eta) |A+A|/|A|   \text{ and } 
  |\widehat{\co}^{d,k}(A)|/|A| \leq (1+\eta)\card \widehat{\co}(A')/\card A'.$$
\end{lem}

\begin{proof}
By  \Cref{cubesreduction2}, we can approximate $A$ by cubes, which we rescale to have unit length, 
obtaining $A_1 = Y+(\{0\}^d\times[0,1]^k)$ for some $Y\sub\mb{Z}^{d+k}$ with $h(Y)=h(A)$ such that
$$|A_1+A_1|/|A_1| \le (1+\eta) |A+A|/|A| \text{ and}$$
$$|\widehat{\co}^{d,k}(A)|/|A| \leq (1+\eta/2)|\widehat{\co}^{d,k}(A_1)|/|A_1|.$$
We consider $N \in \mb{N}$ sufficiently large and let
$A' = NY + \{0\}^d \times \{0,\dots,N-1\}^k$.
We note that $h(A')=h(A)$ and $NA_1 = A' + \{0\}^d \times [0,1]^k$,
so $\card A' = |NA_1| = N^k |A_1|$. We also have
\begin{align*} \card(A'+A') &= |A'+A'+(\{0\}^d \times [0,1]^k)|
\le  |A'+A'+(\{0\}^d \times [0,2]^k)| \\
&= |NA_1 + NA_1| = N^k|A_1 + A_1|,\end{align*}
so $\card(A'+A')/\card A'  \le |A_1+A_1|/|A_1| \le (1+\eta) |A+A|/|A|$.
It remains to estimate $\card \widehat{\co}(A')$.
We note that $\co(A') = \co(NY) +  (\{0\}^d \times [0,\dots,N-1]^k)$
and $\co(NA_1) = \co(NY) + ( \{0\}^d \times [0,\dots,N]^k)$.
Thus 
\begin{align*} 
\card \widehat{\co}(A') 
& = \sum_{x \in N\mb{Z}^d} \card ( \co(A')_x \cap \mb{Z}^k)
= \sum_{x \in N\mb{Z}^d} (1+o_N(1)) |\co(A')_x|\\
&=  \sum_{x \in N\mb{Z}^d} (1+o_N(1)) |\co(NA_1)_x| 
 = (1+o_N(1)) \sum_{y \in \mb{Z}^d} N^k |\co(A_1)_y| \\
&= (1+o_N(1))  N^k |\widehat{\co}^{d,k}(A_1)|. 
\end{align*}
We deduce $\card \widehat{\co}(A')/\card A' = 
(1+o_N(1)) |\widehat{\co}^{d,k}(A_1)| / |A_1|$,
so choosing $N$ sufficiently large $|\widehat{\co}^{d,k}(A)| / |A| \le (1+\eta) \card \widehat{\co}(A')/\card A'$.
\end{proof}

\subsection{Examples}

We conclude this preliminary section by collecting some examples of sets with small doubling
that illustrate the tightness of various results stated in the introduction.

\begin{exmp}[Two boxes] \label{ex:2}
Consider two well-separated boxes of volumes $\delta^2$ and $1-\delta^2$: let
$$A=[0,\sqrt[k]{1-\delta^2}]^k \cup ([0,\sqrt[k]{\delta^2}]^k + v), 
\text{ where } v\in\mathbb{R}^k \text{ with } \|v\| \text{ large}.$$
A convex set of volume $O(|A|)$ can cover volume at most $1-\delta^2$ of $A$.
As $k\to\infty$,
$$\left|\frac{A+A}{2}\right|=1+\left(\frac{\sqrt[k]{1-\delta^2}+\sqrt[k]{\delta^2}}{2}\right)^k
\to 1+\delta\sqrt{1-\delta^2} = 1 + \delta + O(\delta^2).$$
A similar example replacing boxes by GAPs shows that
\Cref{linearandquadraticintegers} is sharp for large $k$ as $\delta\to 0$.
\end{exmp}

\begin{exmp}[AP of boxes] \label{ex:AP}
Let $t \in \mb{N}$, $v\in\mathbb{R}^k$ with $\|v\|$ large and
$$A=([0,1]^{k-1}\times [0,1/t]) +\left\{v,2v,\dots,tv\right\}.$$
Then $|A|=1$, $\left|\tfrac{A+A}{2}\right| = 2-1/t$ and a convex set of volume $O(|A|)$ 
can cover volume at most $1/t$ of $A$.
A similar example replacing boxes by GAPs shows that
\Cref{linearandquadraticintegers} is sharp as $\delta\to 1$.
\end{exmp}

\begin{exmp}[Convex set plus scattered points] \label{ex:scatter}
Let $A = C \cup S \sub \mathbb{R}^k$ where $C$ is a convex body (so of positive measure) of bounded radius
and $S$ is a finite set that is \emph{scattered}, in that any distinct $x,y$ in $S$ are at distance
at least $D$ and have $(x+C) \cap (y+C) = \emptyset$. Then $|A+A| = (2^k + \card S)|A|$ and
$A$ cannot be covered by $\card S$ convex sets of volume $O(|A|)$ as $D \to \infty$.
An analogous example where $C \sub \mathbb{Z}$ is a $k$-GAP
shows tightness of the first bound in \Cref{FewLocationsIntegers}.
\end{exmp}

\begin{exmp}[Cone over GAP] \label{ex:cone}
Let  $t \in \mb{N}$, $v\in\mathbb{Z}^k$ with $\|v\|$ large and
$$ A = \bigcup_{i=1}^{t} ([0,i]^{k}+iv). $$
Then $|A| = \sum_{i=1}^{t}i^k=(1+o_t(1))\frac{t^{k+1}}{k+1}$ and
\begin{align*}
\left|A+A\right|&= \sum_{i=2}^{2t}i^{k}=2^{k+1}\sum_{i=1}^{t}i^k-\sum_{i=1}^{t}\left[(2i)^k-\left(2i-1\right)^k\right]-1\\
&\leq 2^{k+1}|A|-\frac12 (2t)^{k}=\left(2^{k+1}-(1+o(1))2^{k}\frac{k+1}{2t}\right)|A|.
\end{align*}
Hence, $t\geq (1+o(1))\frac{k+1}{2\epsilon}$ where $|A+A|=2^{k}(2-\epsilon)|A|$.
Applying a Freiman isomorphism to a subset of $\mb{Z}$, we see that 
the final bound in \Cref{FewLocationsIntegers} is asymptotically optimal.
\end{exmp}

\begin{exmp}[House] \label{ex:apt}
Let $t\in\mathbb{N},v,\dD>0$ with $v$ large, $\dD$ small and
$$A=\bigcup_{i=-t}^{t}iv+[0,1+\dD (1-|i|/t)].$$
Then $|A|=(2t+1)+\dD (t-1)$, 
$$|A+A|=\sum_{i=-2t}^{i=2t}2(1+\dD (2-|i|/t))=2(4t+1)+4\dD (t-1)
=\left(4-\frac{2}{2t+1}\right)|A|+O(\dD),$$
and
$$\text{ap}_{2t+1}(A)=\bigcup_{i=-t}^{t}iv+[0,1+\dD], 
\text{ so that }|\text{ap}_{2t+1}(A)\setminus A|\geq \Omega(\dD t).$$
\end{exmp}

\section{Continuous Freiman Theorem} \label{sec:cts}

In this section we prove the following continuous version of Freiman's Theorem,
showing that $A \sub \mb{R}^k$ of bounded doubling has a generalised hull
of the correct dimension with volume within a constant factor of that of $A$.

\begin{thm}\label{GreenTaoContinuous}
For any $d,k \in \mb{N}$ and $\epsilon>0$,
if $A\subset \mathbb{R}^k$ with $|A+A|\leq (2^{k+d+1}-\epsilon)|A|$ then 
\[ \left|\gap^{k,d}_{O_{k,d,\epsilon}(1)}(A)\right| = O_{k,d,\epsilon}(|A|),\]
i.e.\  $A\subset X+P+Q$ where $P$ is a proper $d$-GAP, $Q$ is a parallelotope 
and $\card X +  \card P\cdot |Q|/|A| = O_{k,d,\epsilon}(1)$.
\end{thm}

The proof requires the following lemma.

\begin{lem}\label{ContGTlem}
Let $P,I\subset \mathbb{R}^k$, where $P$ is a proper $d$-GAP and $I$ is convex. Then
$$|P+I|\leq \frac{2^{k+d}\card P \cdot|I|}{\max_{x\in\mathbb{R}^k} \card \{(p,i)\in P\times I: p+i=x\} }.$$
\end{lem}

\begin{proof}
Suppose that the maximum in the denominator is achieved at some $x_0$.
Write $x_0 = p_0 + i_0$ with $p_0 \in P$ and $i_0 \in I$.
Replacing $P$ by $P-p_0$ and $I$ by $I-i_0$,
we can assume that $0 \in P \cap I$ 
and the maximum is $\card(P \cap (-I))$ achieved at $0$.
The lemma follows from the inequality
\[ \card(P+P)\cdot |I+I| = \int (P+P)*(I+I) \ge |P+I| \cdot \card(P \cap (-I)). \]
The inequality holds as for every $x=p+i \in P+I$ there are
at least $\card(P \cap (-I))$ solutions $(q,j) \in  (P+P) \times (I+I)$ to $q+j=x$;
indeed, for each $y \in P \cap (-I)$ we have a solution $(p+y,i-y)$.
\end{proof}

\begin{proof}[Proof of \Cref{GreenTaoContinuous}]
We start by applying \Cref{cubesreduction},
with some sufficiently small $\eta = \eta_{d,k,\epsilon} > 0$, obtaining $A'$ that is
a finite union of axis-aligned cubes, say with side length $\bB>0$, such that
$|A\triangle A'|\leq \eta |A|$ and $|A'+A'|\leq |A+A|+\eta |A|$.
We will discretise at some small scale $1/N$ with $N \in \mb{N}$ 
and use asymptotic notation considered as $N \to \infty$.
Let $A_N\subset\mathbb{Z}^k$ be defined by
$$A_N:=\{x\in\mathbb{Z}^k: (x+[0,1]^k)\cap NA'\neq\emptyset\}.$$
For large $N$, by \Cref{gauss} we have $\card A_N = N^k |A'| + O(N^{k-1})$
and $\card(A_N+A_N) = N^k |A' + A'| + O(N^{k-1})$,
so $\card(A_N+A_N)\leq (2^{d+k+1}-\epsilon/2)\card A_N$ for sufficiently large $N$.

By  \Cref{greentaooriginal} we have $A_N \sub X'+P'$
for some $(d+k)$-GAP $P'$ with $\card P' \leq \card A_N$
and $X'$ with $|X'| = O_{d,k,\epsilon}(1)$.
We can also assume that $P'=-P'$ and $2\cdot P'=P'-P'$ is proper.
We write $P'=P(a;\ell)$ with $a \in (\mb{R}^k)^{d+k}$ and $\ell \in \mb{N}^{d+k}$.
 
\begin{clm}
There is a set of step sizes in $a$, say $a_{d+1},\dots,a_{d+k}$,
which span $\mathbb{R}^k$ and all have $\|a_i\|=o(N)$.
\end{clm}
\begin{proof}[Proof of claim]
We will find such $a_i$ all with $\|a_i\| \le N^{1-1/2(k+d)}$ for large $N$.
Suppose for a contradiction that there is a hyperplane $H$
such that any $a_i \notin H$ has $\|a_i\| > N^{1-1/2(k+d)}$.
Consider the projection $\pi:\mathbb{R}^k\to\mathbb{R}$
in the direction orthogonal to $H$.
By \Cref{gauss} we have $\card(\pi(X'+P')) \ge \card(\pi(A_N))=\Omega(N)$,
so $\card(\pi(P'))=\Omega(N)$. Thus there is some $a_i$
with $a_i \notin H$ and $\ell_i = \Omega(N^{1/(k+d)})$.
As $a_i \notin H$ we have $\|a_i\| > N^{1-1/2(k+d)}$.
We write $P' = P(a',\ell') + P(a_i,\ell_i)$ where $(a',\ell')$ 
is obtained from $(a,\ell)$ by removing $(a_i,\ell_i)$.
Now recall that $A'$ is a union of cubes with side length $\bB$.
For any such cube $C$ in $A'$, any translate of $P(a_i,\ell_i)$
intersects $NC$ (which has radius $\sqrt{k}\beta N$) in at most $M := \sqrt{k} \bB N / \|a_i\|$ points. 
Thus for large $N$ we have
\begin{align*} \OO(N^k) &= \card(A_N \cap NC) \le   \card((X'+P') \cap NC) 
= O_{d,k,\eps}( M \cdot \card P(a',\ell') ) \\
&= O_{d,k,\eps} \left(  \frac{ \card P' \cdot N }{ \ell_i \|a_i\| } \right) = o(N^k),\end{align*}
as $\card P' \leq \card A_N = O(N^k)$, $\|a_i\| > N^{1-1/2(k+d)}$ 
and $\ell_i = \Omega(N^{1/(k+d)})$.
This contradiction proves the claim.
\end{proof}

Now we write $P' = P^0 + P^*$ where $P^0 = P(a_1,\dots,a_{k},\ell_{1},\dots,\ell_{k})$
and $P^* = P(a_{d+1},\dots,a_{d+k},\ell_{d+1},\dots,\ell_{d+k})$.
We let $I^*:=\co(P^*)+[0,1]^k$ and note that $NA' \sub X'+P^0+I^*$.

Our next aim is to show that $|P^0+I^*|$ is comparable with $\card P'$. We consider 
\[ T:=\left\{-\sum_{i=d+1}^{d+k}\lambda_ia_i: \lambda_i\in[0,1) \text{ for } d+1 \le i \le d+k \right\}\cap \mathbb{Z}^k. \]
By a standard tiling argument, $\card T=|\co(T)|=\det(a_{d+1},\dots,a_{d+k})$. 
As $\|a_i\|=o(N)$ for $d+1 \le i \le d+k$ we can translate so that $T\subset A_N$.
Clearly any translate of $P^*$ intersects $T$ in at most one point. 
As $T \sub A_N \sub X'+P' = X'+P^0+P^*$, there is a translate $z + P^0$ 
with $\card((z + P^0) \cap T) \ge  \card T/\card X'$. As $-T \sub I^*$ we obtain
at least $\card T/\card X'$ solutions to $p+i=-z$ with $(p,i) \in P^0 \times I^*$. 

By \Cref{ContGTlem} this implies
\[ |P^0+I^*| \le \frac{ 2^{d+k}\card P^0\cdot |I^*| }{ \card T /\card X'} = O_{d,k,\eps}(\card P'), \]
as $|I^*| = O(\card T \cdot \card P^*)$ and $\card P' = \card P^0 \cdot \card P^*$.

As $NA' \sub X'+P^0+I^*$, we can fix $x \in X'$ such that 
$A'_x:=A'\cap (x+P^0+I^*)/N$ has $|A'_x| = \Omega_{d,k,\eps}(|A'|)$.
Recalling $|A\triangle A'|\leq \eta |A|$ with $\eta$ sufficiently small,
$A_x:=A \cap (x+P^0+I^*)/N$ also has $|A_x| = \Omega_{d,k,\eps}(|A|)$.

As $|A+A_x| \le |A+A| = O_{d,k,\eps}(|A_x|)$, by the Ruzsa Covering Lemma (\Cref{ruzsa})
we have $A \sub X + A_x - A_x$ for some $X\subset \mb{R}^k$ with $\card X  = O_{d,k,\eps}(1)$.
We let $P = (P^0-P^0)/N$ and $Q$ be the smallest parallelotope containing $(I^*-I^*)/N$, so that
as $A_x \sub (x+P^0+I^*)/N$ we have $A_x-A_x \sub P+Q$.
Thus $A \sub X + P + Q$ where $P$ is a proper $d$-GAP, $Q$ is a parallelotope,
$\card X  = O_{d,k,\eps}(1)$ and $|P+Q| = O_{d,k}(N^{-k}|P^0+I^*|) =  O_{d,k,\eps}(|A|)$,
as $|P^0+I^*| = O_{d,k,\eps}(\card P')$ and $\card P' \le \card A_N = O(N^k |A|)$.
This completes the proof.
\end{proof}

\section{Merging} \label{sec:merge}

In this section we prove the following result that is halfway between 
the Continuous Freiman Theorem of the previous section
and the Separated Freiman Theorem proved in the next section,
in that we ensure that $P$ is well-separated for $Q$,
but $X$ is not yet well-separated for $P+Q$.

\begin{prop}\label{SeparatedFreiman}
Let $k,d,\ell_0, n_i,s\in\mathbb{N}$, for $1\leq i\leq d$, $\epsilon>0$, and $A\subset\mb{R}^k$. 
If $|A+A|\leq (2^{k+d+1}-\epsilon)|A|$ then $A\subset X+P+Q$, where
\begin{itemize}
    \item $P$ is an $s$-proper $n_{d'}$-full $d'$-GAP with $d'\leq d$,
    \item $Q$ is a parallellotope,
    \item $\card X =O_{k,d,\epsilon,n_{d'+1},\dots,n_{d}}(1)$,
    \item $\card P \cdot |Q|=O_{k,d,\epsilon,\ell_0,s}(|A|)$, and
    \item $s\cdot P$ is $\ell_0Q$-separated.
\end{itemize}
\end{prop}

The technical heart of this section is a merging procedure,
which we will use to prove \Cref{SeparatedFreiman},
and also the two equivalences discussed in the introduction
($\mb{Z}$ versus $\mb{R}$ and $\gap$ versus $\sco$).

\subsection{Geometric preliminaries}

We start with some geometry.

\begin{prop}\label{projectioncontainmentprop}
Consider a convex body $C\subset\mathbb{R}^d$ and a vector $\rho\in\mathbb{R}^d$. 
Let $m:=\max\{m\in\mathbb{R}\mid \exists x\in C: x+m\rho\in C\}$. 
Then there exists a hyperplane $H\subset\mathbb{R}^d$ not containing $\rho$ so that,
if $\pi_\rho:\mathbb{R}^d\to H$ is the projection onto $H$ along $\rho$, 
then $C\subset \pi_\rho(C)+[0,m] \rho$.
\end{prop}

\begin{proof}[Proof of \Cref{projectioncontainmentprop}]
Let $x\in C$ so that $x+m\rho\in C$. Write $C^\circ$ for the interior of $C$. 
Consider $C^\circ\cap (m\rho+C^\circ)$. By maximality of $m$ this intersection is clearly empty. 
As $C$ and $m\rho+C$ are both convex, there is a separating hyperplane $H$. 
Now note that $C$ lies between the hyperplanes $H$ and $H-m\rho$, 
i.e.\ $C\subset H+[-m,0]\rho$. Moreover, $C\subset \pi_\rho(C)+[-m,0]\rho$, 
where $\pi_\rho:\mathbb{R}^k\to H$ is the projection along $\rho$.
\end{proof}

\begin{prop}\label{containedlinesum}
Let $\ell>0$ and $C\subset \mathbb{R}^k$ 
be a convex body containing the points $0$ and $x$.
Then $|C+[-\ell,\ell]x| \le (2\ell k + 1) |C|$. 
\end{prop}

\begin{proof}[Proof of \Cref{containedlinesum}]
We may assume $x=ae_1$ is in direction $e_1$.
We note that $|C+[-\ell,\ell]x|=|C|+ 2\ell a |\pi_1(C)|$.
Next we note for any $y$ in the boundary $\pl \pi_1(C)$ 
of the projection and $\lL \in [0,1]$ that $\lL y \in \pi_1(C)$
and the $e_1$-fibre above $\lL y$ has length at least $(1-\lL)a$.
Thus $|C|$ is at least the volume of a cone of height $a$ on $\pi_1(C)$,
which is $\tfrac1k a |\pi_1(C)|$. The result follows.
\end{proof}

\subsection{Merging lemma}

In this subsection we prove the following merging lemma,
which carries most of the weight of the results in this section.
We note that a similar argument for convex progressions in $\mb{Z}$
appears in unpublished notes of Green.

\begin{lem}\label{Pmerge}
Fix $\ell_0, s, k, d \in \mb{N}$.
Let $P=\phi(C'\cap \mathbb{Z}^d) \sub \mathbb{R}^k$ 
be an $n$-full convex $d$-progression with $C'=-C'$
and $Q\subset\mathbb{R}^k$ be a convex body with $Q=-Q$.
Suppose there are distinct $\rho,\rho' \in C := sC'$
with $\phi(\rho) \in \phi(\rho') + \ell_0 Q$.
Then we have $P + Q \sub P' + Q'$ for some 
$n$-full convex $(d-1)$-progression $P'$ with $P'=-P'$
and convex body $Q'$ with $Q'=-Q'$ such that 
$\card P' \cdot |Q'|=O_{k,d,\ell_0,s}(\card P \cdot |Q|)$.
\end{lem}

\begin{proof}
We can translate so that $\rho'=0$, and then $\phi(\rho) \in \ell_0 Q$.
Let $m:=\max\{m\in\mathbb{R}\mid \exists x\in C: x+m\rho\in C\}$. 
By \Cref{projectioncontainmentprop} there exists a hyperplane 
$H\subset\mathbb{R}^k$ not containing $\rho$
such that $C\subset \pi_\rho(C)+[0,m] \rho$, where 
$\pi_\rho:\mathbb{R}^k\to H$ is the projection onto $H$ along $\rho$.

We will show that $P':=\phi(\pi_{\rho}(C)\cap \pi_{\rho}(\mb{Z}^{d}))$ 
and $Q':=Q+[-m,m] \phi(\rho)$ satisfy the requirements of the lemma.
Clearly $P'$ is an $n$-full convex $(d-1)$-progression with $P'=-P'$
(recall that intersecting a convex body with any lattice gives a convex progression)
and  $Q'$ is a convex body with $Q'=-Q'$. 

Next we show containment of the sumsets.

\begin{clm}
$P+Q\subset P'+Q'$.
\end{clm}
\begin{proof}[Proof of Claim]
By \Cref{projectioncontainmentprop}, 
$C\subset \pi_\rho(C)+[-m,m]\rho$. Intersecting this with $\mathbb{Z}^k$ implies
$$C\cap \mathbb{Z}^k\subset\pi_\rho(C\cap\mathbb{Z}^k)+[-m,m]\rho\subset \left(\pi_\rho(C)\cap\pi_\rho(\mathbb{Z}^k)\right)+[-m,m]\rho.$$
Applying the linear map $\phi$, we find 
\begin{align*}
    P&=\phi(C'\cap \mathbb{Z}^k)\subset \phi(C\cap \mathbb{Z}^k)\\
    &\subset \phi\left(\left(\pi_\rho(C)\cap\pi_\rho(\mathbb{Z}^k)\right)+[-m,m]\rho\right)=P'+[-m,m]\phi(\rho).
\end{align*}
Thus $P+Q\subset P'+[-m,m]p+Q=P'+Q'$, which proves the claim.
\end{proof}

Now we note that as $0,\frac{1}{\ell_0}\phi(\rho)\in Q$ we have
$|Q'|=O_{k,d,\ell_0}(m|Q|)$ by  \Cref{containedlinesum}.
To complete the proof, it suffices to prove the following claim
showing that $\card P' / \card P$ balances $|Q'|/|Q|$.
\begin{clm}
$\card P' =O_{k,d,s}(m^{-1}\card P )$.
\end{clm}

To prove the claim, we first note that 
$\card P' = \card ( \pi(C)\cap \pi(\mb{Z}^{d}) )$ for any projection $\pi$ 
along $\rho$ onto a hyperplane not containing $\rho$.
For notational convenience, we consider the projection $\pi$
where this hyperplane is $\mathbb{R}^{d-1}\times \{0\}$,
which we assume does not contain $\rho$.
Write $\rho=(\rho',r)$ where $\rho'\in\mb{Z}^{d-1}$
and $r$ is a positive integer (reflecting if necessary).
We note that $\mb{Z}^{d-1}$ is a sublattice of index $r$ in $\pi(\mb{Z}^d)$.
As $\pi(C)$ is centrally symmetric, by  \Cref{fullcoprogswitch} we have
\[ \card P' =  \card ( \pi(C)\cap \pi(\mb{Z}^{d}) ) =\Theta_{d}(r |\pi(C)|). \]
To estimate $|\pi(C)|$, we apply the linear map $T$ with $Te_i=e_i$ for $i<d$
and $Te_d = e_d - \rho'/r$, which has $T \rho = re_d$.
As $0, m\rho \in C$ we have $0, mre_d \in T(C)$,
so $|C| = |T(C)| \ge |\pi(C)|mr/d$ by considering 
a cone on $\pi(C)$ of height $mr$. Therefore
\[  \card P' = O_d(r |\pi(C)|) = O_d(m^{-1}|C|)  =  O_{d,s}(m^{-1}\card P), \] 
where the final estimate uses $C=sC'$ and  \Cref{fullcoprogswitch}.
This completes the proof.
\end{proof}

\subsection{Freiman with partial separation}

We now prove the main result of this section.

\begin{proof}[Proof of \Cref{SeparatedFreiman}]
We start by applying \Cref{GreenTaoContinuous}, obtaining 
$A\subset X'_0+P'_0+Q'_0$ where $P'_0$ is a proper $d$-GAP, $Q'_0$ is a parallelotope 
and $\card X'_0 +  \card P'_0\cdot |Q'_0|/|A| = O_{k,d,\epsilon}(1)$.

Next we will suppress some directions in $P'_0$ to make it sufficiently full.
We write $P'_0=P(a,m)$ with $a \in (\mb{R}^k)^d$ and $m \in \mb{N}^d$.
Let $d_0$ be maximal so that $\card \{i: m_i \ge n_{d_0} \} \ge d_0$.
We write $P'_0 = P_0 + P''_0$ where $P_0$ is the GAP 
spanned by those $a_i$ with $m_i \ge n_{d_0}$ 
and $P''_0$ is the GAP spanned by the other $a_i$. 
We note that if $d_0=d$ then $P_0=P'_0$ is $n_d$-full.
Otherwise, by maximality of $d_0$ 
each step $a_i$ in $P''_0$ has $m_i\leq n_{d_0+1}$.
Now we have $A\subset X'_0+P'_0+Q'_0 = X_0+P_0+Q_0$
where $X_0 = X'_0 + P''_0$ and $Q_0=Q'_0$.
Thus we have replaced $P'_0$ by $P_0$ that is $n_{d_0}$-full,
at the cost of increasing $\card X'_0 = O_{k,d,\epsilon}(1)$
to $\card X_0 = O_{k,d,\epsilon,n_{d_0+1}}(1)$
(writing $n_{d+1}=1$ if $d_0=d$).

Now we apply the following iterative process that reduces the dimension
whenever we have a progression that is poorly separated or insufficiently full.
\begin{itemize}
\item Suppose at step $i$ we have $A \sub X_i + P_i + Q_i$,
where $P_i\subset \mb{R}^k$ is an $n_{d_i}$-full $d_i$-GAP with $P_i=-P_i$ 
and $Q_i\subset\mathbb{R}^k$ is a parallelotope with $Q_i=-Q_i$.
\item Stop if $P_i$ is $s$-proper and $s\cdot P_i$ is $\ell_0 Q_i$ separated.
Otherwise, \Cref{Pmerge} applies to give $P_i + Q_i \sub P^*_{i+1} + Q_{i+1}$ for some 
$n_{d_i}$-full convex $(d_i-1)$-progression $P^*_{i+1}$ with $P^*_{i+1}=-P^*_{i+1}$
and convex body $Q_{i+1}$ with $Q_{i+1}=-Q_{i+1}$ such that 
$\card P^*_{i+1} \cdot |Q_{i+1}|=O_{k,d,\ell_0,s}(\card P_i \cdot |Q_i|)$.
\item Apply  \Cref{TaoVu} to find a $(d_i-1)$-GAP $P'_{i+1}$ 
with $P_{i+1}^*\subset P'_{i+1}$, $P'_{i+1}=-P'_{i+1}$ and $\card P'_{i+1}=O_d(\card P^*_{i+1} )$.
\item Let $d_{i+1}$ be maximal such that at least $d_{i+1}$
directions in $P'_{i+1}$ have at least $n_{d_{i+1}}$ steps.
Split $P'_{i+1}$ as $P_{i+1}  + P''_{i+1}$, where $P_{i+1}$ 
contains all directions with at least $n_{d_{i+1}}$ steps.
\item Go to step $i+1$ with $A \sub X_{i+1}+ P_{i+1}+ Q_{i+1}$,
where $X_{i+1} = X_i + P''_{i+1}$ has $\card X_{i+1} \le n_{d_{i+1}+1}^d \card X_i$
by maximality of $d_{i+1}$.
\end{itemize}
In each step the dimension of the GAP decreases,
so the process terminates in at most $d$ steps with $A \sub X_i + P_i + Q_i$
for some $n_{d_i}$-full $d_i$-GAP $P_i$ and parallelotope $Q_i$
such that $P_i$ is $s$-proper and $s\cdot P_i$ is $\ell_0 Q_i$ separated.
As for bounds, we had $\card X'_0 = O_{k,d,\epsilon}(1)$ and in each step
$\card X_{j+1} /  \card X_j \le n_{d_{j+1}+1}^d$, 
so $\card X_i = O_{k,d,\epsilon,n_{d_i+1},\dots,n_d}(1)$,
where we note that this bound does not depend on $n_{d_i}$.
We also had $\card P'_0\cdot |Q'_0| = O_{k,d,\epsilon}(|A|)$
and in each step 
$\card P_{j+1} \cdot |Q_{j+1}|=O_{k,d,\epsilon,\ell_0,s}(\card P_j \cdot |Q_j|)$,
so $\card P_i \cdot |Q_i| = O_{k,d,\epsilon,\ell_0,s}(|A|)$.
The theorem follows.
\end{proof}

\subsection{Equivalence of gap and sco}

Here we give another application of merging
that  establishes the equivalence up to constant factors
of the gap and sco versions of the generalised hull.

\begin{proof}[Proof of \Cref{gap=sco}]
Recall that it suffices to show for any
proper convex $(k,d)$-progression $P'$ with $P'=-P'$ that $P' \sub P+Q$
for some $d$-GAP $P$ and parallelotope $Q$ with $|P+Q|=O_{d,k}(|P'|)$.

Write $P'=\phi(C\cap (\mb{R}^k\times \mb{Z}^d))$, 
where $C\subset\mb{R}^{k+d}$ is convex with $C=-C$, 
$\phi:\mb{R}^{k+d}\to\mb{R}^k$ is a linear map,
and $|P'| = \sum_{x \in \mathbb{Z}^d} |C_x|$,
where each $C_x = C \cap (\mathbb{R}^k\times \{x\})$.

\begin{clm}
We have $C \sub Q \times C'$
where $Q \sub \mb{R}^k$ is a parallelotope,
$C' \sub \mb{R}^d$ is the projection 
of $C$ on the last $d$ coordinates,
and $|Q| \card(C' \cap \mb{Z}^d) = O_{d,k}(|P'|)$.
\end{clm}

We first assume this claim and show how the lemma follows.
By \Cref{TaoVu}, the convex $d$-progression
$\phi(\{0\} \times (C' \cap \mb{Z}^d))$
is contained in a proper $d$-GAP $P$ with 
$\card P =O_{d}(\card (C' \cap \mb{Z}^d))$.
Then $P' \sub P + \phi(Q)$ with $|P+\phi(Q)|=O_{d,k}(|P'|)$,
so the lemma follows.

It remains to prove the claim. We start with $C_1=C$, $T_1=I$
then apply the following procedure for $i=1,\dots,k$.
At the start of step $i$ we have $T_i C \subset C_i$
for some volume-preserving linear transformation $T_i$
and convex $C_i \subset\mb{R}^{k+d}$ with $C_i=-C_i$ of the form
$C_i = B_i \times C'_i$ for some box $B_i \sub \mb{R}^{i-1}$
and convex $C'_i \sub V_i$, where $V_i$ is the span of $\{e_i,\dots,e_{k+d}\}$.

Let $m_i$ be maximal such that $m_i e_i \in C'_i$.
Note that as $C'_i=-C'_i$ the longest vector in $C'_i$
in direction $e_i$ is from $-m_ie_i$ to $m_ie_i$.
By \Cref{projectioncontainmentprop} there is a hyperplane 
$H_i$ in $V_i$ not containing $e_i$ 
such that $C'_i \subset \pi'_i(C'_i)+[-m_i,m_i] e_i$,
where $\pi'_i$ is the projection onto $H_i$ along $e_i$.
By \Cref{containedlinesum} we have
$|C'_i+[-m_i,m_i] e_i| = O_{d,k}(|C'_i|)$.
Let $v^i \in V_i$ be a normal vector to $H_i$ with $v^i_i=1$.
We let
\[ B_{i+1} = B_i \times [-m_i,m_i], \quad
C'_{i+1} = \pi_i(C'_i)\quad \text{and}\quad T_{i+1} = T'_i \circ T_i,\] 
where $T'_i(e_j)=e_j$ for $j \le i$ and $T'_i(e_j)=e_j + v^i_j e_i$ for $j>i$.

To analyse the procedure, we note that $T'_i(H_i) = \{x:x_i=0\} \sub V_i$
and that $T'_i$ is volume-preserving and preserves all fibres of $\pi_i$.
Thus $T_{i+1}$ is volume-preserving and $T_{i+1} C \sub T'_i C_i
= B_i \times T'_i(C'_i) \sub B_i \times ([-m_i,m_i]e_i + C'_{i+1}) = C_{i+1}$.
This shows that the procedure can be completed, ending with some
$T C \sub B \times C'$ where $B \sub \mb{R}^k$ is a box
and $C' \sub V_k$ is convex. Furthermore, 
we have $|B \times C'| = O_{d,k}(|C|)$,
so $|B| \card(C' \cap \mb{Z}^d) = O_{d,k}(|P'|)$ by \Cref{switch2}.
This completes the proof.
\end{proof}

\subsection{Reduction from $\mb{Z}$ to $\mb{R}$}

We conclude this section with one more application of merging
that establishes the reduction from $\mb{Z}$ to $\mb{R}$.

\begin{proof}[Proof of \Cref{zrkequivalence}]
Let  $A\subset \mathbb{Z}$ and $B:=A+[-\epsilon,\epsilon]\subset\mathbb{R}$ 
with $\epsilon = \epsilon(d,t,A) > 0$ sufficiently small.
Clearly $|B|=2\eps \card A$ and $|B+B|= 4\epsilon\card(A+A)$.
We also note  that $B\subset \gap^{0,d}_t(A)+[-\epsilon,\epsilon]$,
and so $|\gap^{1,d}_t(B)|\leq \epsilon \card( \gap^{0,d}_t(A) )$.
It remains to show that
$ \epsilon \card( \gap^{0,d}_t(A) ) = O_{d,t}(|\gap^{1,d}_t(B)|)$.

Write  $\gap^{1,d}_t(B)=X+\phi(P\times[-\eta,\eta])$ for some
linear map $\phi\colon\mathbb{R}^{d+1}\to\mathbb{R}$,
box $P \sub \mb{Z}^d$ and $\card X\leq t$.
For notational convenience, assume that $\phi(e_{d+1})=1$, 
so that $|\phi(P\times [-\eta,\eta])|= 2\eta \card P$. Note that 
$$2\eta \card P=|\phi(P\times [-\eta,\eta])|\leq\left|\gap^{1,d}_t(B)\right|
\leq \epsilon \card\left( \gap^{0,d}_t(A)\right)=O_{A,d,t}(\epsilon).$$

Next we apply merging to replace $P$ by some well-separated $P'$.

\begin{clm} \label{zrkclaim}
We have $\phi(P\times[-\eta,\eta])\subset \phi'(P'\times[-\eta',\eta'])$
for some $\eta'>0$, box $P' \sub \mb{Z}^{d'}$ with $P'=-P'$
and linear map $\phi'\colon\mathbb{R}^{d'+1}\to \mathbb{R}$
with $\phi'(e_{d'+1})=1$ such that
\begin{itemize}
    \item $\left|\phi'(P'\times[-\eta',\eta'])\right|\leq O_{t,d}(|\phi(P\times[-\eta,\eta])|)$, and
    \item  $|\phi(p)-\phi(p')|\geq 4t\eta' \ge 2\epsilon$ for all distinct $p,p'\in P'$.
\end{itemize}
\end{clm}

\begin{proof}[Proof of claim]
Consider the following process for $i \ge 0$ starting from $P_0:=P$, $\eta_0:=\eta$, $\phi_0:=\phi$.
\begin{itemize}
\item Suppose $\phi(P\times[-\eta,\eta]) \sub \phi_i(P_i\times[-\eta_i,\eta_i])$
for some $\eta_i>0$, box $P_i \sub \mb{Z}^{d-i}$ with $P_i=-P_i$
and linear map $\phi_i:\mb{R}^{d-i+1} \to \mb{R}$ with $\phi_i(e_{d-i+1})=1$.
\item If $|\phi(p)-\phi(p')|\geq 4t\eta_i$ for all distinct $p,p'\in P_i$ stop the process.
Otherwise, by \Cref{Pmerge} we have 
$\phi_i(P_i\times[-\eta_i,\eta_i]) \sub P^*_{i+1} + Q_{i+1}$
for some convex $(d-i-1)$-progression $P^*_{i+1}$ with $P^*_{i+1}=-P^*_{i+1}$
and interval $Q_{i+1}$ with $Q_{i+1}=-Q_{i+1}$ such that 
$\card P^*_{i+1} \cdot |Q_{i+1}|=O_{d,t}(\eta_i \card P_i)$.
\item Applying \Cref{TaoVu} to find a $(d-i-1)$-GAP containing $P^*_{i+1}$,
we have $\phi_i(P_i\times[-\eta_i,\eta_i]) \sub \phi_{i+1}(P_{i+1}\times[-\eta_{i+1}\eta_{i+1}])$
for some $\eta_{i+1}>0$, box $P_{i+1} \sub \mb{Z}^{d-i-1}$ with $P_{i+1}=-P_{i+1}$
and linear map $\phi_{i+1}:\mb{R}^{d-i} \to \mb{R}$ with $\phi_{i+1}(e_{d-i})=1$
such that $\eta_{i+1} \card P_{i+1} = O_{d,t}(\eta_i \card P_i)$.
\end{itemize}
This process terminates in at most $d$ steps with some $P' = P_i$ 
satisfying the requirements of the claim. 
It remains to show $4t\eta_i \ge 2\epsilon$.
Suppose $\eta_i < \epsilon/2t$.
Fix $a \in A$. As $X+\phi(P_i+[-\eta_i,\eta_i])$ covers $A+[-\epsilon,\epsilon]$
there is $x \in X$ with $|(a+[-\epsilon,\epsilon])\cap (x+\phi_i(P_i+[-\eta_i,\eta_i]))|\geq 2\epsilon/t$. 
Thus $\card \{ p \in P_i: \phi_i(p) \in a-x + [-\epsilon-\eta_i,\epsilon+\eta_i])  \} 
\ge \lceil \epsilon/t\eta_i \rceil \ge 2$, so there are distinct $p,p'\in P_i$ 
with $|\phi(p)-\phi(p')| \le \frac{2(\epsilon+\eta_i)}{\lceil \epsilon/t\eta_i \rceil} < 4t\eta_i$.
However, this contradicts the process terminating at $P_i$, so the claim follows.
\end{proof}

Recalling that $\eta \card P = O_{A,t,d}(\eps)$, we see that
$\eta' = O_{t,d}( \card P \cdot \eta) = O_{A,t,d}(\eps)$ and 
\[\card P' = O_{t,d}( \card P \cdot \eta/\eta') = O_{A,t,d}(\eps/\eta') = O_{A,t,d}(1).\]
Now for each $x \in X$ we note that 
$(x + \phi'(P' \times [-\eta',\eta'])) \cap \mb{Z} \sub \phi'(H_x)$, 
where $H_x := \phi'^{-1}(\mb{Z}-x) \cap (P' \times [-\eta',\eta'])$.
We can assume $H_x \ne \es$ (or we can remove $x$ from $X$),
fix $p_x \in H_x$ and note that 
$H_x - p_x \sub  \phi'^{-1}(\mb{Z}) \cap (2P'  \times [-2\eta',2\eta'])$.
The following claim will provide a hyperplane $H$ such that
$H_x - p_x \sub P^\dagger := H\cap (2P'\times[-2\eta',2\eta'])$,
which is a $d'$-GAP such that
$A \sub \bigcup_{x \in X} \phi'(H_x) \sub X' + \phi'(P^\dagger)$
where $X' = \{ \phi'(p_x): x \in X\}$.
This will complete the proof, as it implies
\begin{align*}\card \gap^{0,d}_t(A) &\leq \card (X'+\phi'(P^\dagger))\leq 
\frac{t |\phi'(P'\times[-\eta',\eta'])|}{\eta'}\\
&= O_{d,t}\left(\frac{|\phi(P\times[-\eta,\eta])|}{\epsilon}\right)
=O_{d,t}\left(\frac{\left|\gap^{1,d}_{t}(B)\right|}{\epsilon}\right).\end{align*}

\begin{clm}
There is a hyperplane containing $\phi'^{-1}(\mb{Z}) \cap (2P'  \times [-2\eta',2\eta'])$.
\end{clm}

To prove the claim, consider any 
$x,x_1,\dots,x_{d'}\in\phi'^{-1}(\mathbb{Z})\cap (2P'\times [-2\eta',2\eta'])$ 
so that the $x_i$ span a hyperplane $H$. We'll show $x \in H$.
Write $x=(y,z)$ and $x_i=(y_i,z_i)$ with $y,y_i\in 2P'$ and $z,z_i\in[-2\eta',2\eta']$. 
By Cramer's rule, we have $\lambda y=\sum \lambda _i y_i$
with $\lambda:= \det (y_1,\dots, y_{d'})$ 
and\\ $\lambda_i:=\det(y_1,\dots, y_{i-1},y,y_{i+1},\dots,y_{d'})$.
We note that $\lambda,\lambda_i$ are integers,
whose magnitudes are volumes of parallelotopes 
contained in the continuous convex hull of $2dP'$,
and so by  \Cref{fullcoprogswitch} are $O_d(\card P') = O_{A,d,t}(1)$.  

Furthermore, $\phi'(x)$ and $\phi'(x_i)$ are integers,
hence so is $\lL \phi'(x) - \sum_i \lL_i\phi'(x_i) = \lambda z-\sum \lambda_i z_i$,
where the equality holds by linearity of $\phi'$ 
using $\lambda y=\sum \lambda _i y_i$ and $\phi'(e_{d'+1})=1$.
However, we also have $z,z_i\in[-2\eta',2\eta']$,
so $|\lambda z-\sum \lambda_i z_i| = O_{A,d,t}(\eps) < 1$ 
for $\eps$ sufficiently small. We deduce
$\lambda z-\sum \lambda_i z_i=0$, so $x \in H$.
This completes the proof of the claim, and so of the proposition.
\end{proof}

\section{Snapping} \label{sec:snap}

In this section we obtain the separation property for the set $X$ of translates
with respect to the parallelotope progression $P+Q$,
thus completing the proof of our Separated Freiman Theorem.
The idea is that we think of  $x-x' \in P+Q$ for distinct $x,x'$ in $s \cdot X$
as approximate coincidences in the sense of $P+Q$ locality,
which we resolve to equalities (exact coincidences) 
by moving the elements of $X$ slightly (in the $P+Q$ sense),
analogously to the computer graphics process of snapping to a grid.

\subsection{Proof of Separated Freiman}

In this subsection we prove \Cref{SeparatedFreimanThm}
and its more applicable form  \Cref{FreimanTool},
assuming the following main lemma of this section,
which will be proved in the next subsection.
Throughout we let $G = \mb{R}^k$ for some $k \in \mb{N}$
(our bounds will be uniform in $k$).

\begin{lem} \label{BoostedadjustingX}
For any $t,\mu,s\in \mathbb{N}$ there are 
$c = c_{\ref{BoostedadjustingX}}(s,t), \gG = \gG_{\ref{BoostedadjustingX}}(t,\mu,s) \in\mathbb{N}$
such that for all $X\subset G$ with $\card X = t$ and $0\in L\subset G$ with $L=-L$
there exists $c' \le c$ and $f:X\to c^2 c!\gG^2 \cdot L/(c'c!\gG)$ with the following property:
for any $x_1,\dots,x_s,y_1,\dots,y_s\in X$ 
with $\sum x_i - \sum y_i \in c \mu^2 \gG^2 \cdot L/(c' \mu \gG)$
we have $\sum x_i+f(x_i)=\sum y_i+f(y_i)$.
\end{lem}

\begin{proof}[Proof of \Cref{SeparatedFreimanThm}]
Suppose $A\subset\mathbb{R}^k$ with $|A+A|\leq (2^{k+d+1}-\epsilon)|A|$. 
We fix parameters according to the following hierarchy,
which is to be read left to right, meaning that each successive parameter
is sufficiently large with respect to the preceding one.
\[ s_d \ll s_{d-1}\ll \dots \ll s_1 \ll  s_0 . \]
We additionally require $s_i$ to be large in terms of $\ell_i$ as specified through Hierarchy \eqref{ellieq} below and large in terms of $n_{i+1}$.
We note that the conclusion of the theorem is monotone
in the parameters $\ell_i$ and $n_i$, in that increasing
any of the parameters leads to a stronger statement.

For each $i$ we apply \Cref{SeparatedFreiman} with $(s_i,s_i)$ in place of $(s,\ell_0)$
to obtain some $A\subset X_i+P_i+Q_i$ where $Q_i$ is a parallellotope,
$P_i$ is an $s_i$-proper $n_{d_i}$-full $d_i$-GAP,
$\card X_i =O_{k,d,\epsilon,n_{d_i+1},\dots,n_{d}}(1)$,
$\card P_i \cdot |Q_i|=O_{k,d,\epsilon,s_i}(|A|)$ and
$s_i \cdot P_i$ is $s_i Q_i$-separated.
Note that in the proof of \Cref{SeparatedFreiman},
each merging operation applied with parameter $s_i$
can also be applied with parameter $s_{i-1}$, so $d_i \ge d_{i-1}$.
Thus for $i^* = \min \{ i: i \ge d_i \}$ we have $i^* = d_{i^*}$.
We write $P' = P_{i^*}$, $Q' = Q_{i^*}$, $X' = X_{i^*}$,
$d'=d_{i^*}$, $s'=s_{i^*}$,
so that $P'$ is an $s'$-proper $n_{d'}$-full $d'$-GAP,
$\card X' \le t = O_{k,d,\epsilon,n_{d'+1},\dots,n_{d}}(1)$,
$\card P' \cdot |Q'| = O_{k,d,\epsilon,s'}(|A|)
= O_{k,d,\epsilon,\ell_{d'},\dots,\ell_{d},n_{d'+1},\dots,n_{d}}(|A|)$ 
and $s' \cdot P'$ is $s' Q'$-separated.
We fix further parameters according to the hierarchy
\begin{equation}\label{ellieq}  t, \ell_{d'} \ll c \ll \mu \ll \gG \ll s' .\end{equation}

Choosing $c \ge c_{\ref{BoostedadjustingX}}(t)$, $\mu= c!$,
and $\gG \ge \gG_{\ref{BoostedadjustingX}}(t,\mu,\ell_{d'})$,
we can  apply \Cref{BoostedadjustingX} with $(X',P'+Q')$ in place of $(X,L)$,
obtain some $c' \le c$ and $f:X'\to c^2 c!\gG^2 \cdot (P'+Q')/(c'c!\gG)$ such that
for any $x_1,\dots,x_{\ell_{d'}},y_1,\dots,y_{\ell_{i^*}}\in X$ 
with $\sum x_i - \sum y_i \in c \mu^2 \gG^2 \cdot (P'+Q')/(c' \mu \gG)$ 
we have $\sum x_i+f(x_i)=\sum y_i+f(y_i)$.
We will show that the conclusion of the theorem holds for
$X := \{x'+f(x'):x' \in X'\}$, $P := 2c^2 c!\gG^2 \cdot P'/(c'c!\gG)$ and $Q = 2c^2 \gG Q'$. 

We first note that 
$\card X = \card X' = O_{k,d,\epsilon,n_{d'+1},\dots,n_{d}}(1)$ 
and $\card P \cdot |Q| = O_\gG(\card P' \cdot |Q'|)
= O_{k,d,\epsilon,\ell_{d'},\dots, \ell_d,n_{d'+1},\dots,n_{d}}(|A|)$.
Next, we have $A \sub X' + P' + Q' \sub X + P + Q$,
as for each $x' \in X'$ we have 
\begin{align*} x' + P' + Q' &\sub x' + f(x') + P' + c^2 c! \gG^2 \cdot P'/(c'c!\gG) 
+ Q' + c^2 c! \gG^2 \cdot Q'/(c'c!\gG) \\
&\sub  x' + f(x') + P + Q. \end{align*}
Furthermore, $\ell_{d'} \cdot P$ is $\ell_{d'} Q$-separated,
as if we have $x,y \in \ell_{d'} \cdot P$ with $x-y \in \ell_{d'} Q$
then $x' = c'c! \gG x$ and $y' = c'c! \gG y$ 
contradict $s' \cdot P'$ being $s'Q'$-separated,
provided that $s' \ge 2c^2 c!\gG^2$.

It remains to show for all $1\leq\ell\leq\ell_{d'}$ that 
$\ell_{d'}\cdot X$ is $\ell_{d'}\ell\cdot (P+Q)/\ell$-separated.
To see this, consider any $x_1,\dots, x_{\ell_{d'}},y_1,\dots, y_{\ell_{d'}}\in X$ 
so that $\sum x_i - \sum y_i \in \ell_0 \ell \cdot (P+Q)/\ell$.
Write $x_i = x'_i + f(x'_i)$ and $y_i = y'_i + f(y'_i)$ with $x'_i,y'_i \in X'$.
As $f(x'_i), f(y'_i) \in c^2 c!\gG^2 \cdot (P'+Q')/(c'c!\gG)$ we have
\begin{align*} \sum x'_i - \sum y'_i &\in \ell_0 \ell \cdot (P+Q)/\ell + 2\ell_{d'} c^2 c!\gG^2 \cdot (P'+Q')/(c'c!\gG)\\
&\sub  c \mu^2 \gG^2 \cdot (P'+Q')/(c' \mu \gG), \end{align*}
provided that $c! \mid \mu$.
By the above application of \Cref{BoostedadjustingX} 
this implies $\sum x'_i+f(x'_i)=\sum y'_i+f(y'_i)$,
i.e.\ $\sum x_i = \sum y_i$, as required.
\end{proof}

\begin{proof}[Proof of \Cref{FreimanTool}]
Let $k,d\in\mathbb{N}$, $\eta,\epsilon>0$, and $A\subset\mathbb{R}^k$
with $|A+A|\leq (2^{k+d+1}-\epsilon)|A|$. We choose $n_d = n_d(k,d,\eta,\eps)$ large
and apply \Cref{SeparatedFreimanThm} with, $\ell_i=4$
and fullness parameters $n_d \ll n_{d-1} \ll \cdots \ll n_0$.
This gives $A\subset X+P+Q$, 
where $P$ is a proper $n_{d'}$-full $d'$-GAP 
for some  $d'\leq d$ and $Q$ is a parallellotope, such that
$\card X + \card P \cdot  |Q|/|A| = O_{k,d,\epsilon,n_{d'+1},\dots,n_d}(1)$,
and $P+P$ is $4Q$-separated, $X+X$ is $4(P+Q)$-separated.
Writing $A_x:=A\cap (x+P+Q)$ for $x\in X$, by separation of $X+X$, 
if $(A_x+A_y)\cap (A_z+A_w)\neq \emptyset$ then $x+y=z+w$.
Also, by \Cref{BMcor}, for any $x,y \in X$ we have
\begin{align*}
    |A_x+A_y|&\geq \left(|A_x|^{1/(k+d')}+|A_y|^{1/(k+d')}\right)^{k+d'}-n_{d'}^{-1}2^{2d'+k}\card P \cdot |Q|\\
    &\geq \left(|A_x|^{1/(k+d')}+|A_y|^{1/(k+d')}\right)^{k+d'}-\eta (\card X)^{-2}|A|,
\end{align*}
as $n_d = n_d(k,d,\eta,\eps)$ is large and
$\card X + \card P \cdot  |Q|/|A| = O_{k,d,\epsilon,n_{d'+1},\dots,n_d}(1) \ll n_{d'}$.
 \end{proof}

\subsection{Proof of the main lemma}

We start by proving the following version of the main lemma of this section
with slightly weaker parameters, from which we will deduce
the main lemma via a boosting argument.

\begin{lem}  \label{adjustingX}
For any $s,t \in \mb{N}$ there exists  $c = c_{\ref{adjustingX}}(s,t)$ so that 
for any  $X\subset G$ with $\card X=t$ and $0\in L\subset G$ with $L=-L$ there exist 
$f:X\to c \cdot L/c'$ for some $c' \le c$ with the following property:
for any $x_1,\dots,x_s,y_1,\dots,y_s\in X$ with $\sum x_i - \sum y_i \in L$,
we have $\sum x_i+f(x_i)=\sum y_i+f(y_i)$. 
\end{lem}

\begin{proof}
We consider the system of linear equations 
in variables $(p_x: x \in X)$ taking values in $G$,
where for each $2s$-tuple $(x_1,\dots,x_s,y_1,\dots,y_s) \in X^{2s}$
with $\sum x_i - \sum y_i = p \in L$ we include the equation $\sum p_{x_i}-p_{y_i}=p$.
Clearly, this set of equations is consistent as $p_x=x$ is a solution. 

The following claim will provide a solution to the equations
with all $p_x\in c \cdot L/c'$ for some $c' \le c$.
Then $f(x):=-p_x$ will be as required to prove the lemma,
as for any $x_1,\dots,x_s,y_1,\dots,y_s\in X$ with $\sum x_i - \sum y_i  = p \in L$
we have $\sum p_{x_i} - \sum p_{y_i} = p$,
so $\sum x_i+f(x_i)=\sum y_i+f(y_i)$.
Thus it suffices to prove the following claim.
 
\begin{clm} 
For any $k,m,n\in\mathbb{N}$ there is $c = c_{k,m,n}\in\mathbb{N}$
such that any consistent system of linear equations in variables $x_1,\dots,x_k$,
each of the form $\sum e_i x_i= \ell$, 
where each $e_i\in Q_{m,n} := \{-m/n, -(m-1)/n,\dots,m/n\}$ and $\ell\in L$,
has a solution where $x_i\in c \cdot L/c'$ for some $c' \le c$.
\end{clm}

We prove the claim by induction on $k$ (the number of variables).
For $k=1$ we can take $c = c_{1,m,n} = \max\{m,n\}$,
as any consistent system is equivalent to  
a single equation $c'/n x=\ell \in L$ with $c'\in\{-m,\dots, m\}$,
which has the solution $x=n\ell/c'\in n\cdot L/c' \sub c \cdot L/c'$. 

For the induction step, we first note that if $x_k$ does not appear in any equation
then we can immediately reduce to $k-1$ variables by setting $x_k=0$.
Otherwise, we choose any equation that involves $x_k$,
rewrite it to express $x_k$ as a linear form in $x_1,\dots,x_{k-1}$,
and then substitute this expression in the other equations to eliminate $x_k$.
Writing $m_0/n$ for the coefficient of $x_k$ in the chosen equation,
we obtain a consistent system of equations in $x_1,\dots,x_{k-1}$ each of the form
$\sum e_i x_i= \ell$ with $\ell \in L+ Ln/m_0 \subset  (m+n)\cdot L/m_0$
and $e_i \in  Q_{m,n} +  Q_{m,n} n/m_0 \sub Q_{m(m+n),nm_0}$.

Applying the induction hypothesis with $L' = (m+n)\cdot L/m_0$ in place of $L$,
we obtain a solution with $x_1,\dots,x_{k-1} \in d \cdot L'/d'$
for some $d' \le d := c_{k-1,m(m+n),nm_0}$.
Substituting in the equation for $x_k$ we obtain 
$x_k \in (k-1)md \cdot L'/d'm_0 \sub km(m+n)d \cdot L/d'm_0^2$.
This proves the induction step with $c_{k,m,n} = km(m+n) c_{k-1,m(m+n),nm_0}$.
This concludes the proof of the claim, and so of the lemma.
\end{proof}

We will use the following lemma to boost the strength of \Cref{adjustingX}.

\begin{lem}\label{SeparationBoost}
For any $t, \lL \in\mathbb{N}$ there is 
$\gG^* = \gG^*_{\ref{SeparationBoost}}(t,\lL) := \lL^{\binom{t}{2}}$ 
so that for any $Y \sub G$ with $\card Y =t$ 
and $0\in L\subset G$ with $L=-L$ and $c \in \mb{N}$ there exists 
$\gG=\gG(Y,L,c)\leq \gG^*$ with the following property:
for any $x,y\in Y $ with $x-y \in c \lL^2 \gG^2 \cdot L/(c' \lL \gG)$ for some $c'  \le c$
we in fact have $x-y \in c\gG^2 \cdot L/(c''\gG)$ for some $c''  \le c$.
\end{lem}

\begin{proof}
For each distinct $x,y\in Y$ let $d(x,y) \ge 1$ be the smallest $\ell$ for which
$x-y \in c\ell^2 \cdot L/(c'\ell)$ for some $c' \le c$, or $\infty$ if there is no such $\ell$.
We consider the set $Y'$ of all possible $d(x,y)$ for distinct $x,y\in Y$
and order it as $Y'=\{d_1<\dots<d_\ell\}$, where $\ell \le \tbinom{t}{2}$.
If $d_1>\lL$ we can choose $\gG=1$.
If $d_\ell \le \lL^\ell$ we can choose $\gG = \gG^*$. 
Otherwise, there is some $i$ with $d_{i+1} > \lL d_i$,
so we can choose $\gG =d_i$.
\end{proof}

We now deduce the main lemma of this section.

\begin{proof}[Proof of \Cref{BoostedadjustingX}]
Suppose $X\subset G$ with $\card X = t$ and $0\in L\subset G$ with $L=-L$.
Let $Y = s \cdot X$, $t' = \card Y$, $c = c_{\ref{adjustingX}}(s,t')$
and $\gG^* = \gG^*_{\ref{SeparationBoost}}(t',\mu)$.
We apply \Cref{SeparationBoost} to find $\gG=\gG(Y,L,c) \le \gG^*$ so that
for any $x,y\in Y $ with $x-y \in c \mu^2 \gG^2 \cdot L/(c' \mu \gG)$ for some $c'  \le c$
we in fact have $x-y \in c\gG^2 \cdot L/(c''\gG)$ for some $c''  \le c$.
Now we apply \Cref{adjustingX} to $X$ with $L' := c!c\gG^2 \cdot L/(c!\gG)$ in place of $L$
to find $f:X\to c \cdot L'/c'$ for some $c' \le c$ such that
for any $x_1,\dots,x_s,y_1,\dots,y_s\in X$ wih $\sum x_i - \sum y_i \in L'$,
we have $\sum x_i+f(x_i)=\sum y_i+f(y_i)$. 

Now consider any $x_1,\dots,x_s,y_1,\dots,y_s\in X$
with $\sum x_i - \sum y_i \in c \mu^2 \gG^2 \cdot L/(c' \mu \gG)$.
Let $x = \sum x_i$ and $y = \sum x_i$.
Then $x,y \in Y$ with $x-y \in  c \mu^2 \gG^2 \cdot L/(c' \mu \gG)$,
so in fact $x-y \in c\gG^2 \cdot L/(c''\gG)$ for some $c''  \le c$.
We deduce $x-y \in L'$ and so
$\sum x_i+f(x_i)=\sum y_i+f(y_i)$ by \Cref{adjustingX}.
Thus $c'$ and $f$ have the required property.
\end{proof}

\section{Approximate structure} \label{sec:approx}

In this section we prove \Cref{linearroughstructure} (simple 1\% stability)
and the following additional approximate structure results in $\mb{R}^k$,
which imply the corresponding results in $\mb{Z}$, namely
\Cref{linearandquadraticintegers} (sharp 1\% stability),
\Cref{coveringcorintegers} (weaker non-degeneracy assumption),
\Cref{FreimanCor} (structure in an absolute constant fraction), 
and \Cref{ApproximateStructureThmIntegers} (99\% stability).

\begin{thm}\label{linearandquadratic}
For any $k,d \in \mb{N}$ and $\bB,\dD \in (0,1)$ there is $C > 0$
such that for any $A \sub \mb{R}^k$ with $|A+A| \le 2^{k+d}(1+\dD)|A|$
and $|\gap_C^{k,d-1}(A)| \ge C|A|$ there is $A' \sub A$ 
with $|\gap_1^{k,d}(A')|\leq C|A|$ and 
$$\frac{|A\setminus A'|}{|A|}\leq \min\left\{(1+\beta)\delta, \delta^2+60\delta^3\right\}.$$
\end{thm}

\begin{thm}\label{coveringcor}
For any $A\subset \mathbb{R}^k$ with $|A+A|< 2^{k+d}|A|$ and $d^*<d$ we have 
$$\min\left\{\left|\gap^{k,d^*}_{O_{k,d}(1)}(A)\right|,
\left|\gap^{k,d}_{O(2^{k+d}\exp O(9^{d-d^*}))}(A)\right|\right\}<O_{k,d}(|A|).$$
\end{thm}

\begin{thm}\label{ApproximateStructureThm}
For any $\aA>0$ there is $t=t_\aA \in \mb{N}$ such that
for any $k,d \in \mb{N}$ and $\epsilon>0$ there is $C > 0$ such that
if $A \sub \mb{R}^k$ has $|A+A|\leq 2^{k+d}(2-\epsilon)|A|$  
then there is $A'\subset A$ with $|A'|\geq (1-\alpha)|A|$ and
$$\min\left\{\left|\gap^{k,d-1}_C(A)\right|,
\left|\gap^{k,d}_t(A')\right|,
\left|\gap^{k,d+1}_1(A')\right|\right\}
\leq C(|A|).$$
\end{thm}

\begin{cor}\label{FreimanCorRk}
Suppose $A\sub\mb{R}^k$ with $|A+A|\leq 2^{k+d}(2-\epsilon)|A|$ 
and $|\gap^{k,d-1}_{O_{k,d}(1)}(A)| > \OO_{d,k,\eps}(|A|)$.
Then there is $A'\subset A$ with $|A'| \geq \frac{1}{30000}|A|$ 
and $|\gap^{k,d+1}(A')| \leq O_{d,k,\eps}(|A'|).$
\end{cor}

We start by deducing \Cref{linearandquadraticintegers} from \Cref{linearandquadratic}
(we will omit the other similar deductions).

\begin{proof}[Proof of  \Cref{linearandquadraticintegers}]
Let $\bB, \dD \in (0,1)$, $d \in \mb{N}$ and $C=C_{\bB,\dD,d}$ be sufficiently large.
Suppose $A\subset\mathbb{Z}$ with $\card(A+A)\leq 2^{d}(1+\delta)\card A$
and $\card(\gap_C^{0,d-1}(A))\geq C \card A$. 
Consider $B = A + [-\eps,\eps] \sub \mb{R}$, where $\eps>0$ is as in \Cref{zrkequivalence}.
Note that $|B|=2\eps \card A$ and $|B+B| = 4\eps \card(A+A) \le  2^{d}(1+\dD)2|B|$.
By \Cref{zrkequivalence} we also have 
$\card(\gap_C^{1,d-1}(B)) = \OO_{d,t}(\eps  \card(\gap^{0,d}_t(A))) = \OO_{d,t}(C|B|)$. 
Thus by \Cref{linearandquadratic} there is $B' \sub B$ 
with $|\gap^{1,d}(B')|\leq O_d(|B|)$ and 
$|B\setminus B'|/|B| \leq \min\left\{(1+\beta)\delta, \delta^2+60\delta^3\right\}$.

Let $A' = \{x \in A: B' \cap (x+[-\eps,\eps]) \ne \es\}$.
Clearly $\card A' \ge \card  A \cdot |B'|/|B|$,
so it remains to show $\card\left(\gap^{0,d}(A')\right)\leq O_d(\card A)$.
As $A' + [-\eps,\eps] \sub B' + [-2\eps,2\eps]$,
by \Cref{zrkequivalence} we have 
\[ \eps \cdot \card \left(\gap^{0,d}(A')\right) 
\le O_d( | \gap^{1,d}(A' +  [-\eps,\eps]) | ) 
\le O_d( | \gap^{1,d}(B') +  [-2\eps,2\eps] | ).\]
From \Cref{zrkclaim} we have
$\gap^{1,d}(B') \sub P + [-\eta,\eta]$ with $|P+[-\eta,\eta]|=O_d(|\gap^{1,d}(B')|)$
for some $d'$-GAP $P$ with $d' \le d$ and $\eta \ge \eps$.
Thus \begin{align*} | \gap^{1,d}(B') +  [-2\eps,2\eps] |
&\le |P+[-3\eta,3\eta]| \le 3|P+[-\eta,\eta]|\\
&=O_d(|\gap^{1,d}(B')|) = O_d(|B|),\end{align*}
so $\card\left(\gap^{0,d}(A')\right)\leq O_d(\card A)$.
\end{proof}

Throughout this section, we adopt the following set up obtained by applying \Cref{FreimanTool}.

\begin{setup} \label{setup} $ $
\begin{itemize}
\item $A \sub \mb{R}^k$ with $|A+A| < 2^{k+d}(2-\epsilon)|A|$,
\item $A \sub X + P + Q$ for some $d'$-GAP $P$ and parallelotope $Q$
 with $\card X + \card P \cdot  |Q|/|A| = O_{k,d,\epsilon,\eta}(1)$,
\item $A_x:=A\cap (x+P+Q)$ for $x\in X$ and $A^s := \{ x \in X: |A_x| \ge s\}$,
\item $(A+A)_z := (A+A)\cap (z+P+P+Q+Q)$
and $(A+A)^s := \{ z \in X+X: |(A+A)_z| \ge s\}$,
\item $A_0=(P+Q)\cap A$ has $0\in X$ and $|A_0|=\max_{x\in X} |A_x|$,
\item all $|A_x+A_y|\geq \left(|A_x|^{1/(k+d')}+|A_y|^{1/(k+d')}\right)^{k+d'}-\eta (\card X) ^{-2}|A|$,
\item if $(A_x+A_y)\cap (A_z+A_w)\neq \emptyset$ then $x+y=z+w$.
\end{itemize}
\end{setup}

Here we have assumed that $A_0$ is a largest translate by translating $A$.
In all proofs apart from that of \Cref{coveringcor} we can also assume that $d'=d$
by choosing $C$ larger than $\card X + \card P \cdot  |Q|/|A| = O_{k,d,\epsilon,\eta}(1)$.
In the proof  of \Cref{linearroughstructure} we will have $d'=d=0$,
so we can assume  $A \sub X+Q$ and $\eta=0$ by Brunn-Minkowski. 
To estimate the above expression for $|A_x+A_y|$, we often adopt 
the power mean notation $M_p(a,b) := \left( \tfrac12 (a^p + b^p) \right)^{1/p}$,
where $p=1/(k+d')$, and use the inequality $M_p(a,b) \ge \sqrt{ab}$.

\subsection{1\% stability}

We start by proving 1\% stability in the simplest setting,
which illustrates the argument that we will use in general.

\begin{proof}[Proof of \Cref{linearroughstructure}]
Let $A\subset\mathbb{R}^k$ with $\left|\frac{A+A}{2}\right|\leq (1+\delta)|A|$, 
where $\delta \in (0,1)$. We need to find $A'\subset A$ with $|A'|\geq (1-\delta) |A|$
and $|\co(A')|\leq O_{\delta,k}(|A|)$. With Setup \ref{setup},
we have $\delta = 1 - \epsilon$ and $d'=d=0=\eta$.
We note that if $x,y \in A^s$ then $|A_x|, |A_y| \ge s$,
so $|(A+A)_{x+y}| \ge |A_x + A_y| \ge 2^k s$ by Brunn-Minkowski. 
Thus $A^s + A^s \sub (A+A)^{2^k s}$.
We also note that 
$|A| = \sum_x |A_x| = \int_s 1_{0 \le s \le |A_x|} =  \int_s \card A^s$.
Similarly,  $|A+A| = \sum_{z \in X+X} |(A+A)_z|
 = \int_s \card (A+A)^s =  2^k \int_s \card (A+A)^{2^k s}$,
using disjointness of $(A+A)_z$ for distinct $z \in X+X$.
As $|A_0| = \max \{ s: \card A^s > 0\}$ we deduce 
\begin{align*} |A+A|&= 2^k \int_s \card (A+A)^{2^k s} \ge 2^k \int_s \card (A^s + A^s)\\
&\ge 2^k \int_s (2 \card A^s - 1)  = 2^k (2|A|-|A_0|).\end{align*}
Thus $|A \sm A_0| \le 2^{-k}|A+A|-|A| \le \dD |A|$.
As $|\co(A_0)| \le |Q| = O_{k,\epsilon}(|A|)$ the theorem holds for $A'=A_0$.
\end{proof}

Next we prove our 1\% stability result that is asymptotically 
sharp as $\dD \to 0$ or $\dD \to 1$.

\begin{proof}[Proof of \Cref{linearandquadratic}]
The argument is similar to the previous proof (that of \Cref{linearroughstructure}),
but now with $d'=d \ge 0$ and $\eta>0$ sufficiently small.
With Setup \ref{setup},
for any $x,y \in A^s$ (i.e.\  $|A_x|, |A_y| \ge s$) that 
$|A_x+A_y|\geq 2^{k+d} s - \tT$, where $\tT := \eta (\card X) ^{-2}|A|$,
so $A^s + A^s \sub (A+A)^{2^{k+d} s - \tT}$.
We note that $2^{k+d} \int_{s \ge 0} \card (A+A)^{2^{k+d} s - \tT}
= \int_{s \ge 0} \card (A+A)^{s - 2^{-k-d} \tT}
= 2^{-k-d} \tT \card(X+X) + \int_{s \ge 0} \card (A+A)^s
\le \eta |A| + |A+A|$. Thus
\begin{equation} \label{eq:2x-1}
 |A+A| + \eta |A| \ge 2^{k+d}  \int_s \card (A^s + A^s) 
\ge 2^{k+d} \int_s (2 \card A^s - 1)  = 2^{k+d} (2|A|-|A_0|).
\end{equation}
We deduce $|A \sm A_0| \le 2^{-k-d}|A+A| + \eta|A| - |A| \le (\dD + \eta) |A|$.
Taking $\eta < \bB \dD$ gives the linear bound in the theorem.
Now we will show the quadratic bound. We note that
\begin{align}\label{Mtosqrteq} 2^{-(d+k)}|A+A| &\ge  2^{-(d+k)} \sum_{x\in X} |A_0+A_x|
\ge  \sum_{x\in X} M_{1/(d+k)}(|A_0|,|A_x|) - \eta |A|\\
&\ge  \sum_{x\in X} \sqrt{|A_0||A_x|} - \eta |A|.\end{align}
For given $|A_0|$, by concavity this last sum is minimized when $|A_x|$
is $|A|-|A_0|$ for some $x$ and $0$ otherwise. Writing $\aA := 1 - |A_0|/|A|$ we obtain
\[ \delta + \eta \ge  2^{-(d+k)}|A+A|/|A| - 1 \ge - \aA + \sqrt{\aA(1-\aA)}  
= (\sqrt{1-\aA}-\sqrt{\aA})\sqrt{\aA}. \]
We deduce 
$(1-2\sqrt{(1-\alpha)\alpha})\alpha\leq (\delta+\eta)^2$.

From here, some calculations will give the required bound
$\alpha\leq \delta^2+60\delta^3$ for $\eta$ sufficiently small.
Indeed, first note that we can assume $\delta^2+60\delta^3<1$,
and so $\dD < 1/4$ (otherwise the trivial bound $\aA \le 1$ suffices).
From the linear bound $\aA < \dD + \eta$ we calculate
$(1-2\sqrt{(1-\alpha)\alpha})^{-1} < 7.5 < 8$ for $\eta$ small,
so $\alpha \le 8\delta^2$. Substituting this bound we obtain
$\aA \le (\delta+\eta)^2 + 2 \aA^{3/2} <  \delta^2+60\delta^3$ 
for $\eta$ small, as required.
\end{proof}

\begin{rmk}\label{sharplinearquadraticremark}
In \Cref{Mtosqrteq}, we use the bound $M_{1/(d+k)}(x,y)\geq \sqrt{xy}$ to get a clean result. As $M_{1/(d+k)}$ is also concave, the rest of the argument does not depend on this simplification and working with $M_{1/(d+k)}$ would give the optimal result for all $d,k$ as mentioned after \Cref{linearandquadraticintegers}.
\end{rmk}

Now we adapt the above argument to extract structure
under a weaker non-degeneracy assumption.

\begin{proof}[Proof of \Cref{coveringcor}]
Let $A\sub \mb{R}^k$ with $|A+A|< 2^{k+d}|A|$ and $d^*<d$.
With Setup \ref{setup}, we have
 $\eps=1$ and $\eta<\eta(k,d)$ sufficiently small, 
and now we may have $d'<d$. We need to show
$$\min\left\{\left|\gap^{k,d^*}_{O_{k,d}(1)}(A)\right|,
\left|\gap^{k,d}_{O(2^{k+d}\exp O(9^{d-d^*}))}(A)\right|\right\}<O_{k,d}(|A|).$$
As $\card X + \card P \cdot  |Q|/|A| = O_{k,d,\eta}(1)$,
we are done if $d' \le d^*$, so we can assume $d'>d^*$.

\begin{clm}
There is some $s_0$ so that $\card(A^{s_0}+A^{s_0})\leq 2^{d-d'+1/2}\card A^{s_0}$ 
and $|A\cap (A^{s_0}+P+Q)|\geq \frac14 |A|$.
\end{clm}
\begin{proof}[Proof of claim]
As in \eqref{eq:2x-1} we have 
$|A+A| + \eta |A| \ge 2^{k+d'}  \int_s \card (A^s + A^s)$.

Let $s_0$ be minimal so that $\card(A^{s_0}+A^{s_0})\leq 2^{d-d'+1/2} \card A^{s_0}$.
We note that \[\int_{s=0}^{s_0} \card A^s 
=  \int_{s=0}^{s_0} \sum_{x \in X} 1_{|A_x| \ge s} 
\ge |A| - |A\cap(A^{s_0}+P+ Q)|,\]
so $ \int_s \card (A^s + A^s) \ge \int_{0}^{s_0} 2^{d-d'+1/2}\card A^s 
\ge 2^{d-d'+1/2} (|A| - |A\cap (A^{s_0} +P + Q)|)$. Therefore
\[ 2^{k+d}|A| > |A+A| \ge 2^{k+d+1/2} (|A| - |A\cap (A^{s_0} + P + Q)|) - \eta |A| \]
For small $\eta$ we deduce $|A\cap A^{s_0}+P+Q|\geq \frac14 |A|$.
\end{proof}

Now we apply \Cref{greentaooriginal} (Quantitative Freiman) to $A^{s_0}$,
obtaining $A^{s_0} \sub X' + P'$ with $\card X' \le \exp O\left(9^{d-d'}\right)$ 
and some $(d-d')$-GAP $P'$ with $\card P' \le \card A^{s_0}$.
We note that $P^* := P+P'$ is a $d$-GAP.
By the pigeonhole principle we can fix some $x \in X'$ with 
\[|A \cap (x+P^*+Q)| \ge |X'|^{-1} |A \cap (A^{s_0}+P+Q)| \ge |A|/ \exp O\left(9^{d-d'}\right).\]
Applying the Ruzsa Covering Lemma (\Cref{ruzsa}),
we can cover $A$ by $O\left(2^{k+d} \exp O\left(9^{d-d'}\right)\right)$
translates of $(P^* + Q) - (P^* + Q)$, as required.
\end{proof}

\subsection{99\% stability}

Here we will prove our 99\% stability result, via the following proposition,
which also implies the absolute constant result, as we will see that we can take
$t^*_{\ref{moststructurethm}}(\gG,\bB) 
:= \frac{252}{\beta(\gamma-\beta)}+\frac{2520}{(\gamma-\beta)^2}+\frac{21}{\beta}$.

\begin{prop}\label{moststructurethm}
For any $\gG>\bB>0$ there is $t^* = t^*_{\ref{moststructurethm}}(\gG,\bB) \in\mathbb{N}$ 
such that for any $ d,k\in\mathbb{N}$ and $\epsilon>0$
there are $C=C_{\epsilon,d,k}$ and $\eta=\eta_{\epsilon,d,k}$ so that the following holds.
Let $A\subset \mathbb{R}^{k+d}$ with $|A+A|< 2^{k+d}(2-\epsilon)|A|$ 
and $\left|\gap^{k,d-1}_C(A)\right|\geq C|A|$. 
With Setup \ref{setup}, let $t$ be minimal so that there exists $X^\dagger\subset X$ 
with $\sum_{x\in X\setminus X^\dagger}|A_x|\leq \beta |A|$ and $\card X^\dagger =t$. 
If $t\geq t^*$ then there exists  $X'\subset X$ with $\sum_{x\in X\setminus X'}|A_x|\leq \gamma |A|$ 
and an arithmetic progression $T$ with $\card T \leq 3\card X' $ so that $X'\subset T+P+Q$.
\end{prop}

First we assume \Cref{moststructurethm} and 
deduce \Cref{ApproximateStructureThm} and \Cref{FreimanCorRk}.

\begin{proof}[Proof of \Cref{ApproximateStructureThm}]
Let $t_\aA = t^*_{\ref{moststructurethm}}(\aA,\aA/2)$
and find $t$, $X^\dagger$ as in \Cref{moststructurethm}. 
If $t\leq t_\alpha$ then $A':=\bigcup_{x\in X^\dagger}A_x$
has $|A \sm A'| \le \aA |A|$ and 
$|\gap^{k,d}_{t_\alpha}(A')| \le \card X^\dagger \cdot \card P \cdot |Q| = O_{\eps,d,k}(|A|)$.
Alternatively, if $t>t^*$ then find $X'$, $T$ as in \Cref{moststructurethm}
and let $A':=\bigcup_{x\in X'}A_x$. Then $|A \sm A'| \le \aA |A|$ and
$|\gap^{k,d+1}(A')|\leq |T+P+Q|\leq 3\card X \cdot\card P \cdot|Q| = O_{\eps,d,k}( |A|)$.
\end{proof}

\begin{proof}[Proof of  \Cref{FreimanCorRk}]
Suppose $A\sub\mb{R}^k$ with $|A+A|\leq 2^{k+d}(2-\epsilon)|A|$ 
and $\left|\gap^{k,d-1}_{O_{k,d,\epsilon}(1)}(A)\right| > \OO_{d,k,\eps}(|A|)$.
We apply \Cref{moststructurethm} with $\gG=0.5$ and $\bB=0.1$,
for which we calculate $t^*(0.5,0.1) < 25000$.
If $t \ge t^*$ we obtain $A' \sub A$ with $|A'| \ge |A|/2$ and $A' \sub T + P + Q$ 
for some arithmetic progression $T$ with $\card T \leq 3\card X' $,
and  so $|\gap^{k,d+1}(A')| \leq O_{d,k,\eps}(|A'|)$.
If $t < t^*$ we obtain $A^* \sub A$ with $|A^*| \ge 0.9|A|$
and $A^* \sub X^\dagger + P + Q$ with $\card X^\dagger < t^*$.
Then some translate $A' = A \cap (x+P+Q)$ has $|A'| \ge \frac{1}{30000}|A|$
and $|\gap^{k,d+1}(A')| \le |\gap^{k,d}(A')| \le |P+Q| = O_{d,k,\eps}(|A'|)$.
\end{proof}

It remains to prove the proposition. The idea of the proof is to find a large subset of $X$ that has small doubling and is thus close to an AP, which has the additional property of containing all the large $A_x$'s. To control the amount of $A$ outside this AP, we note that most of these points contribute a set of size at least $2^{k+d}\sqrt{|A_x||A_0|}$ to the doubling. Since $|A_x|$ is much smaller than $|A_0|$ this contribution is very large, so that the total weight in all these $A_x$'s must be small.

\begin{proof}[Proof of \Cref{moststructurethm}]
Suppose $A \sub \mb{R}^k$ has $|A+A|\leq 2^{k+d}(2-\eps)|A|$.

Let $s_0:=\min\{s: \card(A^s+A^s)\leq 3\card A^s -4\}$
or $s_0 = |A_0| = \max_{x \in X} |A_x|$ if there is no such $s$.

By Freiman's $3k-4$ theorem \cite{freiman1959addition} 
there is an arithmetic progression $T\subset \mathbb{R}^k$
with $A^{s_0}\subset T$ and $\card T \leq 2\card A^{s_0}$. 
We set $X' = A^{s_0}$ and $A' = (X'+P+Q)\cap A$.

We assume $\card X^\dagger = t \ge t^*$ 
and need to show that $|A \sm A'| \le \gG |A|$.
We partition $X$ into $(X',X^+,X^-)$,
where $X^+ = X^\dagger \sm X'$ and $X^- = X \sm X^\dagger$.
We write $s_1 := \min \{|A_x|: x \in X^\dagger \}$. Thus
\[ |A_x| \ge s_0 \text{ for } x \in X', 
|A_x| \in [s_1,s_0)  \text{ for } x \in X^+, \text{ and }
|A_x| \le s_1 \text{ for } x \in X^-. \]
We denote the corresponding partition of $A$ by $(A',A^+,A^-)$.
By definition of $X^\dagger$ we have 
\begin{equation} \label{eq:A-}
\card X^++\card X'=\card X^\dagger=t \ge t^*
\quad \text{ and } \quad
 \bB |A| - s_1 \le |A^-| \le \bB |A|.
\end{equation}
It remains to show that $|A^+| \le (\gG-\bB)|A|$.
We start by calculating
\begin{align*}
\int_{s=0}^{s_0} \card (A^s + A^s) 
& \ge \int_{s=0}^{s_0} (3\card A^s - 3)
= 3 \left(\int_{s=0}^{s_0}\card A^s\right)  - 3s_0, 
\text{ and } \\
\int_{s>s_0} \card (A^s + A^s) 
& \ge \int_{s=s_0}^{|A_0|} (2\card A^s - 1)
= 2|A| - 2\left(\int_{s=0}^{s_0}\card A^s\right)  - |A_0| + s_0.
\end{align*}
Summing, using $|A+A| + \eta |A| \ge 2^{k+d}  \int_s \card (A^s + A^s)$
as in \eqref{eq:2x-1} and taking $\eta$ small, gives
\[  2|A| + \int_{s=0}^{s_0}\card A^s - |A_0| - 2 s_0  
\le \int_s  \card (A^s + A^s) \le 2^{-k-d}(|A+A| + \eta |A|) < 2|A|. \]
Thus $\int_{s=0}^{s_0}\card A^s \le |A_0| + 2 s_0$. 
As $\int_{s=0}^{s_0} \card A^s 
=  \int_{s=0}^{s_0} \sum_{x \in X} 1_{|A_x| \ge s} 
= |A \sm A'| + s_0 \card A^{s_0}$ we deduce
\begin{equation} \label{eq:A'}
|A \sm A'|  \le |A_0| + 2 s_0 - s_0 \card A^{s_0} \le 3|A_0|.
\end{equation}
Next we claim that
\begin{equation} \label{eq:Xcirc}
X^\circ:=\left\{x\in X\setminus X':(A_x+A_0)\cap (A'+A')\neq\emptyset\right\} 
\text{ satisfies } \card X^\circ \leq 3\card X'. 
\end{equation}
Indeed, we can assume that  $X+X$ is $4\cdot(P+Q)$ separated,
so for any $x \in X^\circ$ we have $x=x+0 \in X' + X' \sub T+T$,
giving $\card X^\circ \leq \card((T+T)\setminus X')\leq 2\card T-\card X' \leq 3\card X'$.
Now we will show that
\begin{equation} \label{eq:sumsqrt}
\sum_{x\in X\sm X^\circ}\sqrt{|A_0||A_x|}\leq 8|A_0|.
\end{equation}
To see this, we first note that as in \eqref{eq:2x-1}, using \eqref{eq:A'}, we have
\begin{align*} 2^{-k-d} (|A'+A'| + \eta |A|) &\ge \int_s \card (A'^s + A'^s)
\ge \int_s (2 \card A'^s - 1)\\
&= 2|A'|-|A_0| \ge 2|A| - 7|A_0|.\end{align*}
We also note that $A+A$ contains disjoint sets $A_x + A_0$ for all $X\sm X^\circ$
that are disjoint from $A'+A'$, satisfying 
$2^{-k-d}(|A_x + A_0| + \eta  (\card X) ^{-2}|A|) \ge M_{1/(k+d)}(|A_x|,|A_0|) \ge \sqrt{|A_x||A_0|}$.
Taking $\eta$ small and using $|A+A| < 2^{k+d+1}|A|$ we deduce \eqref{eq:sumsqrt}.
Now, recalling that our ultimate goal is a bound for $|A^+|$,
our next target is the following bounds for $\card X^+$ and $\card X'$:
\begin{equation} \label{eq:X+-}
\card X^+ \leq 6\card X' + 2 + 16/\bB
\quad \text{ and } \quad
\card X' \ge t^*/7 - 3/\bB.
\end{equation}
To show these bounds, we consider separately the contributions
to the left hand side of \eqref{eq:sumsqrt} from $X^-$ and $X^+$.
By concavity, the contribution from $X^-$ given $|A^-|$ 
is minimised by making each $|A_x|$ zero or $s_1$ (the maximum possible).
Using \eqref{eq:A-} and \eqref{eq:Xcirc} we deduce
\begin{align*}\sum_{x\in X^-\setminus X^\circ}\sqrt{|A_0||A_x|}&\geq \left(\left\lfloor\frac{\beta|A|}{s_1}\right\rfloor -1-\card X^\circ \right) \sqrt{|A_0|s_1}\\
&\geq \left(\frac{\beta|A|}{s_1} -2-3\card X' \right) \sqrt{|A_0|s_1}.\end{align*}
Using $|A_x|\geq s_1$ for all $x \in X^+$ and \eqref{eq:sumsqrt} we find

$$\left(\beta|A|/s_1+\card X^+ -2-6\card X' \right) \sqrt{s_1|A_0|}
\leq \sum_{x\in X\setminus X^\circ}\sqrt{|A_x||A_0|}\leq 8|A_0|.$$
Using  the AM-GM inequality
$$\beta|A|/s_1+\card X^+ -2-6\card X' 
\ge 2 \sqrt{ (\card X^+ -2-6\card X' ) \beta|A|/s_1}$$ we get
\[  2 \sqrt{ (\card X^+ -2-6\card X' ) \beta|A|} \le 8 \sqrt{|A_0|}.  \]
This implies the bound on $\card X^+$ in \eqref{eq:X+-}.
The bound on $\card X'$ then follows using $\card X^+ +\card X' \geq t^*$.

To deduce a bound for $|A^+|$ we also need the following bounds on $s_0$.
\begin{equation} \label{eq:s0}
s_0\leq \frac{5 |A|^2}{(\card X')^2|A_0|}
\quad \text{ and } \quad
s_0 \le |A|/\card X'.
\end{equation}
To see these, we recall that $|A_x| \ge s_0$ for all $x \in X'$ and calculate
$$2^{d+k+1}|A|\geq |A+A|\geq \sum_{x\in X'} |A_x+A_0|
\geq 2^{d+k}\card X' \sqrt{|A_0|s_0}-\eta|A|. $$
This implies the first bound for small $\eta$. The second holds as\\
$|A|\geq \sum_{x\in X'}|A_x|\geq s_0 \card X' $.

Combining \eqref{eq:X+-} and \eqref{eq:s0}, using $|A_x| \le s_0$ for all $x \in X^+$, we obtain
\begin{align*}|A^+|&\leq s_0\card X^+ \leq \left(\frac{16}{\beta}+2\right)s_0+6\card X' \frac{5 |A|^2}{(\card X')^2|A_0|}\\
&\le \left(\frac{16}{\beta}+2\right)\frac{|A|}{\card X' }+ \frac{ 30|A|^2}{\card X'\cdot |A_0|}.\end{align*}
Finally, we recall $|A^+|\leq |A\setminus A'|\leq 3|A_0|$ by \eqref{eq:A'},
so by the inequality $\min\{ x + y, z\} \le  x + \sqrt{yz}$ we obtain
\begin{align*}|A^+|&\leq \left(\frac{16}{\beta}+2\right)\frac{|A|}{\card X' }+\sqrt{3|A_0|\cdot  \frac{ 30|A|^2}{\card X'\cdot |A_0|}}\\
&=\left[\left(\frac{16}{\beta}+2\right)\frac{1}{\card X' }+\sqrt{\frac{90 }{\card X' }}\right]|A|.\end{align*}
Substituting $\card X' \geq t^*/7-3/\beta$ from \eqref{eq:X+-} and taking
\[ t^* \ge t^*_{\ref{moststructurethm}}(\gG,\bB) 
:= \frac{252}{\beta(\gamma-\beta)}+\frac{2520}{(\gamma-\beta)^2}+\frac{21}{\beta} \] 
we obtain the required bound $|A^+| \le (\gG-\bB)|A|$.
\end{proof}

\section{Exact results I: basic regime} \label{sec:exactI}

We will now prove our exact (100\% stability) results
in the basic regime where an extremal or asymptotically extremal example 
is given by a single large generalised convex progression
together with a few scattered points;
the regime where a discrete cone is asymptotically extremal
will be considered in a later section.
In particular, our second result below establishes
the doubling threshold for approximate convexity.
As usual, we work in $\mb{R}^k$ rather than $\mb{Z}$,
with the following goals for this section.

\begin{thm}\label{FewLocationsI}
Suppose $A\subset\mathbb{R}^k$ with
$|A+A|\leq (2^{k+d}+\ell)|A|$ for $\ell\in (0,2^{k+d})$
and $|\gap_{O_{k,d,\ell}(1)}^{k,d-1}(A)|\geq \OO_{k,d,\ell}(|A|)$. 
Then $\left|\gap^{k,d}_{\ell'}(A)\right|\leq O_{k,d,\ell}(|A|)$, where
$$\ell'\leq\begin{cases}
1 & \text{ if } \ell \in (0,1), \\
\ell+1 &\text{ for } \ell \in \mb{N} \text{ if } \ell\leq 0.1 \cdot 2^{k+d} \\ &\text{ or if } \ell\leq 0.315\cdot 2^{k+d} \text{ and } k+d\geq 13,\\
\ell \left(1+O\left( \sqrt[3]{\frac{2^{k+d}}{(2^{k+d}-\ell)^2}}\right)\right) &\text{ if } 0.1\cdot 2^{k+d}\leq \ell\leq \left(1-\frac{1}{\sqrt{2^{k+d}}}\right)2^{k+d}.
\end{cases}$$
\end{thm}

\begin{thm}\label{2^k+dstability}
For any $k,d\in\mathbb{N}$,  $\gG, \epsilon>0$, $\delta \in [0, 1-\epsilon)$,
if $A \sub \mb{R}^k$ with  $|A+A| \le (2^{k+d}+\delta)|A|$ and
$|\gap^{k,d-1}_{O_{k,d,\gG,\eps}(1)}(A)| = \OO_{k,d}(|A|)$
then $|\co^{k,d}(A) \sm A|/|A| \le O_{k,d}(\gG+\dD)$.
\end{thm}

\begin{proof}[Proof of \Cref{FewLocationsI}]
We adopt Setup \ref{setup} with $\eps = 1 - \ell/2^{k+d}$, 
obtained by applying \Cref{FreimanTool} with $\eta$ sufficiently small. As $\ell\leq \left(1-\frac{1}{\sqrt{2^{k+d}}}\right)2^{k+d}$, we find that $\epsilon=\Omega_{k,d}(1)$, so the application of \Cref{FreimanTool} does not depend on $\ell$.
By \Cref{linearandquadratic} we have $|A_0| = \max_{x \in X} |A_x| \geq (1-\eta)\eps |A|$.
We calculate
\begin{equation} \label{eq:calc2} 
 (2^{k+d}+\ell)|A| \ge |A+A| \ge \sum_{x \in X} |A_x + A_0| 
\ge - \eta |A| + 2^{k+d} \sum_{x \in X} M_{1/(k+d)}(|A_x|,|A_0|). 
\end{equation} 

By concavity, noting that $\sum_{0\neq x\in X}|A_x|=|A|-|A_0|$ and all $|A_x|\leq |A_0|$, 
the right hand side of \eqref{eq:calc2} is minimised when $|A_x|$ is non-zero
for $m$ values equal to $|A_0|$ and one value equal to $|A|-m|A_0| < |A_0|$
(this defines $m \in \mb{N}$, i.e. $m=\lfloor |A|/|A_0|\rfloor$).
As $M_{1/(k+d)}(|A_0|,|A_0|) = |A_0|$, $M_{1/(k+d)}(|A_0|,0) = 2^{-k-d} |A_0|$
and $M_{1/(k+d)}(x,y) \ge \min\{x,y\}$ we deduce
\begin{align} \label{eq:calc3} 
(2^{k+d}+\ell+\eta)|A| 
& \ge 2^{k+d}(m|A_0| + M_{1/(k+d)}(|A_0|,|A|-m|A_0|)  ) + (\card X - m-1)  |A_0| \\
& \ge 2^{k+d}|A| + (\card X - m-1)  |A_0|. \nonumber
\end{align} 

Taking $\eta < \eta(k,d,\eps)$ small we obtain the preliminary bound
\begin{equation} \label{eq:X} 
\card X \leq (\ell+\eta)|A|/|A_0|+m+1\leq \eps^{-1}\ell + m + 2.
\end{equation} 

We now consider conditions under which
we can prove sharp bounds on $\card X$.

\begin{clm}
If $\ell \in (0,1)$ then $\card X = 1$.
\end{clm}

To see this, we first note that $\eps > 1 - 2^{-k-d}$, so $m=1$
and the first bound in \eqref{eq:X} gives $\card X \le 3$ for small $\eta$.
Next we show that $\card X = 2$ is impossible.
Indeed, if $X=\{x,y\}$ then
\begin{align*} (2^{k+d} + \ell)|A| &\ge |A+A| \ge |A_x + A_x| + |A_y + A_y| + |A_x + A_y|\\
&\ge 2^{k+d}(|A_x| + |A_y|) + |A| - \eta |A|,  \end{align*}
using $(|A_x|^{1/(k+d)} + |A_y|^{1/(k+d)})^{k+d} \ge |A|$.
However, this is a contradiction for $\eta < 1-\ell$. 
Similarly, we will show that $\card X = 3$ is impossible.
Indeed, suppose $X=\{x,y,z\}$. If there is any non-trivial identity
$a+b=c+d$ with $a,b,c,d \in X$ then this is unique, 
and we relabel so that it is $y+y=x+z$. We then calculate
\begin{align*}
(2^{k+d} + \ell)|A| & \ge |A+A|\\
&\ge |A_x + A_x| + |A_y + A_y| + |A_z + A_z|
+  |A_x + A_y| +  |A_x + A_z| \\
& \ge 2^{k+d} |A| + (|A_x| + |A_y|) + (|A_x| + |A_z|)  - \eta |A|,  
\end{align*}
which again gives  a contradiction for $\eta < 1-\ell$. 
Thus $\card X = 1$.

Now we assume $\ell \in \mb{N}$ with
$1 \le \ell\leq 0.315\cdot 2^{k+d}$, i.e.\ $\eps \ge 0.685$,
so $|A_0| > 0.68 |A|$ and $m=1$. 

If $k+d<13$ we assume $\ell\leq 0.1\cdot 2^{k+d}$,
and note that $\ell \ge 1$ implies $k+d \ge 4$. Let 
\[ \aA := 1 - |A_0|/|A| 
\quad \text{ and } \quad
\aA^* := \left(\frac{0.9(k+d)}{2^{k+d+1}}\right)^{(k+d)/(k+d-1)}. \]
Note that $\alpha<1-(1-\eta)\ell 2^{-(k+d)}$ and $\alpha^*<0.1$.

\begin{clm}
If $\aA < \aA^*$ then $\card X\leq  \ell+1$.
\end{clm}

From \eqref{eq:calc3}, 
as $2^{k+d} M_{1/(k+d)}(1-\aA,\aA) \ge (1-\aA)(1+(k+d)\aA^{1/(k+d)})$
and $\card X - 1 \le 2^{k+d}$ from \eqref{eq:X},
\begin{align*}
2^{k+d}+\ell+\eta
& \ge 2^{k+d}(1-\aA + M_{1/(k+d)}(1-\aA,\aA)  ) + (\card X - 2)  (1-\aA) \\
& \ge 2^{k+d}+\card X - 2 - (2^{k+d} + \card X - 2) \aA + (1-\aA)(1+(k+d)\aA^{1/(k+d)}) \\
& \ge 2^{k+d}+\card X - 1 - 2^{k+d+1} \aA + 0.9(k+d)\aA^{1/(k+d)} \\
&\geq 2^{k+d}+\card X -1 \quad \text{(using $\aA < \aA^*$)}.
\end{align*}
Thus $\card X\leq \lfloor\ell+\eta\rfloor +1\leq \ell+1$ for small $\eta$.
Now we can assume $\aA \ge \aA^*$.

\begin{clm}
If $\ell<0.1\cdot 2^{k+d}$ then $\card X\leq  \ell+1$.
\end{clm}

Calculating as above, now using $M_{1/(k+d)}(1-\aA,\aA) \ge \sqrt{\aA(1-\aA)}$
and $\card X - 2 \leq 0.2\cdot 2^{k+d}$ from  \eqref{eq:X},
\begin{align*}
2^{k+d}+\ell+\eta
& \ge 2^{k+d}+\card X - 2 - (2^{k+d} + \card X - 2) \aA + 2^{k+d}  \sqrt{\aA(1-\aA)}\\
& \ge 2^{k+d}+\card X - 2 + 2^{k+d}\sqrt{\aA} ( \sqrt{1-\aA} - 1.2  \sqrt{\aA} ) \\
&\geq 2^{k+d}+\card X -1,
\end{align*}
where the last inequality holds as
$\sqrt{\aA} ( \sqrt{1-\aA} - 1.2  \sqrt{\aA} ) \ge \sqrt{\aA} ( \sqrt{1-0.1} - 1.2  \sqrt{0.1} ) > 2^{-k-d}$
when $\aA \ge \aA^*$ and $k+d \ge 4$.
As above this implies $\card X\leq \lfloor\ell+\eta\rfloor +1\leq \ell+1$ for small $\eta$.

Thus we come to our final case for proving the sharp bound $\card X \le \ell + 1$.

\begin{clm}
If $0.1 \le \ell/2^{k+d} \le 0.315$, $\aA \ge \aA^*$ and $k+d\geq 13$ then $\card X\leq  \ell+1$.
\end{clm}

Calculating as above, now using 
$\card X - 2 \leq \frac{0.315}{1-0.315}\cdot 2^{k+d} \le 0.46 \cdot 2^{k+d}$ from  \eqref{eq:X},
\begin{align*}
2^{k+d}+\ell+\eta
& \ge 2^{k+d}+\card X - 2 - (2^{k+d} + \card X - 2) \aA + 2^{k+d}  \sqrt{\aA(1-\aA)}\\
& \ge 2^{k+d}+\card X - 2 + 2^{k+d}\sqrt{\aA} ( \sqrt{1-\aA} - 1.46  \sqrt{\aA} ) \\
&\geq 2^{k+d}+\card X -1,
\end{align*}
where the last inequality holds as
$\sqrt{\aA} ( \sqrt{1-\aA} - 1.46  \sqrt{\aA} ) \ge \sqrt{\aA} ( \sqrt{1-0.315} - 1.46  \sqrt{0.315} ) > 2^{-k-d}$
when $\aA \ge \aA^*$ and $k+d \ge 13$.
As above this implies $\card X\leq \lfloor\ell+\eta\rfloor +1\leq \ell+1$ for small $\eta$.

We conclude by proving the asymptotic bound on the number of translates.

\begin{clm} 
If $0.1\cdot2^{k+d}\leq \ell\leq \left(1-\frac{1}{\sqrt{2^{k+d}}}\right)2^{k+d}$
then $\card X \le \ell (1+O(\gG))$ with $\gG := \frac12 \sqrt[3]{2^{-(k+d)} \epsilon^{-2}}$.
\end{clm}

We note that $\epsilon:=1-\ell/2^{k+d} \in \left(\frac{1}{\sqrt{2^{k+d}}},0.9\right)$.
We order $X$ as $x_1,\dots, x_{\card X}$ in non-increasing order
of $|A_i|$, where $A_i:=(x_i+P+Q)\cap A$.
We let  
$ t=\max\left\{t: \sum_{i=1}^t |A_i|\leq (1-\gamma)|A|\right\}$
and $\xi = |A_{t+1}|/|A|$. By \eqref{eq:X} we find
\begin{equation}\label{gammaboundeq} \gG \le \xi(\card X - t) \le \xi  \eps^{-1}(\ell+2) \le  \xi  \eps^{-1} 2^{k+d} . \end{equation}

Note that $\xi \le \gG \le 1/2$.
If $t=0$ then $|A_1|> (1-\gamma)|A|$, 
so as $(2^{k+d}+\card X-1)|A_1|\leq |A+A|+\eta |A|\leq (2^{k+d}+\ell+\eta)|A|$ 
we have $\card X \leq \ell(1+O(\gamma))$. Thus we can assume $t \ge 1$.

We write $A_{\leq t}=\bigcup_{i\leq t}A_i$ 
and note that $|A_{\leq t}|/|A| \ge 1-\gG - \xi \ge 1-2\gG$.
Next we bound  $|A+A|$:
\begin{equation} \label{eq:A+A}
 |A+A|+\eta|A|\geq \max\left\{ 
 2^{k+d}t\sqrt{\xi(\eps-\eta)}|A|,\ 
 2^{k+d+1}|A_{\leq t}|+(\card X-2t-2^{k+d})|A_1| \right\}.
\end{equation}
For the first bound, using $|A_1|\geq (\eps-\eta)|A|$ by \Cref{linearandquadratic} and $|A_i|\geq \xi|A|$ 
for all $i\leq t$ by definition of $\xi$,
$$|A+A|\geq \sum_{i=1}^t |A_i+A_1|
\geq 2^{k+d}t\sqrt{|A_i||A_1|}-\eta |A|\geq 2^{k+d}t\sqrt{\xi(\eps-\eta)}|A|-\eta |A|.$$
For the second bound note that $|X\setminus \{x_1+x_1,x_1+x_2,x_2+x_2,\dots, x_t+x_t\}|\geq \card X-2t$.
Using $|A_i+A_{i+1}|\geq 2^{k+d}|A_{i+1}|-(\card X)^{-2}\eta |A|$ and $|A_1+A_i|\geq |A_1|$, we find
\begin{align*}
|A+A|&\geq |A_1+A_1|+|A_1+A_2|+\dots +|A_t+A_t|+(\card X-2t)|A_1|\\
&\geq 2^{k+d+1}|A_{\leq t}|+\left(\card X-2t-2^{k+d}\right)|A_1|-\eta |A|.
\end{align*}
This proves \eqref{eq:A+A}. From \eqref{gammaboundeq}, \eqref{eq:A+A}, and $|A+A| < 2^{k+d+1} |A|$ we deduce 
\begin{equation} \label{eq:t}
t < 2/\sqrt{\xi(\eps-\eta)}  \le 2^{1+(k+d)/2} / (\eps-\eta) \sqrt{\gamma}.
\end{equation}
We also deduce 
$(2^{k+d}  + \ell + \eta)|A| \ge 2^{k+d+1}(1-2\gG) |A| +  (\card X-2t-2^{k+d})|A_1|$,
which for small $\eta$ and $\gG$ implies $\card X-2t-2^{k+d} \le 0$. 
As $|A_1| \le |A_{\le t}|$ we obtain
\[ (2^{k+d}  + \ell + \eta)|A| 
 \ge ( 2^{k+d} +  \card X-2t) (1-2\gG)|A|
 \ge ( 2^{k+d} +  \card X-2t - 2^{k+d+2} \gG)|A|, \]
so $\card X \le \ell  + 2t + 2^{k+d+2} \gG + \eta$.
Using \eqref{eq:t} and recalling $\gamma=\frac12 \epsilon^{-2/3}2^{-(k+d)/3}$
we see that $\card X \leq \ell(1+O(\gamma))$.
This completes the proof of the claim, and so of the theorem.
\end{proof}

\begin{proof}[Proof of \Cref{2^k+dstability}] 
Suppose  $k,d\in\mathbb{N}$,  $\gG, \epsilon>0$, $\delta \in (0, 1-\epsilon)$,
and $A \sub \mb{R}^k$ with  $|A+A| \le (2^{k+d}+\delta)|A|$
and $|\gap^{k,d-1}_{O_{k,d,\gG,\eps}(1)}(A)| = \OO_{k,d}(|A|)$.
We need to show that  $|\co^{k,d}(A) \sm A|/|A| \le O_{k,d}(\gG+\dD)$.
We can assume $\gG < \gG_0(k,d,\eps)$ is small
and fix further parameters $t \ll \gG^{-1} \ll n$ with $t > t_0(k,d,\eps)$ large.

By \Cref{FewLocationsI} we have $A \sub P + Q$
for some $d$-GAP $P$ and parallelotope $Q$, 
where by \Cref{SeparatedFreimanThm} we can assume that 
$P$ is $n$-full, $2$-proper and $4Q$-separated
with $|\co^{d,k}(A)| \le |P+Q| = O_{k,d,\eps}(|A|)$.
We can assume $\dD < t^{-1}$, otherwise we are done as
$|\co^{k,d}(A) \sm A| \le |\co^{d,k}(A)| = O_{k,d}(|A|) = O_{k,d}(\dD |A|)$.

By \Cref{lift} we obtain a separated lift 
$B = f_{P,Q}(A) \sub \mb{Z}^d \times \mb{R}^k$,
which is Freiman isomorphic to $A$,
so $|B|=|A|$ and $|B+B|=|A+A|$.
We note that $|\widehat{\co}^{d,k}(B)|\geq |\co^{d,k}(A)|$
and $B$ has thickness $h(B) = \OO_{k,d,\eps,\gG}(n)$ by fullness of $P$.
By \Cref{switchtodiscrete} there is $B'\sub\mb{Z}^{d+k}$ with $h(B')=h(B)$, 
$$\card(B'+B')/\card B' \leq (1+n^{-1}) |B+B|/|B|\text{ and}$$ 
  $$|\widehat{\co}^{d,k}(B)|/|B| \leq (1+n^{-1})\card \widehat{\co}(B')/\card B'.$$
Then $\card(B'+B')/\card B' \le (2^{k+d}+O_{k,d}(\dD))\card B'$, and as 
$h(B') = \OO_{k,d,\eps,\gG}(n)$,  $\dD < t^{-1}$ and $t$ is large, \Cref{zkstab} gives
$\card(\widehat{\co}(B') \sm B') \le O_{k,d}(\gG+\dD) \card B'$.
Finally, $|\co^{k,d}(A) \sm A|/|A| \le |\widehat{\co}^{d,k}(B) \sm B|/|B|
\le \card(\widehat{\co}(B') \sm B')/\card B' + O_{k,d,\eps}(n^{-1})
\le O_{k,d}(\gG+\dD)$.
\end{proof}

\section{Exact results II: cone regime} \label{sec:exactII}

It remains to prove our exact results in the regime
where a discrete cone is asymptotically optimal.

\begin{thm}\label{FewLocationsII}
Suppose $A\subset\mathbb{R}^k$ has
$|A+A|\leq (2^{k+d}+\ell)|A|$ with  $\ell = 2^{k+d}(2-\eps)$ where $\eps>0$,
and $|\gap_{O_{k,d,\eps}(1)}^{k,d-1}(A)|\geq \OO_{k,d,\eps}(|A|)$. 
Then $|\gap^{k,d}_{\ell'}(A)|\leq O_{k,d,\eps}(|A|)$
with $\ell'\leq (1+o_{\eps \to 0}(1))\frac{k+d+1}{2\epsilon}$.
\end{thm}

\subsection{Preliminary lemmas}
We first prove two simple analytic lemmas and one compression lemma.

\begin{lem}\label{bigsteplem}
Let $c,N>0$ and $f:\{0,1,\dots, L\}\to\mb{R}_{\geq0}$ 
with $f(i+1)-f(i)\geq -c$ for all $0 \le i < L$. 
Let $I =\{i: f(i+1)-f(i)\geq N \text{ and } i\leq L-N/c\}$. Then
$\sum_{i=0}^L f(i)\geq \frac{N}{4c}\sum_{i\in I} (f(i+1)-f(i))$.
\end{lem}
\begin{proof}
For each $i\in I$ we define $g_i:\{0,1,\dots, L\}\to\mb{R}_{\geq0}$ by 
$$g_i(j):=\begin{cases}0 &\text{if } j\leq i \text{ or } j-i\geq \left\lceil\frac{f(i+1)-f(i)}{c}\right\rceil\\
f(i+1)-f(i)-(j-i)c &\text{otherwise.}\end{cases}$$
We write $g := \sum_{i \in I} g_i$. Then $g(i+1)-g(i)\leq f(i+1)-f(i)$ holds unless $g(i+1)=0$ in which case trivially$f(i+1)\geq g(i+1)$, so by induction $f(j)\geq g(j)$ for all $j=0,\dots,L$. For each $i \in I$ we have $g_i(j)\geq \frac{f(i+1)-f(i)}{2}$
for $1\leq j-i\leq N/2c$, so $\sum_{j=0}^L f(j) \ge \sum_{j=0}^L g(j)
\ge \frac{N}{4c} \sum_{i \in I} (f(i+1)-f(i))$.
\end{proof}

\begin{lem}\label{secondorderBM}
Let $x,y\in\mathbb{R}$ with $x\leq y$, then 
$$(x^{1/\ell}+y^{1/\ell})^{\ell}\geq 2^{\ell}x+2^{\ell-1}(y-x)-O_\ell(y-x)^2/x.$$
\end{lem}
\begin{proof}
Let $\delta=\frac{y-x}{x}$, $\tT = (1+\dD)^{1/\ell}$ and consider
\begin{align*}
(1+\tT)^{\ell}-2^\ell&=(1+\tT-2)((1+\tT)^{\ell-1} + 2(1+\tT)^{\ell-2} + \dots + 2^{\ell-1}) \\
& \ge (\tT-1) \ell 2^{\ell-1} \ge 2^{\ell-1} \dD - O_\ell(\dD)^2.
\end{align*}
As $(x^{1/\ell}+y^{1/\ell})^{\ell} = x(1+(1+\dD)^{1/\ell})^\ell$, the lemma follows.
\end{proof}

\begin{lem}\label{compressionintok+ddimensions}
Under Setup \ref{setup} with $d'=d$, 
ordering $X$ as $\{x_1,\dots,x_{\card X }\}$ with each $|A_{x_i}|\geq |A_{x_{i+1}}|$,
consider $A'\sub \mb{R}^{k+d}\times \mb{Z}$ 
defined by $A' = \bigcup_{i=1}^{\card X } A'_i \times \{i\}$
with each $A'_i = [0,|A_{x_i}|^{1/(k+d)}]^{k+d}$.

Then $|A|=|A'|$, $|A'+A'|\leq |A+A|+\eta |A|$ and $|\co^{k+d,1}(A')|\leq |\co^{k,d+1}(A)|+\eta |A|$.
\end{lem}
\begin{proof}
Clearly $|A|=|A'|$. Next, analogously to the notation $A_x$, $A^s$, etc in Setup \ref{setup}
we write 
\[ A'^s:=\{i\in \mathbb{Z}: |A'_i|\geq s\} = \{i: x_i \in A^s\} \text{ and } \]
\[(A'+A')^s:=\{i \in \mathbb{Z}: |(A'+A')\cap \mathbb{R}^{k+d}\times\{i\}|\geq s\}.\] 
For any $s$ write $T_s = \{(t,t'): (t^{1/(k+d)}+t'^{1/(k+d)})^{k+d} = s\}$.
Let $\tT = \eta (\card X)^{-2}|A|$.

By  \Cref{FreimanTool} and the definition of $A'$ we have
\[  (A+A)^s\supset \bigcup_{(t,t') \in T_{s+\tT}} (A^t+A^{t'})
 \text{ and }
(A'+A')^s= \bigcup_{t,t' \in T_s}  (A'^t+A'^{t'}), \text{ so} \]
\[  \card (A+A)^s \ge \max_{(t,t') \in T_{s+\tT}} ( \card A^t + \card A^{t'} - 1)
 \text{ and }\]
\[\card (A'+A')^s= \max_{t,t' \in T_s} (\card A'^t + \card A'^{t'} - 1 ). \]
We deduce $\card (A+A)^{s-\tT} \ge \card(A'+A')^s$.
Integrating over $s$ we obtain $|A'+A'|\leq |A+A|+\eta |A|$.

A similar argument gives  $|\co^{k+d,1}(A')|\leq |\co^{k,d+1}(A)|+\eta |A|$.
\end{proof}

\subsection{Proof of the cone regime}

Now we prove the main result of this section.

\begin{proof}[Proof of \Cref{FewLocationsII}]
Suppose $A\subset\mathbb{R}^k$ has
$|A+A|\leq (2^{k+d}+\ell)|A|$ with  $\ell = 2^{k+d}(2-\eps)$ where $\eps>0$,
and $|\gap_{O_{k,d,\eps}(1)}^{k,d-1}(A)|\geq \OO_{k,d,\eps}(|A|)$. 
Under Setup \ref{setup}, we need to show 
$\card X \leq (1+o_{\eps \to 0}(1))\frac{k+d+1}{2\epsilon}$.

We consider $A'\sub \mb{R}^{k+d}\times \mb{Z}$ 
as in \Cref{compressionintok+ddimensions} 
defined by $A' = \bigcup_{i=1}^{\card X } A'_i \times \{i\}$
with each $A'_i = [0,h(i)]^{k+d}$, where $h(i) = |A_{x_i}|^{1/(k+d)}$,
satisfying $|A|=|A'|$, $|A'+A'|\leq |A+A|+\eta |A|$ 
and $|\co^{k+d,1}(A')|\leq |\co^{k,d+1}(A)|+\eta |A|$.
We note that $\co^{k+d,1}(A')_i = [0,h^*(i)]^{k+d}$
where $h^*$ is the upper concave envelope of $h$,
i.e.\ the smallest concave function with $h^* \ge h$.
We introduce parameters satisfying
\[ (k+d)^{-1} \gg \zeta \gg \rho \gg \gG \gg \eps \gg \eta > 0
\quad \text{ and } \quad
t:=\lfloor(1-\rho)\card X\rfloor.\]

We can assume $\card X \ge \eps^{-1} \ge \OO_{k,d,\gG}(1)$, otherwise we are done,
so by  \Cref{2^k+dstability} with $(d+1,1/2)$ in place of $(d,\eps)$
we have $|\co^{k,d+1}(A) \sm A|/|A| \le O_{k,d}(\gG)$,
so  $|\co^{k+d,1}(A') \sm A'|/|A'|  \le O_{k,d}(\gG)$.

We will deduce the following bounds on $|A'_0|$ and $|A'_t|$:
\begin{equation} \label{eq:A'0t}
|A'_0| \le 2(k+d+1)\eps |A'|
\quad \text{ and } \quad
|A'_t| \ge \frac{1}{2(k+d)} \rho^{k+d} |A'_0|.
\end{equation}
The estimate for $|A'_0| = h(0)^{k+d} = h^*(0)^{k+d}$ follows from
$h^*(i) \ge h^*(0)(1-i/\card X)$ (by concavity of $h^*$),
which gives $2|A'| > |\co^{k+d,1}(A')| 
= \sum_i h^*(i)^{k+d} \ge |A'_0| \card X / (k+d+1)$.
Now suppose for a contradiction that the bound for $|A'_t|$ fails.
Then $\sum_{i=t}^{\card X} |A'_i| \le \frac{\card X-t+1}{2(k+d)} \rho^{k+d} |A'_0|$,
whereas concavity of $h^*$ gives
$\sum_{i=t}^{\card X} |\co^{k+d,1}(A')_i| = \sum_{i=t}^{\card X} h^*(i)^{k+d}
\ge \card X  \int_{s=0}^\rho (sh^*(0))^{k+d}  
=  \frac{\card X}{k+d+1} \rho^{k+d+1} |A'_0| $.
However, as $|A'_0|\card X \ge |A|$ and $\gG \ll \rho$
this contradicts $|\co^{k+d,1}(A')\setminus A'|/|A'|  \le O_{k,d}(\gG)$,
so \eqref{eq:A'0t} holds.

Next we note that for any $0 \le i \le t$, 
by concavity $h^*(i) \ge (1-i/t)h^*(0) + (i/t)h^*(t)$, so
$|\co^{k+d,1}(A')_i| \ge (1-i/t)^{k+d} |A'_0| + (i/t)^{k+d} |A'_t|$. 
Thus
\begin{align} \label{eq:coA'}
 |\co^{k+d,1}(A')| &\ge \sum_{i=0}^t  |\co^{k+d,1}(A')_i|
\ge (|A'_0|+|A'_t|) \sum_{i=0}^t  (i/t)^{k+d} \\
&\ge (|A'_0|+|A'_t|)t/ (k+d+1).
\end{align}
The following estimate for $|A'_0|+|A'_t|$ will complete the proof.
Indeed, combined with \eqref{eq:coA'} and $|\co^{k+d,1}(A')| \le (1+O_{k,d}(\gG)) |A'|$
it gives the required bound $\card X\leq (1+o_{\eps \to 0}(1))\frac{k+d+1}{2\eps}$.

\begin{clm}
$|A'_0|+|A'_t|\geq (1-o_{\epsilon\to0}(1))2\eps|A'|$. 
\end{clm}

It remains to prove the claim.
We start by showing that $|\co^{k+d,1}(A')_i| = h^*(i)^{k+d}$ 
cannot decrease too rapidly, in that
\begin{equation} \label{eq:slowdec}
 |\co^{k+d,1}(A')_i| - |\co^{k+d,1}(A')_{i+1}| \le |A'_0|/(\rho^2 \card X)
 \text{ for } 0 \le i \le t-1.
\end{equation}
To see this, consider the largest $t' \ge 0$ such that 
the required inequality holds for $0 \le i \le t'-1$.
Suppose for contradiction that $t'<t$. 
As $h^*(i) \le h^*(0)=h(0) = |A'_0|^{1/(k+d)}$ for all $i$, we have
\[  |A'_0|/(\rho^2 \card X) < h^*(t')^{k+d} -  h^*(t'+1)^{k+d}
 \le (k+d) ( h^*(t') -  h^*(t'+1) ) h(0)^{k+d-1}, \]
By concavity, $h^*(i)-h^*(i+1)$ is increasing in $i$, so
\[ h(0) = h^*(0) \ge \sum_{i>t} (h^*(i)-h^*(i+1))
\ge \frac{(\card X - t' - 1) h(0)}{(k+d)\rho^2 \card X}.  \]
We deduce $\card X - t' - 1 \le (k+d)\rho^2 \card X$,
but this contradicts $t' < t = \lfloor(1-\rho)\card X\rfloor-1$, 
so \eqref{eq:slowdec} holds.

We will use our estimate on the rate of change 
of $|\co^{k+d,1}(A')_i|$ to control that of $|A'_i|$.
We consider $f(i) := |\co^{k+d,1}(A')_i| - |A'_i|$,
which satisfies
$\sum_i f(i) \le  |\co^{k+d,1}(A') \sm A'|  \le O_{k,d}(\gG) |A'|$
and
\[ f(i+1) - f(i) \ge |\co^{k+d,1}(A')_{i+1}| - |\co^{k+d,1}(A')_i| \ge -c,\]
where $c := |A'_0|/(\rho^2 \card X).$
We let $I = \{ i \le t: |A'_i| - |A'_{i+1}| \ge \zeta |A'_0| \}$
and note that $f(i+1)-f(i) \ge N := \zeta |A'_0|  -  c$
for all $i \in I$. Applying \Cref{bigsteplem}, which is valid as 
$t \le (1-\rho)\card X \le L - N/c = \card X + 1 - \zeta \rho^2 \card X$,
we obtain $\sum_i f(i)\geq \frac{N}{4c}\sum_{i\in I} (f(i+1)-f(i))$, 
so as $\gG \ll \rho$,
\begin{equation} \label{eq:I}
\sum_{i\in I} (f(i+1)-f(i)) \le O_{k,d}(\gG c/N) |A'| < \sqrt{\gG} |A'_0|. 
\end{equation}
Next we will estimate $|A'+A'|$, using 
$A'+A'\supset \bigcup_{i=0}^{\card X}(A'_i+A'_i)\cup (A'_i+A'_{i+1})$. 
We use the bounds $|A'_i+A'_i|\geq 2^{k+d}|A'_i|$ 
and $|A'_i+A'_{i+1}|\geq 2^{k+d}|A'_{i+1}|$,
except that for $i<t$ with $i \notin I$ we use \Cref{secondorderBM},
which gives the stronger bound
$$|A'_i+A'_{i+1}|\geq 2^{k+d}|A'_{i+1}|+2^{k+d-1}(|A'_i|-|A'_{i+1}|)
- O_{k,d}(|A'_i|-|A'_{i+1}|)^2/|A'_{i+1}|.$$
Combining these bounds, we find
\begin{equation} \label{eq:A'+A'}
|A'+A'| \geq 2^{k+d+1}|A'|-2^{k+d}|A'_0|+ \sum_{i<t: i\notin I}
 \left( 2^{k+d-1}(|A'_i|-|A'_{i+1}|) - O_{k,d}(|A'_i|-|A'_{i+1}|)^2/|A'_{i+1}| \right).
\end{equation}
As $|A'_i| \ge |A'_t|\ge \frac{1}{2(k+d)} \rho^{k+d} |A'_0|$ by \eqref{eq:A'0t},
$|A'_i|-|A'_{i+1}| < \zeta |A'_0|$ for $i \notin I$, and $\rho \ll \zeta$ we have
\[ \sum_{i \notin I} O_{k,d}(|A'_i|-|A'_{i+1}|)^2/|A'_{i+1}|)
   \le O_{k,d}(\zeta \rho^{-k-d}) \sum_i (|A'_i|-|A'_{i+1}|) \le \sqrt{\zeta} |A'_0|. \]
Furthermore, by \eqref{eq:I} we have
$\sum_{i<t: i\notin I} (|A'_i|-|A'_{i+1}|) \ge |A'_0|-|A'_t| - \sqrt{\gG} |A'_0|$,
so by \eqref{eq:A'+A'},
\[ |A'+A'| \geq 2^{k+d+1}|A'| - 2^{k+d}|A'_0|
+2^{k+d-1}(|A'_0|-|A'_t| - \sqrt{\gG} |A'_0|) -\sqrt{\zeta} |A'_0|. \]
Combining this with $|A'+A'|\leq |A+A|+\eta|A| \le 2^{k+d}(2-\eps+\eta)|A'|$,
also using $|A'_0| = O_{k,d}(\eps|A|)$ by  \eqref{eq:A'0t} and $\eta \ll \eps \ll \zeta$,
we deduce $|A'_0|+|A'_t|\geq 2(\eps-\eta)|A'| - 2\sqrt{\zeta} |A'_0| \ge (1-\zeta^{1/3})2\eps|A'|$.

This concludes the proof of the claim and so of the theorem.
\end{proof}

\section{Beyond $3k-4$} \label{sec:3k-4}

In this section we will prove \Cref{4kthm}, which gives a precise structural characterisation
of sets $A\subset\mathbb{R}$ with $|A+A|<4|A|$. 

\subsection{Preliminaries}

One key ingredient will be the following stability result
of Moskvin, Freiman, Yudin \cite{moskvin1974structure} (see also \cite{bilu1998alpha})
for  the Cauchy-Davenport inequality in the circle $\mathbb{R}/\mathbb{Z}$.

\begin{thm}[Moskvin, Freiman, Yudin, 1973]\label{moskvin}
There exists a universal constant $c = c_{\ref{moskvin}}>0$ such that
if $A\subset \mathbb{R}/\mathbb{Z}$ with $|A|\leq c$ and $|A+A|\leq 3|A|$ 
then there is an interval $J$ and $n\in\mathbb{N}$ 
so that $nA\subset J$ and $|J\setminus nA|\leq |A+A|-2|A|$.
\end{thm}

We remark that a general stability result in compact connected abelian groups was proved by Tao \cite{tao2018inverse}. However, this is not quantitative, and there is considerable interest
(e.g. \cite{green2006sets,candela2019step,candela2019sets,lev2023towards})
in improving the constants in specific settings.
In particular, improving the constant $c$ in \Cref{moskvin}
would improve the constant in \Cref{4kthm},
and is closely related to a conjecture of Bilu \cite{bilu1998alpha}.

Our first task in the proof of \Cref{4kthm} is to show the existence of $t$ as in the statement,
i.e.\ that if $|\co(A)|/|A|$ is large and $|A+A|<4|A|$ then we can efficiently cover $A$ 
by an arithmetic progression of intervals. We will quickly deduce this from the above theorem.

\begin{prop}\label{GAPprop}
There exists a universal constant $C = C_{\ref{GAPprop}}>0$ so that
if $A\subset \mathbb{R}$ with $|\co(A)|\geq C|A|$ and $|A+A|=(3+\tT)|A|$ with $\tT<1$
then there exists an  arithmetic progression $X$ and an interval $I$
with $A\subset X+I$ and $|(X+I\cap \co(A))\setminus A|\leq \tT |A|$.
\end{prop}

\begin{proof}
We translate and rescale $A$ so that $\co(A)=[0,1]$ and $|A|<c_{\ref{moskvin}}$.
Let $f:\mathbb{R}\to\mathbb{R}/\mathbb{Z}$ be the canonical quotient map. 
For any $a\in A$ we have $|f^{-1}(a)\cap (A+A)|\geq 2$, as $\{0,1\}+A\subset A+A$. 
We deduce $|f(A)+f(A)| = |f(A+A)| \le |A+A|-|A| = (2+\tT)|f(A)|$,
so by \Cref{moskvin} we find an interval $J$ and $n\in\mathbb{N}$ 
so that $nf(A)\subset J$ and $|J\setminus nf(A)|\leq \tT|A|$.
Pulling back to $\mathbb{R}$ gives $A\subset X+I$ of the required form.
\end{proof}

We note the following useful corollary.
\begin{cor}\label{BiluCor}
Suppose $A\subset \mathbb{R}$ with $|\co(A)|\geq C_{\ref{GAPprop}} |A|$.
If $|\co^{1,1}(A)|>2|A|$ then $|A+A|\geq 4$.
\end{cor}

Next we state our main lemma for this section, which will be proved in the next subsection.

\begin{lem}\label{mainlem}
Let $A\subset\mathbb{R}$ with $|\co(A)|>C_{\ref{GAPprop}}|A|$.
Suppose $|A+A|=(4-2/t+\delta)|A|$,
with $t$ minimal so that $|\co_t(A)|<2|A|$,
where $\delta <\Delta_t := (2t)^{-2}$.
Then $|\co^{1,1}(A) \sm A| \le 150\delta |A|$.
\end{lem}

The main theorem of this section
will follow from an interpolation between \Cref{BiluCor} and \Cref{mainlem}.

We conclude our preliminaries by recording some `$3k-4$ type' inequalities
to be used in the proof of the main lemma.

\begin{lem}\label{basicadditioninequalities}
Consider $X,Y\subset\mathbb{R}$ with $|\co(Y)|\geq |\co(X)|$, then we have the following lower bounds:
$$|X+Y|\geq \begin{cases}
2|X|\\
|X|+|\co(Y)|+\min\{0,|X|-|\co(Y)\setminus Y|\}\\
|\co(X)|+|Y|+\min\{0,|X|-|\co(Y)\setminus Y|\}\end{cases}$$
\end{lem}
\begin{proof}
For the first bound, simply note that $(\min Y)+X$ and $(\max Y)+X$ are disjoint. For the second consider the canonical quotient map $f:\mathbb{R}\to\mathbb{R}/|\co(Y)|$. By the Cauchy-Davenport inequality, we have
$$|f(X+Y)|\geq \min\{|\co(Y)|,|f(X)|+|f(Y)|\}=\min\{|\co(Y)|,|X|+|Y|\}.$$
Hence, we find:
\begin{align*}
|X+Y|&\geq |(X+Y)\cap f^{-1}(X)|+|(X+Y)\setminus f^{-1}(X)|\\
&\geq |(X+\{\min Y, \max Y\})\cap f^{-1}(X)|+|f(X+Y)\setminus X|\\
&\geq 2|X|+\min\{|\co(Y)|,|X|+|Y|\}-|X|\\
&=|X|+|\co(Y)|+\min\{0,|X|-|\co(Y)\setminus Y|\}.
\end{align*}

For the third bound, note that we can translate in such a way that $\min X=\min Y$. Partition $Y$ into $Y_1:=Y\cap \co(X)$ and $Y_2=Y\setminus \co(X)$. Note that $Y_2+\max X$ is disjoint from $Y_1+X$. Hence, we find (using the previous bound applied to the sets $Y_1$ and $X$),
\begin{align*}
|X+Y|&\geq |(\max X)+Y_2|+ |X+Y_1|\\
    &\geq |Y_2|+|Y_1|+\min\{|\co(X)|,|Y_1|+|X|\}\\
    &\geq |Y|+|\co(X)|+\min\{0,|X|-|\co(Y)\setminus Y|\}
\end{align*}
where in the last inequality we used 
$|Y_1|= |\co(X)|-|\co(X)\setminus Y_1|\geq |\co(X)|-|\co(Y)\setminus Y|$.
\end{proof}

\subsection{Proof of the main lemma}

The remainder of this section is occupied with the proof of the main lemma,
which implies \Cref{4kthm}.
 
\begin{proof}[Proof of \Cref{mainlem}]
Let $A\subset\mathbb{R}$ with $|\co(A)|>C_{\ref{GAPprop}}|A|$.
Suppose $|A+A|=(4-2/t+\delta)|A|$,
with $t \ge 2$ minimal so that $|\co_t(A)|<2|A|$,
where $\delta <\Delta_t := (2t)^{-2}$.
By \Cref{GAPprop} we have $A\subset X+I$ for some
arithmetic progression $X$ and interval $I$
with $|(X+I\cap \co(A))\sm A|\leq |A+A|-3|A| = (1-2/t+\delta)|A|$.

Let $X=\{x_1,\dots,x_{t'}\}$ and $A_i:=A\cap (x_i+I)$.
We can assume $A_1$ and $A_{t'}$ are non-empty.
We will show below that $t'=t$, 
but a priori we do not know this yet, 
as some of the $A_i$ might be empty.

Next we translate the problem to its proper context: two dimensions. Akin to \Cref{lift} let
\[ B:=\bigcup_{i=1}^{t'}B_i\subset [t']\times I,
\text{ where each } B_i:=\{i\}\times (A'_i-x_i).\]
We note that $|B|=|A|$ and $|B+B|=|A+A|$, using $|x_{i+1}-x_i|\geq 2|I|$, 
which follows from $|\co(A)|>4|A|$. Furthermore, $|\co^{1,1}(A)| \le |\co^{1,1}(B)|$,
as $A \sub \phi(\co^{1,1}(B))$ for the linear map $\phi: \mb{R}^2  \to \mb{R}$ 
defined by $\phi(i,y) = x_1 + (i-1)(x_2-x_1) + y$.

We will approach the structure of $B$ via that of the following compressed set $B'$.
Write $\{x\in X: x+I\cap A\neq \emptyset\} = \{y_1,\dots,y_t\}$ 
and $A'_i:= A \cap (y_i+I)$, ordered with $|A'_1| \ge \dots \ge |A'_t|$.
We let
\[ B':=\bigcup_{i=1}^t B'_i, \text{ where each } B'_i:=\{i\}\times [0,|A'_i|].\]

\begin{clm}
 $|B'|=|B|=|A|$ and  $|B'+B'| \le |B+B| = |A+A|$.
\end{clm}
\begin{proof}[Proof of Claim]
We consider $|B| = \int_s \card B^s$ with $B^s := \{i: |B_i| \ge s\}$
and similar expressions for $|B+B|$ and $|B'+B'|$. 
We have $(B+B)^s \supset \bigcup_{x+y=s} (B^x + B^y)$,
so $\card (B+B)^s \ge \max_{x+y=s} \card(B^x + B^y)$.
Each $B'^s$ is an initial segment of $\mb{N}$ with $\card B'^s = \card B^s$,
so $\card (B'+B')^s = \max_{x+y=s}  \card(B'^x + B'^y) \le \card (B+B)^s$.
Integrating over $s$ gives $|B'+B'| \le |B+B|$, as claimed.
\end{proof}

Next we consider $\co^{1,1}(B') = C \cap (\mb{Z} \times \mb{R})$,
where $C = \co(B')$ is the continuous convex hull of $B'$ in $\mb{R}^2$.
We denote the vertices of $C$ by 
$V(C) = \{ (1,0), (t,0), (i_1,|A_{i_1}|), \dots, (i_s,|A_{i_s}|) \}$ for some $s \in \mb{N}$.

We show the following stability result for approximating $B'$ by  $\co^{1,1}(B')$,
which can be seen as the discrete version of the one dimensional result 
in \cite[Theorem 1.1]{van2020sharp}.

\begin{clm}\label{T1intro}
$|\co^{1,1}(B')\sm B'|\leq \frac{t}{2(t-1)}\delta |B'|$ (and so $\dD \ge 0$).
\end{clm}
\begin{proof}[Proof of Claim]
We start by noting the following two inclusions (writing $B'_{t+1}:=\es$):
$$T'_1:=T_1(B') := \bigcup_{i=1}^{t} (B'_i+B'_i) \cup (B'_{i}+B'_{i+1})\subset B'+B',\text{ and}$$
$$T'_2:= T_2(B') := \bigcup_{i=1}^{t} (B'_1+B'_i)\cup  (B'_i+B'_{t})\subset B'+B'.$$
 As $|T_1'|= 4|B'|-(|B'_1|+|B'_{t}|)$ and $|T_2'|=2|B'|+(t-1)(|B'_1|+|B'_{t}|)$, we have
 \begin{equation} \label{eq:dD} 
 (4-2/t+\dD)|B'| \ge |B'+B'|\geq \frac{t-1}{t} |T'_1|+\frac{1}{t} |T'_2|= (4-2/t)|B'|,  
 \text{ so } \dD \ge 0.
 \end{equation}
 Now we consider a third inclusion 
 that emphasises the vertices of $\co^{1,1}(B')$:
 $$T'_3:=\bigcup_{j=1}^{s-1} \bigcup_{k=0}^{i_{j+1}-i_j}(B'_{i_j}+B'_{i_j+k})\cup (B'_{i_{j+1}}+B'_{i_j+k})\subset B'+B'.$$
When we compare $T'_3$ to $T'_1$, we find that 
\begin{align} \label{eq:T3-T1} 
 |T'_3|-|T'_1|&=\sum_{j=1}^{s-1} \left(
\left(i_{j+1}-i_j-1\right)\left(|B'_{i_j}|+|B'_{i_{j+1}}|\right)
-2\sum_{k=1}^{i_{j+1}-i_j-1}|B'_{i_j+k}| \right)\\
&=2|\co^{1,1}(B')\sm B'|,
\end{align}
where the second equality holds as
\begin{align*}|\co^{1,1}(B')\sm B'|
&=\sum_{j=1}^{s-1} \sum_{k=1}^{i_{j+1}-i_j-1} \left( |\co^{1,1}(B')_{i_j+k}|-|B'_{i_j+k}| \right)\\
&=\sum_{j=1}^{s-1} \sum_{k=1}^{i_{j+1}-i_j-1} \left(
  \Big(\frac{k}{i_{j+1}-i_j}|B'_{i_{j+1}}|+\Big(1-\frac{k}{i_{j+1}-i_j}\Big)|B'_{i_j}|\Big)-|B'_{i_j+k}| \right).
\end{align*}
From \eqref{eq:dD} and \eqref{eq:T3-T1}
we deduce the following inequality that proves the claim
\begin{align*} (4-2/t+\dD)|B'| & \ge  |B'+B'|\geq \frac{t-1}{t} |T'_3|+\frac{1}{t} |T'_2|\\
&=(4-2/t)|B'| +\frac{2(t-1)}{t}|\co(B')\setminus B'|. \nonumber \qedhere \end{align*}
\end{proof}

This approximation of $B'$ by $\co^{1,1}(B')$ implies the following size estimates.

\begin{clm}\label{Bisize}
$|B_1'|+|B_t'|\geq (\frac2t-\DD_t)|B'|$ and
$|B_i'|\geq \left(\frac{2}{t(t-1)}-\DD_t\right)|B'|$ for all $2\leq i\leq t-1$.
\end{clm}
\begin{proof}[Proof of Claim]
From $|T_1'|= 4|B'|-(|B'_1|+|B'_{t}|) = |T'_1| \le |B'+B'| \le (4-2/t+\DD_t)|B'|$ we obtain the first bound.
For the second, it suffices to show $h(i) := |\co^{1,1}(B'_i)| \ge \frac{2}{t(t-1)} |B'|$
for $2\leq i\leq t-1$. We can assume $t \ge 3$ (otherwise it is vacuous).
As $h$ is concave, it suffices to show this for $i=2$ and $i=t-1$.

Suppose for a contradiction that $h(2) < \frac{2}{t(t-1)} |B'|$.
By concavity, $h(i)-h(i+1)$ is non-decreasing in $i$, so for $i \ge 2$ we obtain
$$h(i) \leq h(1) - (i-1)(h(1)-h(2))< \frac{2(i-1)}{t(t-1)}|B'| - (i-2)h(1).$$
However, summing this over $i \ge 2$ gives  
$|\co^{1,1}(B')|-h(1) < |B'|-h(1)\tbinom{t-2}{2}$,
so $|\co^{1,1}(B')| < |B'|$, which is a contradiction.
Similarly, if $h(t-1) < \frac{2}{t(t-1)} |B'|$ then for $i \ge 1$ we obtain
$$h(t-i) \le h(t) + (t-i)(h(t-1)-h(t)) <  \frac{2(t-i)}{t(t-1)}|B'| - (t-i-1)h(t),$$
so $|\co^{1,1}(B')|-h(t) < |B'|-h(t)\tbinom{t-2}{2}$, again a contradiction.
\end{proof}

Returning our attention to $B$ instead of $B'$,
we can now deduce that all $B_i$ are non-empty.
Similarly to \Cref{T1intro} we write
$$ T_1 := T_1(B):=\bigcup_{i=1}^{t'} (B_i+B_i) \cup (B_{i}+B_{i+1})\subset B+B,\text{ and}$$
$$T_2 := T_2(B):=\bigcup_{i=1}^{t'} (B_1+B_i)\cup  (B_i+B_{t})\subset B+B.$$

\begin{clm} \label{t=t'}
We have $t=t'$, i.e.\ all $B_i \ne \es$.
\end{clm}
\begin{proof}[Proof of claim]
Assume for a contradiction $t'>t$. 
Let $\{p_1<\dots<p_t\}=P:=\{i: B_i\neq \emptyset\} = \pi_1(B)$ 
be the projection of $B$ onto the first coordinate.
As $\card P = t$ and $\card \co(P)=t'>t$ we have $\card(P+P)\geq 2t$. 
As $\card \pi_1(T_1)=2t-1$,
we can find $(i,j)$ with $T_1 \cap (B_{p_i}+B_{p_j}) =\es$.
As $|T_1| \ge 4|B|-(|B_1|+|B_{t'}|)$ and $|T_2| \ge 2|B'|+(t-1)(|B_1|+|B_{t'}|)$, 
similarly to \eqref{eq:dD} we have
\begin{align*} (4-2/t+\DD_t)|B| &\ge \frac{t-1}{t} (|T_1|+|B_{p_i}+B_{p_j}|)+\frac{1}{t} |T_2|\\
&\ge (4-2/t)|B| +  \frac{t-1}{t} \left(\frac{2}{t(t-1)}-\DD_t\right)|B|, \end{align*}
by \Cref{Bisize}. This is a contradiction for $\DD_t = (2t)^{-2}$, so indeed $t'=t$. 
\end{proof}

Now we have established that $\pi_1(B)$ is an interval.
Our next comparison will be to the following set $D$ where each line $B_i$ 
is replaced by its containing interval $\co(B_i)$. Let
\[ D:=\bigcup_{i=1}^t D_i, \text{ where each } D_i:=\co(B_i).\]

\begin{clm}\label{DminusBfirst}
$|D \sm B|\leq 6\dD |B|$. 
\end{clm}

\begin{proof}[Proof of claim]
We start by showing that 
\begin{equation} \label{eq:Di-Bi}
\sum_{i \in I} |D_i\sm B_i|\leq \frac{t}{t-1} \dD |B|,
\text{ where } I = \left\{i: |B_i| \ge  \frac{t}{t-1} \dD |B| \right\}.
\end{equation}
For any $i \in I$ we claim that $|B_i+B_i| < 3|B_i|$.
Indeed, otherwise $|T_1| \ge |B_i| + 4|B|-(|B_1|+|B_{t'}|)$, 
so similarly to \eqref{eq:dD} we have
\[ (4-2/t+\dD)|B| \ge \frac{t-1}{t} |T_1| +\frac{1}{t} |T_2|
 \ge (4-2/t)|B| +  \frac{t-1}{t} |B_i|, \]
which is a contradiction by defininition of $I$.
As $|B_i+B_i| < 3|B_i|$ we have
$|B_i+B_i|\geq 2|B_i|+|\co(B_i)\sm B_i|$
by Freiman's $3k-4$ theorem,
so $|T_1| \ge 4|B|-(|B_1|+|B_{t'}|) + \sum_{i \in I} |D_i\sm B_i|$.
Repeating the previous calculation gives \eqref{eq:Di-Bi}. 

From \Cref{Bisize} and $\DD_t=(2t)^{-2}$ we have $|[t] \sm I| \le 1$. 
Suppose $[t] \sm I = \{x\}$.  We fix $x' \in \{x-1,x+1\} \cap [t]$
and compare $|D_x \sm B_x|$ to $|D_{x'} \sm B_{x'}|$.
First we consider the case $|\co(B_x)| \ge |\co(B_{x'})|$.
Then $|B_x+B_{x'}|\geq 2|B_{x'}|$ by  \Cref{basicadditioninequalities}, 
so  $|T_1| \ge 4|B|-(|B_1|+|B_{t'}|) + (|B_{x'}|-|B_x|)$,
and as in \eqref{eq:dD} we have $\frac{t}{t-1} (|B_{x'}|-|B_x|) \le \dD |B|$.
However, this gives 
$|B_x| \ge |B_{x'}| - \frac{t}{t-1} \dD |B| > \frac{t}{t-1} \dD |B|$ 
by \Cref{Bisize} and $\dD \le \DD_t = (2t)^{-2}$, 
so $x \in I$, which is a contradiction. 

We deduce $|\co(B_{x'})| \ge |\co(B_x)|$, so by \Cref{basicadditioninequalities}
\begin{align*}  |B_x+B_{x'}|&\geq |\co(B_x)|+|B_{x'}| + \min\{0,|B_x|-|\co(B_{x'})\sm B_{x'}|\}\\
&\ge  |\co(B_x)|+|B_{x'}|  - |\co(B_{x'})\sm B_{x'}|.\end{align*}
Now $|T_1| \ge 4|B|-(|B_1|+|B_{t'}|) +  |\co(B_x) \sm B_x| - |\co(B_{x'})\sm B_{x'}|$,
so  as in \eqref{eq:dD} we have
$ |\co(B_x) \sm B_x| \le |\co(B_{x'})\sm B_{x'}| +  \frac{t}{t-1} \dD |B|$.
As $|\co(B_{x'})\sm B_{x'}| \le \frac{t}{t-1} \dD |B|$ by \eqref{eq:Di-Bi}
the claim follows.
 \end{proof}

Next we will compare the sumsets $B_i + B_j$ and $D_i + D_j$.

\begin{clm}\label{coBminusB}
For any $i,j \in [t]$ we have $|(D_i+D_j)\sm(B_i+B_j)|\leq 2|D_i\sm B_i|$,
except for at most one $i$ with $|D_i| < 6\dD|B|$,
for which $|(D_i+D_j)\sm(B_i+B_j)|\leq |D_i \sm B_i| + 6\dD|B|$.
\end{clm}
\begin{proof}[Proof of claim]
If $|\co(B_i)|\geq |\co(B_j)|$ the claim holds 
by \Cref{basicadditioninequalities}, which gives
\begin{align*}|B_i+B_j|&\geq |\co(B_j)|+|B_i|+\min\{0,|B_j|-|\co(B_i)\sm B_i|\\
&\geq |\co(B_j)|+|\co(B_i)|-2|\co(B_i)\sm B_i|.\end{align*}
If $|\co(B_i)|\leq |\co(B_j)|$ then the claim holds by
\Cref{basicadditioninequalities} and \Cref{DminusBfirst}, which give
\begin{align}
|B_i+B_j| & \geq |\co(B_j)|+|B_i|+\min\{0,|B_i|-|\co(B_j)\sm B_j|\} \nonumber \\
& \geq |\co(B_j)|+|\co(B_i)|-|\co(B_i)\sm B_i|+\min\{0,|B_i|-6\dD|B|\}.
\nonumber \qedhere 
\end{align}
\end{proof}

Next we denote the vertices of $\co(D)=\co(B)$ by \\
$\{\min B_{i^-_1},\dots,\min B_{i^-_{s^-}},\max B_{i^+_1},\dots,\min B_{i^+_{s^+}}\}$,
where the minima / maxima are the vertices on the lower / upper envelope of $\co(D)$.

Similarly to \Cref{T1intro}, for $X \in \{B,D\}$ we now introduce the third type of reference set
$$T^\pm_3(X):=\bigcup_{j=1}^{s^\pm-1} \bigcup_{k=0}^{i^\pm_{j+1}-i^\pm_j}
\left(X_{i^\pm_j}+X_{i^\pm_j+k}\right)\cup\left(X_{i^\pm_{j+1}}+X_{i^\pm_j+k}\right)\subset X+X.$$

We write $T(D) = T_1(D)\cup T_2(D)\cup T_3^-(D)\cup  T_3^+(D)$
for the combined reference set for $D+D$
(with $T_1(\cdot)$,  $T_2(\cdot)$ as above),
which we will use to approximate $D$ by its generalised convex hull, as follows.

\begin{clm}\label{DplusD}
We have $\frac{2(t-1)}{t}|\co^{1,1}(D)\sm D| \le |T(D)| - (4-2/t)|D|$.
\end{clm}

\begin{proof}[Proof of claim]
Consider $T'(D) := T_1(D)\cup T_3^-(D)\cup  T_3^+(D)$.
First we show that each $T'(D)_i$ is an interval. Indeed, if this was false for some $i$, 
then some interval of $T'(D)_i$ would contain the interval $T_1(D)_i$, and another 
would contain some interval $D_p + D_q$ disjoint from $T_1(D)$.
However, this implies $(B_p + B_q) \cap T_1(B) = \es$,
which would give a contradiction as in the proof of \Cref{t=t'}.

Now we will show that $|T'(D) \sm T_1(D)| = 2|\co^{1,1}(D) \sm D|$.
This easily implies the claim, using
$|T(D)| \ge \frac{t-1}{t} |T'(D)| + \frac{1}{t} |T_2(D)|$
and $\frac{t-1}{t} |T_1(D)| + \frac{1}{t} |T_2(D)| \ge (4-2/t)|D|$.

We write $|T'(D)_i| = F(i) + G(i)$, where $\min T'(D)_i = \min T^-_3(D)_i = (i,-F(i))$ 
and $\max T'(D)_i = \max T^+_3(D)_i = (i,G(i))$. 
Similarly, we write $|D_i| = f(i) + g(i)$,
where $\min D_i = (i,-f(i))$ and $\max D_i = (i,g(i))$,
and $|T_1(D)_i| = f_1(i) + g_1(i)$,
where $\min T_1(D)_i = (i,-f_1(i))$ and $\max T_1(D)_i = (i,g_1(i))$.

As in the proof of \Cref{T1intro} we have 
\[ \sum_i (G(i)-g_1(i)) = 2\sum_i (\co(g)(i)-g(i)),\]
where $\co(g)$ is the upper concave envelope of $g$
(we can assume $g,G,g_1$ are non-negative 
by adding some positive constant to all functions,
as this does not affect the identity).
Adding the similar identity for $f$ gives
$|T'(D)| - |T_1(D)| = 2|\co^{1,1}(D) \sm D|$,
as required to prove the claim.
\end{proof}

The following bound on $|T(D) \sm (B+B)|$
will make the approximation in \Cref{DplusD} effective.

\begin{clm}\label{DminusBcontrol}
We have $|T(D) \sm (B+B)| \leq 48\dD |B| + 16|D\sm B|$. 
\end{clm}
\begin{proof}
This holds as for each $X \in \{T_1(D),T_2(D),T^+_3(D),T^-_3(D)\}$
we have $|X \sm (B+B)| \le 12\dD |B| + 4|D \sm B|$.
To see this, consider the following inequalities and \Cref{coBminusB}.
\begin{align*}
|T_1(D)\sm (B+B)|&\leq \sum_{i=1}^{t} |(D_i+D_i)\sm(B_i+B_i)| 
 + |(D_i +D_{i+1})\sm(B_{i}+B_{i+1})|,
\end{align*}
\begin{align*}
|T_2(D)\sm (B+B)|&\leq \sum_{i=1}^{t} |(D_1+D_i)\sm(B_1+B_i)| 
+ |(D_{i}+D_{t})\sm(B_{i}+B_{t})|,
\end{align*}
\begin{align*}
|T^\pm_3(D)\sm T^\pm_3(B)|&\leq\sum_{j=1}^{s^\pm-1} \sum_{k=0}^{i^\pm_{j+1}-i^\pm_j}\left|\left(D_{i^\pm_j}+D_{i^\pm_j+k}\right)\setminus \left(B_{i^\pm_j} +B_{i^\pm_j+k}\right)\right|\\
&\ \ \ \ \ +\left|\left(D_{i^\pm_{j+1}}+D_{i^\pm_j+k}\right)\setminus \left(B_{i^\pm_{j+1}}+B_{i^\pm_j+k}\right)\right|\\
&\leq\sum_{j=1}^{s^\pm-1} \sum_{k=0}^{i^\pm_{j+1}-i^\pm_j}4\left|D_{i^\pm_j+k}\setminus  B_{i^\pm_j+k}\right|.
\end{align*}
Summing the above bounds proves the claim.
\end{proof}

Combining  \Cref{DplusD}, \Cref{DminusBcontrol} and
$|\co^{1,1}(D) \sm D| = |\co^{1,1}(B) \sm B| - |D \sm B|$, we obtain
\begin{align*} |B+B| &\ge |T(D)| - |T(D) \sm (B+B)|\\
&\ge (4-2/t)|B|   +  \frac{2(t-1)}{t}|\co^{1,1}(B)\sm B|
  -   48\dD |B| - 15|D\sm B|. \end{align*}
Recalling $|D \sm B| \le 6\dD |B|$ from \Cref{DminusBfirst},
$|A|=|B|$ and $|A+A|=|B+B| \le (4-2/t + \dD)|B|$, we conclude that
$|\co^{1,1}(A) \sm A| \le |\co^{1,1}(B) \sm B| \le 150\dD |A|$.
\end{proof}

\begin{proof}[Proof of \Cref{4kthm}]
Suppose $A\subset \mathbb{R}$ with $|A+A|<4|A|$
and $|\co(A)|\geq C_{\ref{GAPprop}}|A|$.
By  \Cref{GAPprop} we can define $t$ minimal so that $|\co_t(A)|<2|A|$.
Let $\delta := \frac{|A+A|}{|A|} - (4-2/t)$.
If $\delta<(2t)^{-2}$ then \Cref{mainlem} gives 
$|\co^{1,1}(A)\setminus A|\leq 150\delta|A|$.

It remains to compare $\co^{1,1}(A)$ to $\ap_t(A)$. We use the setup from the previous proof. 
We aim to show that $(|\co^{1,1}(B)_1|,\dots,|\co^{1,1}(B)_t|)$ is close to an arithmetic progression.
We let $L(x) = \frac{ (x-1)|\co^{1,1}(B)_t| + (t-x)|\co^{1,1}(B)_1|}{t-1}$ 
be the linear function with $L(1)=|\co^{1,1}(B)_1|$ and $L(t)=|\co^{1,1}(B)_t|$.
We let $c = \max\{ |\co^{1,1}(B)_i|-L(i): i \in [t]\}$.
Then $|\ap_t(A) \sm A| \le ct + |\co^{1,1}(B) \sm B|$
and $|\co^{1,1}(B)| \ge ct/2 + \sum_{i=1}^t L(i) = ct/2 + (|\co^{1,1}(B)_1|+|\co^{1,1}(B)_t|)t/2$.
By the proof of  \Cref{DminusBcontrol},
also using $D_1=\co^{1,1}(B)_1$ and $D_t=\co^{1,1}(B)_t$, we have
\[   |B+B| + 36\dD |B|  \ge |T_1(D)| = 4|D| - (|D_1|+|D_t|) 
 \ge 4|D| - |\co^{1,1}(B)|2/t + c,\]
where  $|\co^{1,1}(B)| \le |B| + 150\dD |B|$,
so $(4-2/t+\dD)|B| + 36\dD |B|
\ge 4|D| - (|B| + 150\dD |B|)2/t +c$,
i.e.\ $c \le (37+300/t)\dD  |B| $.
We deduce  $|\ap_t(A) \sm A| \le ct + |\co^{1,1}(B) \sm B|
\le (37t + 450)\dD |B| \le 500t \dD |B|$.
\end{proof}

\section{Summing distinct sets}

Though the focus throughout the paper has been on sumsets of the form $A+A$, many of the results extend to the context of sumsets $A+B$. Many of the corresponding proofs for summing distinct sets are essentially analogous to those for summing sets to themselves, but for the sake of accessibility we have chosen to not include them. In this section, we explore some of the results and sketch what extra steps are necessary in the proofs. We focus on the case $|A|=|B|$, see for instance \Cref{ruzsaconj} for the general case. 

The main technical tool of this paper \Cref{SeparatedFreimanThm} and its corollary \Cref{FreimanTool} extend to distinct sets as follows.

\begin{thm}\label{ABFreimanThm}
Let $k,d,\ell_i,n_i\in\mathbb{N}$ for $1\leq i\leq d$, $\epsilon>0$, and $A,B\subset\mathbb{R}^k$, with $|A|=|B|$. 
If $|A+B|\leq (2^{k+d+1}-\epsilon)|A|$ then $A,B\subset X+P+Q$, where
\begin{itemize}
    \item $P=-P$ is an $\ell_{d'}$-proper $n_{d'}$-full $d'$-GAP with $d'\leq d$,
    \item $Q=-Q$ is a parallelotope,
    \item $\card X=O_{k,d,\epsilon,n_{d'+1},\dots,n_{d}}(1)$
    \item $\card P \cdot  |Q|=O_{k,d,\epsilon,\ell_{d'},\dots, \ell_{d},n_{d'+1},\dots,n_{d}}(|A|)$,
    \item $\ell_{d'}\cdot X$ is $\ell_{d'}\ell\cdot \frac{(P+Q)}{\ell}$-separated for all $1\leq\ell\leq\ell_{d'} $, and
    \item $\ell_{d'}\cdot P$ is $\ell_{d'}Q$-separated.
\end{itemize}
For all $\eta>0$ there is a choice of $\ell_i$ and $n_i$ so that writing $A_x:=A\cap (x+P+Q)$ and $B_x:=B\cap (x+P+Q)$ for $x\in X$, for all $x,y,z,w \in X$ we have
\begin{itemize}
        \item  $|A_x+B_y|\geq \left(|A_x|^{1/(k+d')}+|B_y|^{1/(k+d')}\right)^{k+d'}-\eta (\card X) ^{-2}|A|$, and
        \item if $(A_x+B_y)\cap (A_z+B_w)\neq \emptyset$ then $x+y=z+w$.
\end{itemize} 
\end{thm}

\begin{proof}[Sketch of proof]
Letting $S=\co(A\cup B)$, we find that $S+S=(A+A)\cup(A+B)\cup (B+B)$. By Pl\"unnecke's inequality $|A+B|=O_{k,d}(|A|)$ implies $|A+A|,|B+B|=O_{k,d}(|A|)$, so that $|S+S|\leq O_{k,d}(|S|)$. Hence, applying \Cref{SeparatedFreimanThm} with suitable parameters we can find $X,P,Q$ and $d'=O_{d,k}(1)$ as in the statement. By following the proof of \Cref{FreimanTool} we obtain the analogous statements here for $A_x + B_y$, so it remains to deduce $d'\leq d$ when $\eta$ is sufficiently small.

Letting (as ever) $A_x:=(x+P+Q)\cap A$ and $B_x:=(x+P+Q)\cap B$ for $x\in X$, find (without loss of generality) $x_0\in X$ so that $|A_{x_0}|=\max\{|A_x|,|B_x|:x\in X\}$. Then we can estimate
\begin{align*}
|A+B|&\geq \bigcup_{x\in X} A_{x_0}+B_x=\sum_{x\in X} |A_{x_0}+B_x|\\
&\geq \sum_{x\in X}\left(|A_{x_0}|^{1/(k+d')}+|B_x|^{1/(k+d')}\right)^{k+d'}-\eta (\card X) ^{-2}|A|\\
&\geq \sum_{x\in X}2^{k+d'}|B_x|-\eta (\card X) ^{-2}|A|\geq (2^{k+d'}-\eta)|A|.
\end{align*}
Hence, choosing $\eta$ small, we find $d'\leq d$.
\end{proof}
A more detailed version of this proof is included in \cite{Ruzsaconjecture}.

As this framework extends to $A\neq B$ many of the proofs in the paper translate without many conceptual difficulty though the computational details might be considerably more involved. As an illustration, we will demonstrate the corresponding result for the small $\delta$ side of \Cref{linearroughstructure}.

\begin{thm}\label{ABlinearroughstructure}
Let $A,B\subset\mathbb{R}^k$ of equal volume with $\left|\frac{A+B}{2}\right|\leq (1+\delta)|A|$, 
with $\delta\leq 0.01$. Then there exist subsets $A'\subset A$ and $B'\subset B$ with $|A\setminus A'|,|B\setminus B'|\leq \delta^2(1+O(\delta)) |A|$
and (up to translation) $|\co(A'\cup B')|\leq O_{k}(|A|)$.
\end{thm}

\begin{proof}[Proof of \Cref{ABlinearroughstructure}]
Define $A_x,B_x,(A+B)_x,A^s,B^s,(A+B)^s$ analogously to \Cref{setup}, let $|A_0|=\max\{|A_x|: x\in X\}$ and $|B_0|=\max\{|B_x|: x\in X\}$, which can be arranged by translating $A$ and $B$ separately. Write $s_A=|A_0|$ and $s_B=|B_0|$, where without loss of generality $s_A\geq s_B$.

Note that if $x \in A^s$, $y\in B^s$ then $|A_x|, |B_y| \ge s$,
so $|(A+B)_{x+y}| \ge |A_x + B_y| \ge 2^k s$ by Brunn-Minkowski. 
Thus $A^s + B^s \sub (A+B)^{2^k s}$.
Recall that
$|A|=|B| =  \int_s \card B^s$
and  $2^k(1+\delta)|A|\geq|A+B| 
 = \int_s \card (A+B)^s =  2^k \int_s \card (A+B)^{2^k s}$.
Hence, we find
\begin{align*} |A+B|&= 2^k \int_s \card (A+B)^{2^k s} \ge 2^k \int_s \card (A^s + B^s)\\
&\ge 2^k \int_{s=0}^{s_B} (\card A^s+ \card B^s - 1) =2^k|B|+2^k\int_{s=0}^{s_B} (\card A^s - 1).\end{align*}
Let $A':=A\cap (A^{s_B}+Q)$ be the subset of $A$ in those localities in which $A$ is at least as large as all localities of $B$. Then we see $|A\setminus A'|\leq \int_{s=0}^{s_B} (\card A^s - 1)\leq \delta |A|$, so that most mass of $A$ is contained in large localities.

On the other hand, we find
\begin{align*}
|A+B|&\geq \sum_{x\in X} |A_0+B_x|\geq 2^k\sum_{x\in X} \sqrt{|A_0|\cdot|B_x|}\\
&=2^k|B|+ 2^k\sum_{x\in X} |B_x|\left(\sqrt{\frac{|A_0|}{|B_x|}}-1\right),
\end{align*}
so that if we let $B':=B\cap (B^{2s_A/3}+Q)$ then 
$$|B\setminus B'|\leq \sum_{x\in X } |B_x|\frac{\sqrt{\frac{|A_0|}{|B_x|}}-1}{\sqrt{3/2}-1}\leq \frac{\delta}{\sqrt{3/2}-1} |B|\leq 5 \delta |B|.$$
In words: most of $B$ is in localities not much smaller than $A_0$.

By definition $|A'|\leq |A^{s_B}|\cdot s_A$ and $|B'|\geq |B^{2s_A/3}|\cdot 2s_A/3$, so $$\left|B^{2s_A/3}\right|\leq \frac{|B'|}{2s_A/3}\leq\frac{|B|}{2s_A/3}\leq \frac{|A'|}{(1-\delta)2s_A/3}\leq \frac{3}{2(1-\delta)} |A^{s_B}|. $$
Hence, our previous bound gives
\begin{align*}
\delta|A|&\geq \int_{s=0}^{s_B} (\card A^s - 1)
\geq s_B (|A^{s_B}|-1)\geq  s_B \left(\frac{2(1-\delta)}{3}\left|B^{2s_A/3}\right|-1\right).
\end{align*}
Assuming for a contradiction $|B^{2s_A/3}|>1$, this would imply (using $\delta$ small)
\begin{align*}
\delta|A|\geq  s_B\left|B^{2s_A/3}\right| \left(\frac{2(1-\delta)}{3}-\frac{1}{\left|B^{2s_A/3}\right|}\right)\geq |B'|\left(\frac{2(1-\delta)}{3}-\frac{1}{2}\right)\geq \frac17 |B|,
\end{align*}
which is of course absurd.
Hence, we find $B^{2s_A/3}$ is a singleton and $s_B=|B'|=|B_0|$. It immediately follows that also $A^{s_B}$ is a singleton and $s_A=|A'|=|A_0|$.

We consider again the bound
$$(1+\delta)|A|\geq \sum_{x\in X} \sqrt{|A_0|\cdot|B_x|}$$
with the conditions that $|B_x|\leq |A_0|$, $|A_0|\geq 0.9|A|$, and $\sum_x |B_x|=|B|$, to find using concavity that this sum is minimized if $|B_0|=|A_0|$ and the rest of $B$ is contained in one locality. Some computation shows that indeed we must have  $|A\setminus A'|=|A\setminus A_0|\leq \delta^2(1+O(\delta))|A|$ (see the proof of \Cref{linearandquadratic} for details). In fact, a very similar optimisation, fixing both $|A_0|$ and $|B_0|$ as constants bigger than $0.9|A|$ shows that $\sum_{x\in X} \sqrt{|A_0|\cdot|B_x|}$ is minimized if all of $B$ is contained in one locality besides $B_0$. Hence,
$$(1+\delta)|A|\geq \sqrt{|A_0|\cdot|B_0|}+\sqrt{|A_0|\cdot(|B\setminus B_0|)},$$
which using the bound on $|A_0|$ reduces to
$$1+\delta \geq \sqrt{\frac{1}{1-\delta^2(1+O(\delta))}}\left(\sqrt{\frac{|B_0|}{|B|}}+\sqrt{1-\frac{|B_0|}{|B|}}\right)$$
Hence, we find that indeed also $|B\setminus B'|=|B\setminus B_0|\leq \delta^2(1+O(\delta))|B|$.
\end{proof}

\section{Questions and directions} \label{sec:q}

We have presented several results taking the perspective
of John-type approximation for the Polynomial Freiman-Ruzsa Conjecture,
obtaining very sharp estimates for stability and locality,
at the expense of the constants in our non-degeneracy assumption
and our notion of closeness.
This suggests two further intermediate conjectures towards PFR
by removing either but not both of these two assumptions.

\begin{conj} \label{conj1}
Suppose $A \subset \mb{R}^k$ with $|A+A|< 2^{k+d}|A|$. Then
$$ \left|\co^{k,d}_{ \exp(O(k+d)) }(A)\right| \le O_{k,d}(|A)|.$$
\end{conj}

\begin{conj} \label{conj2}
Suppose $A \subset \mb{R}^k$ with $|A+A|< 2^{k+d}|A|$.
If  $\left|\co^{k,d-1}_{O_{k,d}(1)}(A)\right|>\OO_{k,d}(1)|A|$ then
$$ \left|\co^{k,d}_{ \exp(O(k+d)) }(A)\right| \le  \exp(O(k+d)) |A|.$$
\end{conj}

In \Cref{conj2} it is plausible by analogy with \Cref{moststructurethm} 
that the number of locations could be $O(2^{k+d})$ instead of $\exp(O(k+d))$ ,
i.e.\ the polynomial dependence on the doubling may in fact be linear.

There are also open problems even under the assumptions in this paper.
We conjecture the following linear stability statement that would extend
 \Cref{4kthm}, \Cref{2^0+dstability} and \Cref{2^k+dstability}.

\begin{conj}
For any $k,d\in\mathbb{N}$ and $\epsilon>0$ there is $t_{k,d,\epsilon}>0$ so that 
if $A \subset \mathbb{R}^k$ with $|A+A|=\left(2^{k+d-1}(2-\frac{k+d}{2t})+\delta'\right)|A|<\left(2^{k+d}+1-\epsilon\right)|A|$, for some $t\geq t_{k,d,\epsilon}$ then
$$\left|\co^{k,d-1}_{t}(A)\right|= O_{k,d}(|A|)
\quad  \text{ or } \quad
\left|\co^{k,d}(A)\setminus A\right|\leq  O_{k,d}\left(\delta'\right) |A|.$$
\end{conj}
In the domain where $|A+A|\leq 2^{k+d}|A|$ this would imply the following weaker conjecture
which would be already be an interesting result. 

\begin{conj}
For any $k,d\in\mathbb{N}$ there is $t_{k,d}>0$ (perhaps $t_{k,d}=\exp(O(k+d))$?)
so that  if $A \subset \mathbb{R}^k$ with $|A+A|\leq 2^{k+d}|A|$ then for all $t\geq t_{k,d}$ we have
$$\left|\co^{k,d-1}_{t}(A)\right|= O_{k,d}(|A|)
\quad  \text{ or } \quad
 \left|\co^{k,d}(A)\setminus A\right|\leq  O_{k,d}\left(t^{-1}\right) |A|.$$
\end{conj}

\subsection{Connection between \Cref{linearroughstructure} and stability of the Pr\'ekopa-Leindler inequality.}\label{prekopasec}
A recent breakthrough by B\"or\"ocky, Figalli, and Ramos \cite{boroczky2022quantitative} establishes the first stability of Pr\'ekopa-Leindler for arbitrary measurable functions, obtaining quantative control over the symmetric difference between the functions and some log-concave function. Given a parameter $\lambda\in(0,1)$, consider measurable functions $f,g,h\colon \mathbb{R}^n\to \mathbb{R}_{\geq 0}$ with the property that $h(\lambda x+(1-\lambda)y)\geq f(x)^\lambda g(y)^{1-\lambda}$ for all $x,y\in\mathbb{R}^n$ and $\int f=\int g=1$. Their result states that if $\int h\leq 1+\delta$ with $\delta$ sufficiently small in terms of $\lambda$ and $n$, then (up to translation of $f$ and $g$) there exists a log-concave function $\ell\colon \mathbb{R}^n\to \mathbb{R}_{\geq 0}$ so that $\int |h-\ell|+|f-\ell|+|g-\ell|\leq \delta^{O_{\lambda,n}(1)}$. 

Consider the special case where $n=1$, $\lambda=1/2$ and $f=g$. For any function $f:\mathbb{R}\to\mathbb{R}_{\geq0}$ with $\int f=1$, we can find a restriction $f'$ of $f$ supported on some closed interval so that $\int f'\geq 1-\delta$. If we let $h'(z):=\sup_{x+y=2z} \sqrt{f'(x)f'(y)}$, then clearly $\int h'\leq \int h\leq (1+\delta)\int f\leq (1+3\delta)\int f'$. Let $h_k(z):= \sup_{x+y=2z} \left(\frac{f'(x)^{1/k}+f'(y)^{1/k}}{2}\right)^k$ and note that $h_k$ converges pointwise to $h'$. Analogously, let $\ell$ be the log-concave function minimising $\int |h-\ell|+2|f-\ell|$, and let $\ell_k$ be a sequence of $1/k$-concave functions so that $\ell_k$ converges to $\ell$ pointwise on the support of $f'$. Now we get that $\lim_{k\to\infty}\int |h_k-h'|= 0$ and $\lim_{k\to\infty}\int_{\text{supp}(f)} |\ell_k-\ell|=0$.

The theorem by B\"or\"ocky, Figalli, and Ramos implies that for a function $f'$ with bounded support if $\lim_{k\to \infty}\int h_k\leq (1+\delta')\int f'$, then there exist $1/k$-concave functions $\ell_k$ so that $\lim_{k\to \infty}\int |h_k-\ell_k|+2|f'-\ell_k|\to \delta'^{O(1)}\int f'.$ For the converse implication consider a converging subsequence of $\ell_k$.

This equivalent form of the theorem has the following geometric interpretation. Given $k\in\mathbb{Z}$ and a measurable function $f'\colon [0,1]\to\mathbb{R}_{\geq 0}$ with $\int f'=1$, consider the set $A_{f',k}:=\bigcup_{x}\{x\}\times [0,f'(x)^{1/k}]^k$, so that $|A_{f',k}|=1$. Note that $\frac{A_{f',k}+A_{f',k}}{2}=A_{h_k,k}$. Hence the theorem is equivalent to saying that if $\lim_{k\to \infty} \left|\frac{A_{f',k}+A_{f',k}}{2}\right|\leq 1+\delta$ (with $\delta$ sufficiently small), then there are convex sets $A_{\ell_k,k}$ so that $\lim_{k\to \infty}|A_{\ell_k,k}\triangle A_{f',k}|+|A_{\ell_k,k}\triangle A_{h_k,k}|\leq \delta^{O(1)}$. This suggests that perhaps a similar implication holds in every dimension:
\begin{conj}There is an absolute constant $\Delta>0$ so that if $\delta<\Delta$ and $A\subset\mathbb{R}^k$ with $\left|\frac{A+A}{2}\right|\leq (1+\delta)|A|$ then there is some convex $K\subset\mathbb{R}^k$ with $|K\triangle A|\leq O_{\delta}(1)|A|$.
\end{conj} Here, the $O_\delta(1)$ bound suggested by the result of B\"or\"ocky, Figalli, and Ramos is $\delta^{O(1)}$, but we suspect this bound might be as strong as $O(\delta^2)$. Noting that $|K\triangle A|=|A\setminus K|+|K\setminus A|$, \Cref{linearroughstructure} (see also \Cref{ABlinearroughstructure}) solves half of this problem giving the optimal bound $|A\setminus K|\leq \delta |A|$.

\bibliographystyle{alpha}

\bibliography{references}

\end{document}